\documentclass[12pt, twoside]{article}
\usepackage{amsmath,amsthm,amssymb}
\usepackage{times}
\usepackage{enumerate}

\def\titlerunning#1{\gdef\titrun{#1}}
\makeatletter
\def\author#1{\gdef\autrun{\def\and{\unskip, }#1}\gdef\@author{#1}}
\def\address#1{{\def\and{\\\hspace*{18pt}}\renewcommand{\thefootnote}{}%
\footnote {#1}}%
\markboth{\autrun}{\titrun}} \makeatother
\def\email#1{e-mail: #1}
\def\subjclass#1{{\renewcommand{\thefootnote}{}%
\footnote{\emph{Mathematics Subject Classification (2010):} #1}}}
\def\keywords#1{\par\medskip
\noindent\textbf{Keywords.} #1}

\usepackage{amssymb,latexsym}
\usepackage{mathrsfs}
\usepackage{amssymb}
\usepackage{amsmath,latexsym}
\usepackage{amscd,latexsym} 
\usepackage[all]{xy}

\newcommand{\K}{{\mathbb K}}
\newcommand{\F}{{\mathbb F}}
\newcommand{\N}{{\mathbb N}}

\newcommand{\R}{{\mathbb R}}

\newtheorem{theorem}{Theorem}[section]

\newtheorem{corollary}[theorem]{Corollary}

\newtheorem{remark}[theorem]{Remark}

\newtheorem{lemma}[theorem]{Lemma}

\newtheorem{proposition}[theorem]{Proposition}
\newtheorem{claim}[theorem]{Claim}

\numberwithin{equation}{section}

\begin{document}
\baselineskip=17pt

\titlerunning{***}

\title{The splitting  lemmas for nonsmooth functionals
on Hilbert spaces II. The case  at infinity\footnote{IN MEMORY OF
PROFESSOR SHUZHONG SHI (1939--2008)}}

\author{Guangcun Lu}

\date{January 25, 2014}

\maketitle

\address{F1. Lu: School of Mathematical Sciences, Beijing Normal University,
Laboratory of Mathematics  and Complex Systems,  Ministry of
  Education,    Beijing 100875, The People's Republic
 of China; \email{gclu@bnu.edu.cn}}


\begin{abstract}
 We  generalize the Bartsch-Li's splitting lemma at
infinity for $C^2$-functionals in \cite{BaLi} and some later
variants of it to a class of continuously directional differentiable
functionals on Hilbert spaces.   Different from the previous flow
methods  our proof is to combine the ideas of the Morse-Palais lemma
due to Duc-Hung-Khai \cite{DHK} with some techniques from \cite{JM},
\cite{Skr}, \cite{Va1}. A simple application is also presented.

\end{abstract}


\section{Introduction and main results}\label{sec:1}

The Gromoll-Meyer's generalized Morse lemma (so called splitting
lemma) is one of key results in infinite dimensional Morse theory.
As a supplement of it, Thomas Bartsch and Shujie Li proved in 1997 a
splitting lemma at infinity (see \cite{BaLi} ) and used it to develop a
kind of Morse theory to study some variational problem without
compactness (\cite{BaLi}, \cite{HiLiWa} and \cite{Li}).   Recently,
Shaowei Chen and Shujie Li generalized it \cite{ChenLi1} (in a
Hilbert space frame) and \cite {ChenLi2} (in a Banach space frame).
These were successfully used by them in studying problems with
(strong) resonance. Their proof adopted the flow method as done for
the usual splitting lemma as in \cite{Ch}, \cite{MaWi}. So the functionals
are assumed to be at least $C^2$. Based on the proof ideas of the
Morse-Palais lemma due to Duc-Hung-Khai \cite{DHK} and some
techniques from \cite{JM}, \cite{Skr}, \cite{Va1} we find a new method to
establish  the splitting theorems for nonsmooth functionals on
Hilbert spaces in  \cite{Lu2}, \cite{Lu3}.  We shall follow the notations
therein.

Recall that a neighborhood of infinity in a Banach space $(X,
\|\cdot\|_X)$ is a set containing
 $\{u\in X\,|\, \|u\|_X>R\}$ for some $R>0$.
 A map $A$ from a neighborhood of infinity in $X$ to $X$
 is said to be {\it strictly Fr\'{e}chet differentiable  at $\infty$} if
there exists an operator $A'(\infty)\in\mathcal{ L}(X)$ such that
$$
\frac{\|A(x_1)-A(x_2)-A'(\infty)(x_1-x_2)\|_X}{\|x_1-x_2\|_X}\to 0.
$$
as $x_1\ne x_2$ and $(\|x_1\|_X, \|x_2\|_X)\to (\infty,\infty)$. (We
also say that $A$ has a strict Fr\'{e}chet derivative $A'(\infty)$.)
The  map $A$ is called {\it Fr\'{e}chet differentiable at $\infty$}
if $\|A(x)- A'(\infty)x\|_X=o(\|x\|_X)$ as $\|x\|_X\to\infty$. The
operator $A'(\infty)$ is called {\it Fr\'{e}chet derivative} of $A$
at $\infty$. ( Be careful not to confuse the two concepts! )

 Let $H$ be a Hilbert space with inner product
$(\cdot,\cdot)_H$ and the induced norm $\|\cdot\|$, and let $X$ be a
Banach space with norm $\|\cdot\|_X$, such that
\begin{enumerate}
\item[(S)]  $X\subset H$ is dense in $H$ and
 the inclusion $X\hookrightarrow H$ is continuous (and hence we
may assume $\|x\|\le \|x\|_X\;\forall x\in X$).
\end{enumerate}

In this paper for $R>0$ we  write
\begin{eqnarray*}
&&B_X(\infty, R):=\{x\in X\,|\,\|x\|_X> R\},\quad \bar B_X(\infty,
R):=\{x\in X\,|\,\|x\|_X\ge R\},\\
&&B_H(\infty, R):=\{x\in H\,|\,\|x\|> R\},\quad \bar B_H(\infty,
R):=\{x\in H\,|\,\|x\|\ge R\}.
\end{eqnarray*}

 Let $V_\infty$ be an open neighborhood of infinity in $H$.
 Then $V_\infty\cap X$ is  open
 in $X$, and also star-shaped with respect to infinity provided
 $V_\infty$ star-shaped with respect to infinity. For clearness we shall write $V_\infty\cap X$ as
 $V_\infty^X$ when it is equipped with the induced topology from
 $X$.

 Suppose that  a functional
$\mathcal{ L}:V_\infty\to\mathbb{R}$ satisfies the following
conditions:
\begin{enumerate}
\item[$({\rm F1_\infty})$] $\mathcal{ L}$ is continuous and continuously directional
differentiable on $V_\infty$.

\item[(${\rm F2_\infty}$)] There exists a continuous and continuously directional
differentiable map $A: V^X_\infty\to X$,  such that
$$
D\mathcal{ L}(x)(u)=(A(x), u)_H\quad\hbox{for all}\; x\in V_\infty\cap
X,\quad\hbox{for all}\; u\in X.
$$
(This actually implies that $\mathcal{ L}|_{V^X_\infty}\in
C^1(V^X_\infty, \R)$.)
\item[(${\rm F3_\infty}$)] There exists a map $B$ from
$(V_\infty\cap X)\cup\{\infty\}$ to the space $L_s(H)$ of bounded
self-adjoint linear operators of $H$ such that $B(\infty)(X)\subset
X$ and
$$
(DA(x)(u), v)_H=(B(x)u, v)_H\quad\hbox{for all}\; x\in V_\infty\cap
X,\quad\hbox{for all}\; u, v\in X.
$$
(This implies: $DA(x)=B(x)|_X$ for all $x\in V_\infty\cap X$, and
thus  $B(x)(X)\subset X\quad\hbox{for all}\; x\in (V_\infty\cap
X)\cup\{\infty\}$.)

\item[(${\rm C1_\infty}$)] Either $0\notin \sigma(B(\infty))$ or $0$
is  an isolated point of the spectrum $\sigma(B(\infty))$.
 \footnote{The claim is actually implied in the following condition (${\rm D}_\infty$)
by Proposition~B.2 in \cite{Lu2} and \cite{Lu3}. In order to state some
results without the condition (${\rm D}_\infty$) we still list it
here.}

\item[(${\rm C2_\infty}$)]  If $u\in H$ such that $B(\infty)(u)=v$ for some
$v\in X$, then $u\in X$.

\item[(${\rm D_\infty}$)] The map $B:(V_\infty\cap X)\cup\{\infty\}\to
L_s(H)$  has a decomposition
$$
B(x)=P(x)+ Q(x)\quad\hbox{for all}\; x\in (V_\infty\cap X)\cup\{\infty\},
$$
where $P(x):H\to H$ is a positive definite linear operator and
$Q(x):H\to H$ is a compact linear  operator with the following
properties:
\begin{enumerate}
\item[(${\rm D1_\infty}$)]  All eigenfunctions of the operator $B(\infty)$ that correspond
to negative eigenvalues belong to $X$;

\item[(${\rm D2_\infty}$)] For any sequence $\{x_k\}\subset
V_\infty\cap X$ with $\|x_k\|\to \infty$ it holds that
$\|P(x_k)u-P(\infty)u\|\to 0$ for any $u\in H$;

\item[(${\rm D3_\infty}$)] The  map $Q:(V_\infty\cap X)\cup\{\infty\}\to
L(H)$ is continuous at $\infty$ with respect to the topology induced
from $H$ on $V_\infty\cap X$, i.e. $\|Q(x)-Q(\infty)\|_{L(H)}\to 0$
as $x\in V_\infty\cap X$ and $\|x\|\to\infty$;

\item[(${\rm D4_\infty}$)] For any sequence $\{x_n\}\subset V_\infty\cap X$ with
$\|x_n\|\to \infty$ (as $n\to\infty$), there exist
 constants $C_0>0$ and $n_0>0$ such that
$$
(P(x_n)u, u)_H\ge C_0\|u\|^2\quad\hbox{for all}\; u\in H,\;\hbox{for all}\; n\ge n_0.
$$
\end{enumerate}
\end{enumerate}

As before let $H^0_\infty={\rm Ker}(B(\infty))$, which is contained
in $X$ by (${\rm C2_\infty}$). Then
$H_\infty^\pm:=(H^0_\infty)^\bot$ is equal to the range of
$B(\infty)$ by (${\rm C1_\infty}$). (See Proposition~B.2 in
\cite{Lu2} and \cite{Lu3}). Obverse that  $H_\infty^\pm$ splits as
$H^\pm_\infty=H^+_\infty\oplus H^-_\infty$, where $H^+_\infty$
(resp. $H^-_\infty$) is positive (resp. negative) definite subspace
of $B(\infty)$, that is, there exists some $a_\infty>0$ such that
\begin{equation}\label{e:1.1}
\left.\begin{array}{ll}
 &(B(\infty)u^+, u^+)_H\ge 2a_\infty\|u^+\|^2\quad\quad\hbox{for all}\; u\in
H^+_\infty,\\
&(B(\infty)u^-, u^-)_H\le -2a_\infty\|u^-\|^2\quad\;\hbox{for all}\; u\in
H^-_\infty.
\end{array}
\right\}
\end{equation}
 Write $X^\pm_\infty:=H^\pm_\infty\cap X$ and
$X^\ast_\infty:=H^\ast_\infty\cap X$, $\ast=+, -$. We get
topological direct sum decompositions $X=H^0_\infty\oplus
X^\pm_\infty$ and $X^\pm_\infty=X^+_\infty\oplus X^-_\infty$. In
addition, $H^0_\infty$ and $X^-_\infty$ have finite dimensions by
(${\rm D_\infty}$). ({\it Note}: As in the proof of
\cite[Lemma~2.13]{Lu2} or \cite[Lemma~3.1]{Lu3} the condition
$H^0_\infty\subset X$ is enough for the following
Lemmas~\ref{lem:2.2} and \ref{lem:2.3} because this implies that
$H^0_\infty\subset X$ is complete in both $H$ and $X$ and therefore
that $H$ and $X$  induce equivalent norms on $H^0_\infty$ in the
case). Let
$$
\nu_\infty:=\dim H^0_\infty\quad\hbox{and}\quad \mu_\infty:=\dim
H^-_\infty.
$$
They are called the {\it nullity and Morse index} of $\mathcal{ L}$
at infinity, respectively. Denote by $P^\ast_\infty$ the orthogonal
projections from $H$ onto $H^\ast_\infty$, $\ast=+, 0, -$.

As in the proof of \cite[Lemma~2.13]{Lu2} or \cite[Lemma~3.1]{Lu3}
we get that
$$B(\infty)|_{X^\pm_\infty}:
X^\pm_\infty\to X^\pm_\infty
$$
is a Banach space isomorphism.  Let
$$
C^\infty_1=\|(B(\infty)|_{X^\pm_\infty})^{-1}\|_{L(X^\pm_\infty,
X^\pm_\infty)}\quad\hbox{and}\quad C^\infty_2=\|I-P^0_\infty\|_{L(X,
X^\pm_\infty)}.
$$

We shall give our results in cases $\nu_\infty>0$ and
$\nu_\infty=0$, respectively. For the former case  we
further assume the following condition to be satisfied.
\begin{enumerate}
\item[(${\rm E_\infty}$)]  $M(A):=\lim_{R\to\infty}\sup\{\|(I-P^0_\infty)A(z)\|_X\,:
\;z\in H^0_\infty, \|z\|_X\ge R\}<\infty$, and there exist $R_1>0$,
$\kappa>1$ and
 $\rho_A\in (\frac{\kappa}{\kappa-1}C_1^\infty M(A),
\infty)$,  such that
\begin{eqnarray}
&&\|(I-P^0_\infty)A(z_1+x_1)- B(\infty)x_1
-(I-P^0_\infty)A(z_2+x_2)+
B(\infty)x_2\|_{X^\pm_\infty}\nonumber\\
&&\le \frac{1}{\kappa C_1^\infty}\|z_1+ x_1-z_2-x_2\|_X\label{e:1.2}
\end{eqnarray}
for all $x_i\in B_X(\theta, \rho_A)\cap X^\pm_\infty$ and $z_i\in
H^0_\infty$ with $\|z_i\|\ge R_1$, $i=1,2$.
Moreover, if (\ref{e:1.2}) holds with $\rho_A=\infty$ the assumption that $M(A)<\infty$ is not needed.
(Obverse that
(\ref{e:1.2}) is satisfied if
\begin{eqnarray}
&&\|A(z_1+x_1)
-A(z_2+x_2)- B(\infty)(x_1-x_2)
\|_{X}\nonumber\\
&&\le \frac{1}{\kappa C_1^\infty C_2^\infty}\|z_1+
x_1-z_2-x_2\|_X\label{e:1.3}
\end{eqnarray}
for all $x_i\in B_X(\theta, \rho_A)\cap X^\pm_\infty$ and $z_i\in
H^0_\infty$ with $\|z_i\|\ge R_1$, $i=1,2$.) \footnote{If $R_1>0$ is
large enough then $z+x\in V_\infty\cap X$ for any $z\in
B_{H^0_\infty}(\infty, R_1)$ and any $x\in X^\pm_\infty$.}
\end{enumerate}

Clearly, (${\rm E_\infty}$) is satisfied if the following assumption holds.

\begin{enumerate}
\item[(${\rm SE_\infty}$)]  $M(A):=\lim_{R\to\infty}\sup\{\|(I-P^0_\infty)A(z)\|_X\,:
\;z\in H^0_\infty, \|z\|_X\ge R\}<\infty$, and there exists
$\rho_A\in (C_1^\infty M(A), \infty)$ such that
$$
\frac{\|(I-P^0_\infty)A(z_1+x_1)- B(\infty)x_1
-(I-P^0_\infty)A(z_2+x_2)+ B(\infty)x_2\|_{X^\pm_\infty}}{\|z_1+
x_1-z_2-x_2\|_X}\to 0
$$
{uniformly} in $x_1, x_2\in B_X(\theta, \rho_A)\cap X^\pm_\infty$
 as $(z_1, z_2)\in H^0_\infty\times
H^0_\infty$ and $(\|z_1\|, \|z_2\|)\\
\to(\infty, \infty)$. (Note: $\rho_A>\frac{\kappa}{\kappa-1}C_1^\infty M(A)$
if $\kappa>1$ is large enough.)
Moreover, if this holds with $\rho_A=\infty$ the assumption that $M(A)<\infty$ is not needed.
\end{enumerate}

  Note: Since the norms $\|\cdot\|$ and
$\|\cdot\|_X$ are equivalent on $H^0_\infty$ and we have assumed
$\|u\|\le\|u\|_X\;\forall u\in X$, which implies
$\|z+x\|_X^2\ge\|z+x\|^2=\|z\|^2+\|x\|^2\ge\|z\|^2$ for any
$(z,x)\in H^0_\infty\times X^\pm_\infty$, if
$B(\infty)|_{X}\in L(X)$ and $A$ has the strict Fr\'{e}chet
derivative $B(\infty)|_{X}$ at $\infty$, it is easily proved that
(${\rm SE_\infty}$) holds for any $\rho_A\in (0,\infty]$.

The following assumption is slightly weaker than (${\rm E_\infty}$).

\begin{enumerate}
\item[(${\rm E'_\infty}$)]
$M(A):=\lim_{R\to\infty}\sup\{\|(I-P^0_\infty)A(z)\|_X:\;z\in
H^0_\infty, \|z\|_X\ge R\}<\infty$, and there exist $R_1>0$,
$\kappa>1$ and $\rho_A\in (\frac{\kappa}{\kappa-1}C_1^\infty M(A),
\infty)$  such that
\begin{eqnarray}
&&\|(I-P^0_\infty)A(z+x_1)- B(\infty)x_1
-(I-P^0_\infty)A(z+x_2)+
B(\infty)x_2\|_{X^\pm_\infty}\nonumber\\
&&\le \frac{1}{\kappa C_1^\infty}\| x_1-x_2\|_X\label{e:1.4}
\end{eqnarray}
 holds for all $x_i\in B_X(\theta,
\rho_A)\cap X^\pm_\infty$ and $z\in H^0_\infty$ with $\|z\|\ge R_1$.
Moreover, if (\ref{e:1.4}) holds with $\rho_A=\infty$ the assumption
that $M(A)<\infty$ is not needed.
 (Clearly, (\ref{e:1.4}) is satisfied if (\ref{e:1.3}) is
satisfied for all $x_i\in B_X(\theta, \rho_A)\cap X^\pm_\infty$ and
$z_1=z_2\in H^0_\infty$ with $\|z_i\|\ge R_1$.)
\end{enumerate}

As above  (${\rm E'_\infty}$) is satisfied under the following assumption.

\begin{enumerate}
\item[(${\rm SE'_\infty}$)]
$M(A):=\lim_{R\to\infty}\sup\{\|(I-P^0_\infty)A(z)\|_X:\;z\in
H^0_\infty, \|z\|_X\ge R\}<\infty$, and there exists $\rho_A\in
(C_1^\infty M(A), \infty)$  such that
$$
\frac{\|(I-P^0_\infty)A(z+ x_1)- B(\infty)x_1 -(I-P^0_\infty)A(z+
x_2)+ B(\infty)x_2\|_{X^\pm_\infty}}{\|x_1- x_2\|_X}\to 0
$$
uniformly in $x_1, x_2\in B_X(\theta, \rho_A)\cap X^\pm_\infty$ as
$z\in H^0_\infty$ and $\|z\|\to\infty$. (Note:
 $\rho_A>\frac{\kappa}{\kappa-1}C_1^\infty M(A)$ if $\kappa>1$ is large enough.)
 Moreover, if this holds with $\rho_A=\infty$
the assumption that $M(A)<\infty$ is not needed.
\end{enumerate}

{\bf [}  Note: If  $B(\infty)|_{X}\in L(X)$ and $A$ has the strict
Fr\'{e}chet derivative $B(\infty)|_{X}$ at $\infty$, then
 (${\rm SE'_\infty}$) holds for any $\rho_A\in
(0,\infty]$. {\bf ]}

We have the following splitting lemmas at infinity on Hilbert
spaces.

\begin{theorem}\label{th:1.1}
Under the above assumptions $({\rm S})$, $({\rm F1_\infty})$--$({\rm
F3_\infty})$ and $({\rm C1_\infty})$--$({\rm C2_\infty})$, $({\rm
D_\infty})$, also suppose that $\nu_\infty>0$ and $({\rm
E}'_\infty)$ is satisfied and  that
\begin{equation}\label{e:1.5}
\mathcal{ L}(u)=\frac{1}{2}(B(\infty)u,u)_H+
o(\|u\|^2)\quad\hbox{as}\quad \|u\|\to\infty. \footnote{This
condition is weaker than the assumption (${\rm A}_\infty$) in
\cite{BaLi}. See \S 3.1 below.}
\end{equation}
Then there exist a positive number $R$, a (unique) continuous map $
h^\infty: B_{H^0_\infty}(\infty, R)\to X^\pm_\infty$ (which takes
values in $\bar B_{X^\pm_\infty}(\theta, \rho_A)$ in the case
$M(A)<\infty$) satisfying
\begin{equation}\label{e:1.6}
 (I-P^0_\infty)A(z+ h^\infty(z))=0\quad\hbox{for all}\; z\in \bar B_{H^0_\infty}(\infty, R),
\end{equation}
and a homeomorphism $\Phi:  B_{H^0_\infty}(\infty, R)\oplus
H^\pm_\infty\to B_{H^0_\infty}(\infty, R)\oplus H^\pm_\infty$ of
form
\begin{equation}\label{e:1.7}
\Phi(z+ u^++ u^-)=z+ h^\infty(z)+\phi_z(u^++ u^-)
\end{equation}
with $\phi_z(u^++ u^-)\in H^\pm_\infty$ and
$\Phi(B_{H^0_\infty}(\infty, R)\oplus X^\pm_\infty)\subset X$, such
that
\begin{equation}\label{e:1.8}
\mathcal{ L}\circ\Phi(z+ u^++ u^-)=\|u^+\|^2-\|u^-\|^2+ \mathcal{
L}(z+ h^\infty(z))
\end{equation}
for all $(z, u^+ + u^-)\in B_{H^0_\infty}(\infty, R)\times
H^\pm_\infty$.  The homeomorphism $\Phi$ has also properties:
\begin{enumerate}
\item[{\rm (a)}] For each $z\in B_{H^0_\infty}(\infty, R)$, $\Phi(z,
\theta)=z+ h^\infty(z)$, and $\phi_z(u^++ u^-)\in H^-_\infty$ if and
only if $u^+=\theta$;

\item[{\rm (b)}] The restriction of $\Phi$ to $B_{H^0_\infty}(\infty,
R)\oplus H^-_\infty$ is a homeomorphism from $B_{H^0_\infty}(\infty,
R)\oplus H^-_\infty\subset X$ onto $\Phi(B_{H^0_\infty}(\infty,
R)\oplus H^-_\infty))\subset X$ even if the topologies on these two
sets  are chosen as the induced one by $X$.
\end{enumerate}
The map $h^\infty$ and the function
$\mathcal{ L}^\infty: B_{H^0_\infty}(\infty, R)\to\R, \; z\mapsto
\mathcal{ L}(z+ h^\infty(z))$
also satisfy:
\begin{enumerate}
\item[{\rm (i)}] $\lim_{\|z\|_X\to\infty}\|h^\infty(z)\|_X=0$ provided that
 $$
 \lim_{R\to\infty}\sup\{\|(I-P^0_\infty)A(z)\|_X:\;z\in H^0_\infty,
\|z\|_X\ge R\}=0;
$$
\item[{\rm (ii)}] If $A$ is $C^1$, then $h^\infty$ is $C^1$ and
$$
\hspace{11mm} dh^\infty(z)=-\bigl[(I-P^0_\infty)A'(z+
h^\infty(z))|_{X^\pm_\infty}\bigr]^{-1}(I-P^0_\infty)A'(z+
h^\infty(z))|_{H^0_\infty},
$$
moreover the function $\mathcal{ L}^\infty$ is  $C^{2}$ and
\begin{equation}\label{e:1.9}
\left.\begin{array}{ll} & d\mathcal{ L}^\infty(z_0)(z)=(A(z_0+
h^\infty(z_0)), z)_H,\\
&\quad \hbox{for all}\; z_0\in B_{H^0_\infty}(\infty, R),\; z\in H^0_\infty.
\end{array}\right\}
 \end{equation}

\item[{\rm (iii)}] If $\mathcal{ L}$
is  $C^2$ then $h^\infty$ is also $C^1$ as a map to $H^\pm_\infty$
(hence $X^\pm_\infty$).
\end{enumerate}
If  $({\rm E}'_\infty)$ is replaced by the slightly strong $({\rm
E}_\infty)$ {\rm (}and $\rho_A$ is given by (${\rm E}_\infty$){\rm
)} one has:
\begin{enumerate}
\item[{\rm (iv)}]  The map $h^\infty$ is Lipschitz, and has a strict Fr\'echet derivative zero at
$\infty$;

\item[{\rm (v)}]  $\mathcal{ L}^\infty$ is $C^{1}$ and (\ref{e:1.9}) holds;

\item[{\rm (vi)}] If $B(\infty)\in L(X)$ and $A$ has a strict Fr\'{e}chet derivative $B(\infty)|_X$
 at $\infty$, then $\mathcal{ L}^\infty$ is $C^{2-0}$
 and $d\mathcal{ L}^\infty$ has the strict  Fr\'{e}chet derivative zero at $\infty$.
  ({\rm In this case, as noted below (${\rm SE_\infty}$) we
 may choose $\rho_A$ above to be any positive number, but $R$ depends on
 this choice.})
\end{enumerate}
\end{theorem}

\begin{remark}\label{rm:1.2}
{\rm Similar conclusions to Remarks~2.2,2.3 in \cite{Lu2}, \cite{Lu3} also
hold. Namely, we only use Lemmas~\ref{lem:2.4} and \ref{lem:2.5} in
the proof of Lemma~\ref{lem:2.6}. Hence the condition (${\rm
D}_\infty$) can be replaced by the following
\begin{enumerate}
\item[(${\rm D}'_\infty$)] There exist a  subset $U_\infty\subset V_\infty$ of form
$U_\infty=\bar B_{H^0_\infty}(\infty, R')\oplus H^\pm_\infty$,
 a positive number $c_\infty$ and a function $\omega_\infty:
U_\infty\cap X\to [0, \infty)$ with the property that
$\omega_\infty(x)\to 0$ as $x\in U_\infty\cap X$ and $\|x\|\to
\infty$, such that
\begin{enumerate}
\item[(${\rm D}'_{\infty1}$)] the kernel $H^0_\infty$ and negative definite
subspace $H^-_\infty$ of $B(\infty)$ are finite dimensional
subspaces contained in $X$;
\item[(${\rm D}'_{\infty2}$)] $(B(x)v, v)_H\ge c_\infty\|v\|^2\quad\hbox{for all}\; v\in H^+_\infty, x\in U_\infty\cap X$;
\item[(${\rm D}'_{\infty3}$)] $|(B(x)u,v)_H-(B(\infty)u,v)_H|\le\omega_\infty(x)\|u\|\cdot\|v\|\quad\hbox{for all}\; u\in H,
 v\in H^-_\infty\oplus H^0_\infty$;
\item[(${\rm D}'_{\infty4}$)] $(B(x)u,u)_H\le-c_\infty\|u\|^2\quad\hbox{for all}\; u\in H^-_\infty, x\in U_\infty\cap X$.
\end{enumerate}
\end{enumerate}}
\end{remark}

In order to state our second result, for positive numbers $R$ and
$\delta$ we set
$$
C_{R, \delta}:= B_{H^0_\infty}(\infty, R)\oplus
B_{H^+_\infty}(\theta, \delta)\oplus B_{H^-_\infty}(\theta, \delta).
$$
(It is often identified with $B_{H^0_\infty}(\infty, R)\times
B_{H^+_\infty}(\theta, \delta)\times B_{H^-_\infty}(\theta,
\delta)$).

\begin{theorem}\label{th:1.3}
Under the above assumptions $({\rm S})$, $({\rm F1_\infty})$--$({\rm
F3_\infty})$ and $({\rm C1_\infty})$--$({\rm C2_\infty})$, $({\rm
D_\infty})$, also suppose that $\nu_\infty>0$  and $({\rm
E}'_\infty)$ is satisfied. Then  for any $r\in (0, \infty)$ there
exist positive numbers $R$, $\delta_r>0$ and  a (unique) continuous
map $h^\infty:B_{H^0_\infty}(\infty, R)\to X^\pm_\infty$ (which
takes values in $\bar B_{X^\pm_\infty}(\theta, \rho_A)$ in the case
$M(A)<\infty$ ) satisfying
\begin{equation}\label{e:1.10}
 (I-P^0_\infty)A(z+ h^\infty(z))=0\quad\forall z\in B_{H^0_\infty}(\infty, R),
 \end{equation}
an open set $V(R, r)$ in $H$ with $V(R, r)\subset \overline{C_{R,
r+\rho_A}}$, and a homeomorphism $\Phi: C_{R, \delta_r}\to V(R,r)$
of form
$$
\Phi(z+ u^++ u^-)=z+ h^\infty(z)+\phi_z(u^++ u^-)
$$
with $\phi_z(u^++ u^-)\in H^\pm_\infty$ and  $\Phi(C_{R,
\delta_r}\cap X)\subset X$,  such that
$$
\mathcal{ L}\circ\Phi(z, u^++ u^-)=\|u^+\|^2-\|u^-\|^2+ \mathcal{
L}(z+ h^\infty(z))
$$
for all $(z, u^+, u^-)\equiv z+ u^+ + u^-\in C_{R, \delta_r}$. The
homeomorphism $\Phi$  also possesses properties:
\begin{enumerate}
\item[{\rm (a)}] For each $z\in B_{H^0_\infty}(\infty, R)$,
$\Phi(z, \theta)=z+ h^\infty(z)$, $\phi_z(u^++ u^-)\in H^-_\infty$
if and only if $u^+=\theta$;

\item[{\rm (b)}] The restriction of $\Phi$ to $B_{H^0_\infty}(\infty,
R)\oplus B_{H^-_\infty}(\theta, \delta_r)$ is a homeomorphism
from\linebreak
 $B_{H^0_\infty}(\infty, R)\oplus
B_{H^-_\infty}(\theta, \delta_r)\subset X$ onto
$\Phi(B_{H^0_\infty}(\infty, R)\oplus B_{H^-_\infty}(\theta,
\delta_r))\subset X$ even if the topologies on these two sets
 are chosen as the induced one
by $X$.
\end{enumerate}

 The map $h^\infty$ and the function
$\mathcal{ L}^\infty: B_{H^0_\infty}(\infty, R)\to\R,\;  z\mapsto
\mathcal{ L}(z+ h^\infty(z))$
satisfy the conclusions {\rm (i)--(iii)} in Theorem~\ref{th:1.1}, and
also {\rm (iv)--(vi)} in Theorem~\ref{th:1.1} if $({\rm E_\infty})$
holds and $\rho_A$ is given by (${\rm E}_\infty$).
\end{theorem}

In Theorems~\ref{th:1.1},~\ref{th:1.3}, if $\mathcal{ L}$ is $C^2$
and $D^2\mathcal{ L}(w)=B(\infty)+ o(1)$ as $\|w\|\to\infty$,   we
shall prove in Remark~\ref{rm:2.15} that $\Phi^{-1}$  is $C^1$
outside the submanifold of codimension $\mu_\infty$ if $R>0$ is
large enough.

\begin{remark}\label{rm:1.4}
{\rm Similar conclusions to Remarks~2.2,~2.3 in \cite{Lu2},\cite{Lu3} also
hold. By the note below Lemma~\ref{lem:2.5}, we can still get
Theorem~\ref{th:1.3} if we replace the condition (${\rm D}_\infty$)
by the following
\begin{enumerate}
\item[(${\rm D}''_\infty$)] There exist a subset of $X$ of form
$$
W_\infty=\bar B_{H^0_\infty}(\infty, R')\oplus (\bar B_H(\theta,
r')\cap X^\pm_\infty)
 \subset V_\infty\cap X,
 $$
  a positive number $c_\infty$ and a function $\omega_\infty:
W_\infty\to [0, \infty)$ with the property that $\omega_\infty(x)\to
0$ as $x\in W_\infty$ and $\|x\|\to \infty$, such that
\begin{enumerate}
\item[(${\rm D}''_{\infty1}$)] the kernel $H^0_\infty$ and negative definite
subspace $H^-_\infty$ of $B(\infty)$ are finite dimensional
subspaces contained in $X$;
\item[(${\rm D}''_{\infty2}$)] $(B(x)v, v)_H\ge c_\infty\|v\|^2\;\hbox{for all}\; v\in H^+_\infty, x\in W_\infty$;
\item[(${\rm D}''_{\infty3}$)] $|(B(x)u,v)_H-(B(\infty)u,v)_H|\le\omega_\infty(x)\|u\|\cdot\|v\|\;\hbox{for all}\; u\in H,
 v\in H^-_\infty\oplus H^0_\infty$;
\item[(${\rm D}''_{\infty4}$)] $(B(x)u,u)_H\le-c_\infty\|u\|^2\;\hbox{for all}\; u\in H^-_\infty,  x\in W_\infty$.
\end{enumerate}
\end{enumerate}}
\end{remark}

\begin{corollary}\label{cor:1.5}
Suppose that one of the following condition groups holds:
\begin{enumerate}
\item [{\rm (a)}]  $({\rm S})$,
$({\rm F1_\infty})$--$({\rm F3_\infty})$ and $({\rm
C1_\infty})$--$({\rm C2_\infty})$, $({\rm D_\infty})$ and $({\rm
E_\infty})$;

\item[{\rm (b)}] $({\rm S})$,
$({\rm F1_\infty})$--$({\rm F3_\infty})$ and $({\rm
C1_\infty})$--$({\rm C2_\infty})$, $({\rm D_\infty})$ and $({\rm
E'_\infty})$, and $A$ being $C^1$.
\end{enumerate}
Then each critical point $z$ of the function $\mathcal{
L}^\infty:B_{H^0_\infty}(\infty, R)\to \R$ gives a critical point of
$\mathcal{ L}$, $z+ h^\infty(z)$.
\end{corollary}

\begin{proof}
Under the condition group (a) or (b),
$\mathcal{ L}^\infty$ is at least $C^1$. For a critical point $z$ of
it (\ref{e:1.9}) shows that $(A(z+ h^\infty(z)), z')_H=0$ for all
$z'\in H^0_\infty$, i.e.,
$$
(P^0_\infty A(z+ h^\infty(z)), u)_H=0\quad\hbox{for all}\; u\in H.
$$
This and (\ref{e:1.10}) imply $A(z+ h^\infty(z))=\theta$. Since $X$
is dense in $H$,  the desired claim follows from the condition
(${\rm F2}_\infty$).
\end{proof}


When $X=H$ Theorems~\ref{th:1.1},~\ref{th:1.3} have the following
corollaries, respectively.

\begin{corollary}\label{cor:1.6}
Let $V_\infty$ be a neighborhood of infinity in a Hilbert space $H$,
and let ${\cal L}:V_\infty\to\R$ be a $C^1$-functional. Suppose that
$\nabla{\cal L}:V_\infty\to H$ is  continuously directional
differentiable and that there exists a map $B$ from
$V_\infty\cup\{\infty\}$ to the space $L_s(H)$ of bounded
self-adjoint linear operators of $H$ such that
$$
(D\nabla{\cal L}(x)(u), v)_H=(B(x)u, v)_H\quad\hbox{for all}\; x\in
V_\infty,\quad\hbox{for all}\; u, v\in H.
$$
(So ${\cal L}$ has the G\^ateaux derivative  of second order ${\cal
L}''(x)=B(x)$ at $x\in V_\infty$.) Write ${\cal L}$ as
$$
{\cal L}(x)=\frac{1}{2}(B(\infty)x,x)_H+ g(x).
$$
($g$ has the G\^ateaux derivative  of second order
$g''(x)=B(x)-B(\infty)$ at $x\in V_\infty$.) Suppose
\begin{enumerate}
\item[{\rm (a)}] $g(x)=o(\|x\|^2)$ as $\|x\|\to\infty$;
\item[{\rm (b)}] $0\in\sigma(B(\infty))$ and $B(\infty)=P(\infty)+ Q(\infty)$, where
$P(\infty)\in{\cal L}_s(H)$ is positive definite and
$Q(\infty)\in{\cal L}_s(H)$ is compact;
\item[{\rm (c)}] For any sequence $\{x_n\}\subset V_\infty$ with $\|x_n\|\to
\infty$ (as $n\to\infty$), there exist
 constants $C_0>0$ and $n_0>0$ such that
$$
([B(x_n)-Q(\infty)]u, u)_H\ge C_0\|u\|^2\quad\hbox{for all}\; u\in
H,\;\hbox{for all}\; n\ge n_0.
$$
\item[{\rm (d)}]  $H^0_\infty:={\rm Ker}(B(\infty))\ne\{\theta\}$ and $H^\pm_\infty:=(H^0_\infty)^\bot$,
$C^\infty_1=\|(B(\infty)|_{H^\pm_\infty})^{-1}\|_{L(H^\pm_\infty)}$,
if $M(A):=\lim_{R\to\infty}\sup\{\|(I-P^0_\infty)A(z)\|:\;z\in
H^0_\infty, \|z\|\ge R\}<\infty$ with $A=\nabla{\cal L}$,  there
exist constants $R_1>0, \kappa>1$ and $\rho_A\in
(\frac{\kappa}{\kappa-1}C_1^\infty M(A), \infty)$ such that for all
$y\in B_{H^\pm_\infty}(\theta, \rho_A)$, $z\in \bar
B_{H^0_\infty}(\theta, R_1)$,
\begin{eqnarray*}
\|(I-P^0_\infty)[B(z+y)-B(\infty)]|_{H^\pm_\infty}\|_{L(H^\pm_\infty)}
\le \frac{1}{\kappa C_1^\infty }.
\end{eqnarray*}
Moreover, that $M(A)<\infty$ is not needed if  there exists a
constant $R_1>0, \kappa>1$ such that for all $y\in
H^\pm_\infty,\,z\in \bar B_{H^0_\infty}(\theta, R_1)$,
\begin{eqnarray*}
\|(I-P^0_\infty)[B(z+y)-B(\infty)]|_{H^\pm_\infty}\|_{L(H^\pm_\infty)}\le\frac{1}{\kappa
C_1^\infty }.
\end{eqnarray*}
\end{enumerate}
Then there exist a positive number $R\ge R_1$, a (unique) continuous
map $ h^\infty: B_{H^0_\infty}(\infty, R)\to H^\pm_\infty$
satisfying (\ref{e:1.6}) with $A=\nabla{\cal L}$, which takes values
in $\bar B_{H^\pm_\infty}(\theta,\rho_A)$ in the case $M(A)<\infty$,
and a homeomorphism $\Phi: B_{H^0_\infty}(\infty, R)\oplus
H^\pm_\infty\to B_{H^0_\infty}(\infty, R)\oplus H^\pm_\infty$  such
that
$$
\mathcal{ L}\circ\Phi(z+ u^++ u^-)=\|u^+\|^2-\|u^-\|^2+ \mathcal{
L}(z+ h^\infty(z))
$$
for all $(z, u^+ + u^-)\in B_{H^0_\infty}(\infty, R)\times
H^\pm_\infty$. Moreover, if ${\cal L}$ is $C^2$ then the map
$h^\infty$ is $C^1$ and the function $B_{H^0_\infty}(\infty,
R)\to\R, \; z\mapsto \mathcal{ L}^\infty(z):=\mathcal{ L}(z+
h^\infty(z))$ is $C^2$.
\end{corollary}

\begin{proof}
By Propositions~B.2,B.3 in \cite{Lu2} or \cite{Lu3}, $0$ is an
isolated spectrum point of $B(\infty))$, and  $B(\infty)$ has the
finite dimensional kernel $H^0_\infty$ and negative definite
subspace $H^-_\infty$. For $x\in V_\infty$ let $P(x)=P(\infty)+
g''(x)=B(x)-Q(\infty)$ and $Q(x)\equiv Q(\infty)$. Then $B(x)=P(x)+
Q(x)$. The condition (iii) implies that (${\rm D4_\infty}$) is
satisfied. It follows that $P(x)$ is positive definite for each $x$
in a neighborhood of infinity in $H$. Hence (${\rm D_\infty}$) is
satisfied.

Next we shows that the condition (iv) implies (${\rm E'_\infty}$).
Since $g'(x)=A(x)-B(\infty)x$ with  $A=\nabla{\cal L}$, and
$g''(x)=B(x)-B(\infty)$ using the mean value theorem in inequality
form we deduce that
\begin{eqnarray*}
&&\frac{\|(I-P^0_\infty)A(z+ x_1)-B(\infty)x_1-(I-P^0_\infty)A(z+
x_2)+ B(\infty)x_2\|}{\|x_1-x_2\|}\\
&=&\frac{\|(I-P^0_\infty)g'(z+x_1)-(I-P^0_\infty)g'(z+x_2)\|}{\|x_1-x_2\|}\\
&\le& \sup_{t\in [0, 1]}\|(I-P^0_\infty)g''(z+ tx_1+
(1-t)x_2)|_{H^\pm_\infty}\|_{L(H^\pm_\infty)}\le\frac{1}{\kappa
C_1^\infty}
\end{eqnarray*}
for all $z\in \bar B_{H^0_\infty}(\infty, R_1)$ and $x_i\in
B_{H^\pm_\infty}(\theta,\rho_A)$, $i=1, 2$ and $x_1\ne x_2$.
Moreover, since $I-P^0_\infty\ne 0$,
$C^\infty_2=\|I-P^0_\infty\|_{L(H, H^\pm_\infty)}=1$. So the
condition (${\rm E'_\infty}$) holds. Corollary~\ref{cor:1.6}
immediately follows from Theorem~\ref{th:1.1}.
\end{proof}

In Corollary~\ref{cor:1.6}, if ${\cal L}$ is $C^2$ and $g''(x)=o(1)$
as $\|x\|\to\infty$ then the conditions (c)-(d) are satisfied
automatically. This almost leads to
 the splitting lemmas at infinity
first established by Thomas Bartsch and Shujie Li
\cite[p. 431]{BaLi}. See Section~\ref{sec:3.1} below for a detailed explanation. As
in the proof of Corollary~\ref{cor:1.6} Theorem~\ref{th:1.3} leads
to

\begin{corollary}\label{cor:1.7}
Under the assumptions (b)--(d) of Corollary~\ref{cor:1.6}, for any
$r\in (0, \infty)$ there exist positive numbers $R\ge R_1$,
$\delta_r>0$ and  a (unique) continuous map
$h^\infty:B_{H^0_\infty}(\infty, R)\to X^\pm_\infty$ (which takes
values in $\bar B_{X^\pm_\infty}(\theta, \rho_A)$ in the case
$M(A)<\infty$ ) satisfying (\ref{e:1.10}) with $A=\nabla{\cal L}$,
an open set $V(R, r)$ in $H$ with $V(R, r)\subset \overline{C_{R,
r+\rho_A}}$, and a homeomorphism $\Phi: C_{R, \delta_r}\to V(R,r)$
such that
$$
\mathcal{ L}\circ\Phi(z, u^++ u^-)=\|u^+\|^2-\|u^-\|^2+ \mathcal{
L}(z+ h^\infty(z))
$$
for all $(z, u^+, u^-)\equiv z+ u^+ + u^-\in C_{R, \delta_r}$.
 Moreover, if ${\cal L}$ is $C^2$ then the map
$h^\infty$ is $C^1$ and the function $ B_{H^0_\infty}(\infty, R)\ni
z\mapsto \mathcal{ L}(z+ h^\infty(z))\in\R$ is $C^2$.
\end{corollary}

This corollary generalizes not only a slightly different version of
Bartsch-Li splitting lemmas at infinity \cite{BaLi} given in
\cite[Proposition~3.3]{HiLiWa} but also Theorem~2.1 in \cite{ChenLi1}.
Moreover, we do not need the assumption (\ref{e:1.5}). See Section~3.2
below for a detailed explanation.

The premise of the assumptions $({\rm E}_\infty)$ and $({\rm
E}'_\infty)$) is $\nu_\infty>0$. When $\nu_\infty=0$ the proofs of
Theorems~\ref{th:1.1},~\ref{th:1.3} cannot be completed if no
further conditions are imposed. The following  may be viewed as a
corresponding version of them in the case $\nu_\infty=0$.

\begin{theorem}\label{th:1.8}
Under the above assumptions $({\rm S})$, $({\rm F1_\infty})$--$({\rm
F3_\infty})$ and $({\rm C1_\infty})$--$({\rm C2_\infty})$, $({\rm
D_\infty})$,  also suppose that $\nu_\infty=0$ and that there exist constants $R>0$ and
 $\lambda\in (0, a_\infty)$ such that
\begin{eqnarray}
|{\cal L}(u)-(B(\infty)u,u)/2|\le\lambda\|u\|^2\quad&&\hbox{for all}\; u\in \bar B_H(\infty, R),\label{e:1.11}\\
\|A(u)-B(\infty)u\|\le\lambda\|u\|\quad&&\hbox{for all}\; u\in \bar
B_H(\infty, R)\cap X.\label{e:1.12}
\end{eqnarray}
\begin{enumerate}
\item[{\rm (a)}] If $\mu^-_\infty=0$ then there exist a number  $\mathfrak{R}>0$ and a homeomorphism
$\phi$ from $B_H(\infty, \mathfrak{R})$ onto an open subset of $H$
to satisfy:
\begin{eqnarray*}
&&{\cal L}(\phi(u))=\|u\|^2\quad\hbox{for all}\; u\in B_H(\infty, \mathfrak{R}),\\
&&\frac{\|u\|}{\sqrt{2a_\infty}}\le\|\phi(u)\|\le\frac{1}{\sqrt{a_\infty-\lambda}}\|u\|
\quad\hbox{for all}\; u\in B_H(\infty, \mathfrak{R}).
\end{eqnarray*}

\item[{\rm (b)}] If $\mu^-_\infty>0$ then there exist a number $\mathfrak{R}>0$  and
a homeomorphism $\phi$ from $B_{H^+_\infty}(\infty,
\mathfrak{R})\oplus H^-_\infty$ onto an open subset of $H$ such that
for all $(u,v)\in B_{H^+_\infty}(\infty, \mathfrak{R})\times
H^-_\infty$,
\begin{eqnarray*}
&{\cal L}(\phi(u+v))=\|u\|^2-\|v\|^2, &\\
&\frac{\|u\|}{\sqrt{2\|B(\infty)\|}}\le\|P^+_\infty\circ\phi(u+v)\|\le\sqrt{a_\infty-\lambda}\|u\|,&\\
&P^-_\infty\circ\phi\bigl(B_{H^+_\infty}(\infty,
\mathfrak{R})\oplus H^-_\infty\bigr)=H^-_\infty,&
\end{eqnarray*}
where $P^+_\infty$ and $P^-_\infty$ are the orthogonal projections
onto $H^+_\infty$ and $H^-_\infty$, respectively.
\end{enumerate}
 \end{theorem}

\begin{corollary}\label{cor:1.9}
Under the above assumptions $({\rm S})$, $({\rm F1_\infty})$--$({\rm
F3_\infty})$ and $({\rm C1_\infty})$--$({\rm C2_\infty})$, $({\rm
D_\infty})$, let $\nu_\infty=0$, (\ref{e:1.5}) hold and
\begin{equation}\label{e:1.13}
\|A(u)-B(\infty)u\|=o(\|u\|)\quad\hbox{as}\quad u\in
X\;\hbox{and}\;\|u\|\to\infty.
\end{equation}
Then the conclusions in Theorem~\ref{th:1.8} hold with
$\lambda=a_\infty/2$ and some $\mathfrak{R}>0$.
 \end{corollary}

 Perhaps, the condition (\ref{e:1.5}) (resp.
(\ref{e:1.11})) may be derived from (\ref{e:1.13}) (resp.
(\ref{e:1.12})). But the author does not know how to do.

One of main applications of the splitting lemmas at infinity is to
compute the critical group at infinity of ${\cal L}$, $C_\ast({\cal
L},\infty):=\lim_{\leftarrow}H_\ast(H, \{{\cal L}\le a\};\F)$ the
inverse limit of the system $\{H_\ast(H, {\cal L}^a)\to H_\ast(H,
{\cal L}^b)\,|\, -\infty<a\le b<\infty\}$, where the homomorphism
$H_\ast(H, {\cal L}^a)\to H_\ast(H, {\cal L}^b)$ is induced by the
inclusion $(H, {\cal L}^a))\hookrightarrow (H, {\cal L}^b))$.  In
the case $\nu_\infty=0$ and $\mu_\infty>0$ it follows from
(\ref{e:1.5}) that ${\cal L}$ is bounded from below on $H^+_\infty$
and that ${\cal L}(u)\to-\infty$ for $u\in H^-_\infty$ as
$\|u\|\to\infty$. By Proposition~3.8 of \cite{BaLi} we get that
$C_j({\cal L},\infty)=\delta_{kj}\F$ for $k=\mu_\infty=\dim
H^-_\infty$. If $\nu_\infty=\mu_\infty=0$ this also holds because
$C_\ast({\cal L},\infty)=H_\ast(H, \{\|u\|\ge R\};\F)$ for any
sufficiently large $R>0$.

For Theorems~\ref{th:1.1},~\ref{th:1.3} and~\ref{th:1.8} we can also
give a corresponding result with Theorem~2.25 of \cite{Lu2} or
Theorem~6.1 of \cite{Lu3}.

The proofs of Theorems~\ref{th:1.1},~\ref{th:1.3} and~\ref{th:1.8}
will be given in Section 2.  Some relations between these theorems
and previous ones will be discussed  in Section 3. In Section 4, as
a simple application we give a generalization of Theorem~5.2 in
\cite{BaLi}. It shows that our results may give better results even
if for $C^2$ functionals. Our theory can be used to deal with a
class of more general functionals of form $J(u)=\int_\Omega
F(x,u(x),\nabla u(x))dx$ (with lower smoothness than $C^2$ usually),
see \cite{Lu2}--\cite{Lu4}.

\section{Proofs of Main Theorems}\label{sec:2}
\setcounter{equation}{0}

 For reader's conveniences we here state  the following
parameterized version of Theorem 1.1 in \cite{DHK}. Its proof was
given in Appendix A of \cite{Lu2} and \cite{Lu3}.

\begin{theorem}\label{th:2.1}
Let $(H, \|\cdot\|)$ be a normed vector space and let $\Lambda$ be a
 compact topological space. Let $J:\Lambda\times B_H(\theta,
2\delta)\to\R$ be continuous, and for every $\lambda\in\Lambda$ the
function $J(\lambda, \cdot): B_H(\theta, 2\delta)\to\R$ is
continuously directional differentiable.
 Assume that there exist a closed
vector subspace $H^+$ and a finite-dimensional vector subspace $H^-$
of $H$ such that $H^+\oplus H^-$ is a direct sum decomposition of
$H$ and
\begin{enumerate}
\item[{\rm (a)}] $J(\lambda, \theta)=0$ and $D_2J(\lambda, \theta)=0$,
\item[{\rm (b)}] $[D_2J(\lambda, x+ y_2)-D_2J(\lambda, x+ y_1)](y_2-y_1)<0$ for any $(\lambda, x)\in\Lambda\times\bar
B_{H^+}(\theta,\delta)$, $y_1, y_2\in\bar B_{H^-}(\theta,\delta)$
and $y_1\ne y_2$,

\item[{\rm (c)}] $D_2J(\lambda, x+y)(x-y)>0$ for any $(\lambda, x, y)\in\Lambda\times\bar B_{H^+}(\theta,
\delta)\times\bar B_{H^-}(\theta,\delta)$ and $(x,y)\ne (\theta,
\theta)$,

\item[{\rm (d)}] $D_2J(\lambda, x)x>p(\|x\|)$ for any $(\lambda, x)\in\Lambda\times\bar
B_{H^+}(\theta,\delta)\setminus\{\theta\}$, where $p:(0, \delta]\to
(0, \infty)$ is a non-decreasing function.
\end{enumerate}
Then there exist a positive $\epsilon\in\R$, an open neighborhood
$U$ of $\Lambda\times\{\theta\}$ in $\Lambda\times H$ and a
homeomorphism
$$
\phi: \Lambda\times \bigl(B_{H^+}(\theta, \sqrt{p(\epsilon)/2})+
B_{H^-}(\theta, \sqrt{p(\epsilon)/2})\bigr)\to U
$$
such that
$$
J(\lambda, \phi(\lambda, x+ y))=\|x\|^2-\|y\|^2\quad\hbox{and}\quad
\phi(\lambda, x+ y)=(\lambda, \phi_\lambda(x+y))\in\Lambda\times H
$$
for all $(\lambda, x,y)\in \Lambda\times B_{H^+}(\theta,
\sqrt{p(\epsilon)/2})\times B_{H^-}(\theta, \sqrt{p(\epsilon)/2})$.
Moreover, for each $\lambda\in\Lambda$, $\phi_\lambda(0)=0$,
$\phi_\lambda(x+y)\in H^-$ if and only if $x=0$,  and $\phi$ is a
homoeomorphism from $\Lambda\times B_{H^-}(\theta,
\sqrt{p(\epsilon)/2})$ onto $U\cap (\Lambda\times H^-)$ according to
any topology on both induced by any norm on $H^-$.
\end{theorem}

\subsection{Proofs of Theorems~\ref{th:1.1},~\ref{th:1.3}}\label{sec:2.1}

Unlike the proof of \cite[Theorem~2.1]{Lu2,Lu3},
 we cannot directly apply Theorem~\ref{th:2.1} to the function
$F^\infty$ in (\ref{e:2.10}) because $\bar B_{H^0_\infty}(\infty,
R_1)$ is only locally compact. We must directly prove corresponding
conclusions with those in  Steps 1, 6, 7 of the proof of it given in
Appendix A of \cite{Lu2}, \cite{Lu3}. Moreover, in some steps we may prove
the same parts of Theorem~\ref{th:1.1} and Theorem~\ref{th:1.3} in a
unite way, in other steps we  must deal with those two cases,
respectively.

The following Lemma~\ref{lem:2.2} (resp. Lemma~\ref{lem:2.3}) is the
analogue of \cite[Lemma~2.13]{Lu2} or \cite[Lemma~3.1]{Lu3} under
the condition (${\rm E_\infty}$) (resp. (${\rm E'_\infty}$)).

\begin{lemma}\label{lem:2.2}
Under the above assumptions (${\rm S}$), $({\rm F1_\infty})$--$({\rm
F3_\infty})$ and $({\rm C1_\infty})$--$({\rm C2_\infty})$, and $({\rm
E_\infty})$  there exists  a unique map $h^\infty:\bar
B_{H^0_\infty}(\infty, R_1)\to \bar B_{X^\pm_\infty}(\theta, \rho_A)$ (by increasing $R_1>0$ if necessary),
 which is Lipschitz continuous,  such that
\begin{enumerate}
\item[{\rm (a)}] $(I-P^0_\infty)A(z+
h^\infty(z))=\theta$ for all $z\in \bar B_{H^0_\infty}(\infty,
R_1)$;
\item[{\rm (b)}] $h^\infty$
is strictly F-differentiable at infinity
 and $dh^\infty(\infty)=0$ under the assumption $({\rm
SE_\infty})$;

\item[{\rm (c)}] $\lim_{\|z\|_X\to\infty}\|h^\infty(z)\|_X=0$ provided
that $M(A)=0$ in the assumption $({\rm E}_\infty)$;

\item[{\rm (d)}]  the
function $\bar B_{H^0_\infty}(\infty, R_1)\to\R,\;z\mapsto \mathcal{
L}^\infty(z):=\mathcal{ L}(z+ h^\infty(z))$
 is  $C^{1}$ and
$$
d\mathcal{ L}^\infty(z_0)(z)=(A(z_0+ h^\infty(z_0)),
z)_H\quad\hbox{for all}\; z_0\in \bar B_{H^0_\infty}(\infty, R_1),\; z\in
H^0_\infty;
 $$

 \item[{\rm (e)}] If $P^0_\infty\circ A:X\to X^0_\infty$ has a  strict Fr\'{e}chet derivative $S\in L(X, X^0_\infty)$
 at infinity, (for instance this is true when $A$ is strictly
F-differentiable at infinity), then the function $\mathcal{
L}^\infty$
 is  $C^{2-0}$, and $d\mathcal{ L}^\infty$ has a strict
Fr\'{e}chet derivative zero provided $S=P^0_\infty\circ
B(\infty)|_{X}$;

\item[{\rm (f)}] If $A$ is $C^1$ the maps $h^\infty$ and $\mathcal{
L}^\infty$ are $C^1$ and $C^2$, respectively, and
$$
dh^\infty(z)=-\bigl[(I-P^0_\infty)A'(z+
h^\infty(z))|_{X^\pm_\infty}\bigr]^{-1}(I-P^0_\infty)A'(z+
h^\infty(z))|_{H^0_\infty};
$$
\item[{\rm (g)}] If $\mathcal{ L}$ is  $C^2$ then $h^\infty$ is also $C^1$ as a map to
$H^\pm_\infty$ (hence $X^\pm_\infty$).
\end{enumerate}
\end{lemma}

\begin{proof}
{\rm (a)}  Consider the map $S^\infty: \bar B_{H^0_\infty}(\infty, R_1)\times \bar
B_{X^\pm_\infty}(\theta, \rho_A)\to X^\pm_\infty$,
\begin{eqnarray*}
(z,x)\mapsto
-(B(\infty)|_{X^\pm_\infty})^{-1}(I-P^0_\infty)A(z+x)+ x.
\end{eqnarray*}
 Let $z_1, z_2\in \bar B_{H^0_\infty}(\infty, R_1)$, and let
$x_1, x_2\in \bar B_{X^\pm_\infty}(\theta, \rho_A)$. Noting
that $B(\infty)x_i\in X^\pm_\infty$ and $B(\infty)z_i=0$, $i=1,2$,
it follows from (\ref{e:1.2}) that
\begin{eqnarray}\label{e:2.1}
&&\quad\|S^\infty(z_1, x_1)-S^\infty(z_2, x_2)\|_{X^\pm_\infty}\\
&&\le C^\infty_1\cdot\|(I-P^0_\infty)A(z_1+x_1)- B(\infty)x_1 -
(I-P^0_\infty)A(z_2+x_2)+ B(\infty)x_2\|_{X^\pm_\infty}\nonumber\\
&&\le\frac{1}{\kappa}\|z_1+ x_1-z_2-x_2\|_X.\nonumber
\end{eqnarray}
In particular, for any $z\in \bar B_{H^0_\infty}(\infty, R_1)$ and
$x_1, x_2\in \bar B_{X^\pm_\infty}(\theta, \rho_A)$, we get
\begin{eqnarray}\label{e:2.2}
\|S^\infty(z, x_1)-S^\infty(z,
x_2)\|_{X^\pm_\infty}\le\frac{1}{\kappa}\|x_1-x_2\|_X.
\end{eqnarray}

$\bullet$ If $\rho_A<\infty$  in $({\rm
E_\infty})$, this means that  $M(A)<\infty$ and
 $\rho_A\in (\frac{\kappa}{\kappa-1}C_1^\infty M(A),
\infty)$.  By
increasing $R_1>\rho_A$ we may derive
$$
\sup\{\|(I-P^0_\infty)A(z)\|_X:\;z\in H^0_\infty, \|z\|_X\ge R_1\}
\le \frac{\kappa-1}{\kappa}\frac{\rho_A}{C_1^\infty}
$$
and hence
$$
\|S^\infty(z, \theta)\|_{X^\pm_\infty}\le
\|(B(\infty)|_{X^\pm_\infty})^{-1}\|_{L(X^\pm_\infty)}\cdot\|(I-P^0_\infty)A(z)\|_{X^\pm_\infty}\le
\frac{\kappa-1}{\kappa}\rho_A.
$$
It follows from this and (\ref{e:2.2}) that
\begin{eqnarray}\label{e:2.3}
\|S^\infty(z, x)\|_{X^\pm_\infty}&\le & \|S^\infty(z, x)-S^\infty(z,
\theta)\|_{X^\pm_\infty}+ \|S^\infty(z, \theta)\|_{X^\pm_\infty}\nonumber\\
&\le&\frac{1}{\kappa}\|x\|_X + \frac{\kappa-1}{\kappa}\rho_A
\le \frac{1}{\kappa}\rho_A+ \frac{\kappa-1}{\kappa}\rho_A \le\rho_A
\end{eqnarray}
for any $z\in \bar B_{H^0_\infty}(\infty, R_1)$ and $x\in \bar
B_{X^\pm_\infty}(\theta, \rho_A)$. Hence the Banach fixed point theorem gives a unique
map
$h^\infty:\bar B_{H^0_\infty}(\infty, R_1)\to \bar
B_{X^\pm_\infty}(\theta,\rho_A)$,
 which is also continuous,
such that $S^\infty(z, h^\infty(z))=h^\infty(z)$ or equivalently
\begin{equation}\label{e:2.4}
 (I-P^0_\infty)A(z+ h^\infty(z))=\theta\quad
\hbox{for all}\; z\in \bar B_{H^0_\infty}(\infty, R_1).
\end{equation}
This and (\ref{e:2.1}) imply
\begin{eqnarray*}
\|h^\infty(z_1)-h^\infty(z_2)\|_{X}&=&\|S^\infty(z_1, h^\infty(z_1))-
S^\infty(z_2, h^\infty(z_2))\|_{X^\pm_\infty}\\
&\le&\frac{1}{\kappa}\|z_1+ h^\infty(z_1)-z_2-h^\infty(z_2)\|_X
\end{eqnarray*}
and hence
\begin{equation}\label{e:2.5}
\|h^\infty(z_1)-h^\infty(z_2)\|_{X}\le
\frac{1}{\kappa-1}\|z_1-z_2\|_X\quad\hbox{for all}\; z_1, z_2\in \bar
B_{H^0_\infty}(\infty, R_1).
\end{equation}
That is, $h^\infty$ is Lipschitz  continuous.

$\bullet$ If $\rho_A=\infty$ in $({\rm E_\infty})$,  then
(\ref{e:2.1}) holds for any $z\in\bar B_{H^0_\infty}(\infty, R_1)$
and $x_1, x_2\in X^\pm_\infty$. The Banach fixed point theorem gives
a unique map $ h^\infty:\bar B_{H^0_\infty}(\infty, R_1)\to
X^\pm_\infty$,
  which is continuous, such that (\ref{e:2.4}) and (\ref{e:2.5}) also hold.

{\rm (b)} If $M(A)<\infty$ in (${\bf SE_\infty}$) we choose
$\kappa>1$ so large that  $\rho_A>\frac{\kappa}{\kappa-1}C_1^\infty
M(A)$. Then (\ref{e:1.2}) is satisfied by increasing $R_1>0$ (if
necessary). Hence (\ref{e:2.1})--(\ref{e:2.5}) are still effective
for these $\kappa$ and $R_1$. For $z_i\in \bar
B_{H^0_\infty}(\theta, R_1)$ set $x_i=h^\infty(z_i)$ in
(\ref{e:2.1}), $i=1,2$. We obtain
\begin{eqnarray}\label{e:2.6}
&&\|h^\infty(z_1)-h^\infty(z_2)\|_{X^\pm_\infty}=
\|S^\infty(z_1, h^\infty(z_1))-S^\infty(z_2, h^\infty(z_2))\|_{X^\pm_\infty}\nonumber\\
&\le& C^\infty_1\cdot \|(I-P^0_\infty)A(z_1 + h^\infty(z_1))- B(\infty)(z_1+h^\infty(z_1))\nonumber\\
&&\hspace{15mm}-(I-P^0_\infty)A( h^\infty(z_2))+ B(\infty)(z_2+
h^\infty(z_2))\|_{X}.
\end{eqnarray}
 For any given small $\varepsilon>0$, since
$$
\|z_i+ h^\infty(z_i)\|_X^2\ge\|z_i +
h^\infty(z_i)\|^2=\|z_i\|^2+\|h^\infty(z_i)\|^2\ge\|z_i\|^2,
$$
and $\|z_i\|\to\infty\Longleftrightarrow \|z_i\|_X\to\infty$ for
$z_i\in H^0_\infty$, $i=1,2$,  by (${\bf SE_\infty}$) there exists
$R>R_1$ such that for any $z_i\in \bar B_{H^0_\infty}(\infty, R)$,
$i=1,2$ we have
\begin{eqnarray*}
&&\|(I-P^0_\infty)A(z_1 + h^\infty(z_1)) - B(\infty)(z_1+
h^\infty(z_1))\nonumber\\
&&\hspace{20mm} -(I-P^0_\infty)A(z_2 +
h^\infty(z_2))+ B(\infty)(z_2+h^\infty(z_2))\|_{X}\nonumber\\
&&\le\varepsilon\|z_1 + h^\infty(z_1)-z_2 -
h^\infty(z_2)\|_X
\le \frac{\kappa}{\kappa-1}\varepsilon\|z_1-z_2\|_X
\end{eqnarray*}
by (\ref{e:2.5}).  From this and (\ref{e:2.6}) we derive that
 \begin{equation}\label{e:2.7}
\|h^\infty(z_2)- h^\infty(z_1)\|_{X}\le
\frac{\kappa}{\kappa-1}C^\infty_1\varepsilon \|z_2-z_1\|_X
\end{equation}
for any $z_i\in \bar B_{H^0_\infty}(\infty, R)$, $i=1,2$. This shows
that $h^\infty$ has the strict Fr\'{e}chet derivative zero  at
$\infty$.

{\rm (c)} Recall that $h^\infty(z)$ is a unique fixed point in
$\bar B_{X^\pm_\infty}(\theta, \rho_A)$ of the map
$$
x\mapsto
S^\infty(z,x)=-(B(\infty)|_{X^\pm_\infty})^{-1}(I-P^0_\infty)[A(z+x)-
B(\infty)x].
$$
Since $M(A)=0$, for any small $0<\epsilon<\rho_A$ there exists a
large $R>R_1$ such that
$$
\|(I-P^0_\infty)A(z)\|_{X^\pm}<\frac{(\kappa-1)\epsilon}{C_1^\infty\kappa}
$$
for any $z\in \bar B_{H^0_\infty}(\infty, R)$. By the deduction of
(\ref{e:2.3}), for any $z\in \bar B_{H^0_\infty}(\infty, R)$ and
$x\in \bar B_{X^\pm_\infty}(\theta, \epsilon)$ we have
\begin{eqnarray*}
\|S^\infty(z, x)\|_{X^\pm}&\le& \frac{1}{\kappa}\|x\|_X +
\|(B(\infty)|_{X^\pm_\infty})^{-1}(I-P^0_\infty)A(z)\|_{X^\pm}\\
&\le& \frac{1}{\kappa}\|x\|_X +
C_1^\infty\|(I-P^0_\infty)A(z)\|_{X^\pm_\infty}
\le\frac{\epsilon}{\kappa}+\frac{(\kappa-1)\epsilon}{\kappa}<\epsilon.
 \end{eqnarray*}
So the map
$\bar B_{X^\pm_\infty}(\theta, \epsilon)\to \bar B_{X^\pm_\infty}(\theta,
\epsilon),\;x\mapsto S^\infty(z, x)$
has a unique fixed point, which is, of course, contained in $\bar
B_{X^\pm_\infty}(\theta, \rho_A)$ and hence must be
$h^\infty(z)$. This shows $\|h^\infty(z)\|_X\le\epsilon$.

{\rm (d)} The proof is similar to  Step 2 of proof of
\cite[Lemma~2.13]{Lu2} or \cite[Lemma~3.1]{Lu3}. For any $z_0\in
\bar B_{H^0_\infty}(\infty, R_1)$, $z\in H^0_\infty$ and
$t\in\R\setminus\{0\}$ with $z_0+ tz\in\bar B_{H^0_\infty}(\infty,
R_1)$,  by the mean value theorem we have $s\in (0, 1)$ such that
\begin{eqnarray}\label{e:2.8}
&&\mathcal{ L}^\infty(z_0+ tz)-\mathcal{ L}^\infty(z_0)\nonumber\\
&=&D\mathcal{ L}(z_{s,t})(tz + h^\infty(z_0+ tz)-h^\infty(z_0))\nonumber\\
&=&(A(z_{s,t}), tz + h^\infty(z_0+ tz)-h^\infty(z_0))_H\nonumber\\
&=&(A(z_{s,t}), tz)_H + ((I-P^0_\infty)A(z_{s,t}), h^\infty(z_0+
tz)-h^\infty(z_0))_H
\end{eqnarray}
because $h^\infty(z_0+ tz)-h^\infty(z_0)\in X^\pm_\infty\subset
H^\pm_\infty$, where
$z_{s,t}=z_0+ h^\infty(z_0)+ s[tz+ h^\infty(z_0+
tz)-h^\infty(z_0)]$.
Note that (\ref{e:2.5}) implies
$$
\|h^\infty(z_0+ tz)-h^\infty(z_0)\|_H\le\|h^\infty(z_0+
tz)-h^\infty(z_0)\|_X\le \frac{1}{\kappa-1}|t|\cdot\|z\|_X.
$$
Let $t\to 0$, we have
\begin{eqnarray*}
&&\left|\frac{((I-P^0_\infty)A(z_{s,t}), h^\infty(z_0+
tz)-h^\infty(z_0))_H}{t}\right|\\
&\le&\frac{\|(I-P^0_\infty)A(z_{s,t})\|_H\cdot\|h^\infty(z_0+
tz)-h^\infty(z_0)\|_H}{|t|}\\
&\le& \frac{1}{\kappa-1}\|z\|_X\|(I-P^0_\infty)A(z_{s,t})\|_{X}\\
&\to& \frac{1}{\kappa-1}\|z\|_X\cdot\|(I-P^0_\infty)A(z_0+
h^\infty(z_0))\|_{X}=0
\end{eqnarray*}
because of (\ref{e:2.4}). From this and (\ref{e:2.8}) it follows
that
$$
D\mathcal{ L}^\infty(z_0)(z)=\lim_{t\to 0}\frac{\mathcal{
L}^\infty(z_0+ tz)-\mathcal{ L}^\infty(z_0) }{t}= (A(z_0+
h^\infty(z_0)), z)_H.
$$
That is, $\mathcal{ L}^\infty$ is G\^{a}teaux differentiable at
$z_0$. Clearly, $z\mapsto D\mathcal{ L}^\infty(z_0)(z)$ is linear
and continuous, i.e. $\mathcal{ L}^\infty$ has a linear bounded
G\^{a}teaux derivative at $z_0$, $D\mathcal{ L}^\infty(z_0)$, given
by
$$
D\mathcal{ L}^\infty(z_0)z=(A(z_0+ h^\infty(z_0)), z)_H=(P^0_\infty
A(z_0+ h(z_0)), z)_H\;\hbox{for all}\; z\in H^0_\infty.
$$
Note that  $B(\infty)|_{H^0_\infty}=0$,
$B(\infty)(H^\pm_\infty)\subset H^\pm_\infty$ and $h^\infty(z_0),
h^\infty(z_0')\in X^\pm_\infty\subset H^\pm_\infty$ for any $z_0,
z_0'\in \bar B_{H^0_\infty}(\infty, R_1)$. We have
\begin{eqnarray*}
(P^0_\infty B(\infty)(z_0+ h^\infty(z_0)), z)_H=(P^0_\infty
B(\infty)(z'_0+ h^\infty(z'_0)), z)_H=0
\end{eqnarray*}
for all $z\in H^0_\infty$.
From this  it easily follows that
\begin{eqnarray*}
&&|D\mathcal{ L}^\infty(z_0)z- D\mathcal{ L}^\infty(z'_0)z|\nonumber\\
&=&\left|\bigl(P^0_\infty A(z_0+ h^\infty(z_0))-P^0_\infty A(z'_0+ h^\infty(z'_0)), z\bigr)_H\right|\nonumber\\
&=&\bigl|\bigl(P^0_\infty A(z_0+ h^\infty(z_0))- P^0_\infty B(\infty)(z_0+ h^\infty(z_0)), z\bigr)_H\nonumber\\
&&-\bigl(P^0_\infty A(z'_0+ h^\infty(z'_0))- P^0_\infty B(\infty)(z'_0+ h^\infty(z'_0)), z\bigr)_H\bigr|\nonumber\\
&\le &\|P^0_\infty A(z_0+ h^\infty(z_0))- P^0_\infty B(\infty)(z_0+ h^\infty(z_0))\nonumber\\
&&-P^0_\infty A(z'_0+h^\infty(z'_0))+ P^0_\infty B(\infty)(z'_0+ h^\infty(z'_0))\|_H\cdot\|z\|_H\nonumber\\
&\le &\Big[\|A(z_0+ h^\infty(z_0))-  A(z'_0+h^\infty(z'_0))\|_X+
\nonumber\\
&&\|B(\infty)(z_0+ h^\infty(z_0))- B(\infty)(z'_0+
h^\infty(z'_0))\|_H\Big]\cdot\|z\|_X
\end{eqnarray*}
and hence
\begin{eqnarray*}
\|D\mathcal{ L}^\infty(z_0)- D\mathcal{
L}^\infty(z'_0)\|_{(X^0_\infty)^\ast} &\le &\|A(z_0+ h^\infty(z_0))-
A(z'_0+h^\infty(z'_0))\|_X
\nonumber\\
&+&\|B(\infty)(z_0+ h^\infty(z_0))- B(\infty)(z'_0+
h^\infty(z'_0))\|_H,
\end{eqnarray*}
where $(X^0_\infty)^\ast=(H^0_\infty)^\ast=L(X^0_\infty,\R)$.

Since
both $A:X\to X$ and $B(\infty):H\to H$ are continuous by (${\rm
F2}_\infty$), from (\ref{e:2.5})  we derive that $z_0\mapsto
D\mathcal{ L}^\infty(z_0)$ is continuous and therefore
 that $\mathcal{ L}^\infty$ is Fr\'echet differentiable at
$z_0$ and its Fr\'echet differential $d\mathcal{
L}^\infty(z_0)=D\mathcal{ L}^\infty(z_0)$. Moreover, the above
estimate also shows that $z_0\mapsto d\mathcal{ L}^\infty(z_0)$ is
continuous.

{\rm (e)} Since $P^0_\infty\circ A$ has the strict Fr\'{e}chet
derivative $S\in L(X, X^0_\infty)$  at $\infty$ then
\begin{equation}\label{e:2.9}
\|P^0_\infty\circ A(x_1)- P^0_\infty\circ A(x_2)-S(x_1-x_2)\|_X\le
\widehat K_R\|x_1-x_2\|_X
\end{equation}
for all $x_1, x_2\in B_X(\infty, R)$ with constant $\widehat K_R\to
0$ as $R\to \infty$.

Let $C>0$ be such that $\|z\|_X\le C\|z\|\;\forall z\in H^0_\infty$.
For $R>R_1$ and any $z_0, z_0'\in B_{H^0_\infty}(\infty, R)$, since
\begin{eqnarray*}
\|z+ h^\infty(z)\|^2_X\ge\|z+ h^\infty(z)\|^2=\|z\|^2+
\|h^\infty(z)\|^2\ge\|z\|^2\quad\hbox{for}\;z=z_0, z_0',
\end{eqnarray*}
it follows from the proof of (d), (\ref{e:2.9}) and (\ref{e:2.5})
that
\begin{eqnarray*}
&&|d\mathcal{ L}^\infty(z_0)z- d\mathcal{ L}^\infty(z'_0)z- (S(z_0+ h^\infty(z_0)- z'_0-
h^\infty(z'_0) ),z)_H|\nonumber\\
&=&\Bigl|\bigl(P^0_\infty A(z_0+ h^\infty(z_0))-P^0_\infty A(z'_0+ h^\infty(z'_0)), z\bigr)_H-\nonumber\\
&&\qquad (S(z_0+ h^\infty(z_0)- z'_0- h^\infty(z'_0) ),z)_H\Bigr|\nonumber\\
&\le &\|P^0_\infty A(z_0+ h^\infty(z_0))- P^0_\infty
A(z'_0+h^\infty(z'_0))\nonumber\\
&&\qquad - S(z_0+ h^\infty(z_0)- z'_0- h^\infty(z'_0) )
\|_H\cdot\|z\|_H\nonumber\\
&\le& \|P^0_\infty A(z_0+ h^\infty(z_0))-
P^0_\infty A(z'_0+h^\infty(z'_0))\nonumber\\
&&\qquad - S(z_0+ h^\infty(z_0)- z'_0- h^\infty(z'_0) )\|_X\cdot\|z\|_X\nonumber\\
 &&\hspace{-5mm}\le \widehat K_{R}\cdot\|z_0 + h^\infty(z_0)-z_0' -
h^\infty(z_0')\|_X\cdot\|z\|_X\nonumber\\
&\le& \frac{\kappa}{\kappa-1}\widehat K_{R}\cdot
\|z_0-z_0'\|_X\cdot\|z\|_X\\
&\le& \frac{\kappa}{\kappa-1}C^2\widehat K_{R}\cdot
\|z_0-z_0'\|_X\cdot\|z\|
\end{eqnarray*}
for any $z\in H^0_\infty$. Hence
\begin{eqnarray*}
&&\|d\mathcal{ L}^\infty(z_0)- d\mathcal{ L}^\infty(z'_0)\|_{L(H^0_\infty,\R)}\nonumber\\
&\le& \frac{\kappa}{\kappa-1}C^2\widehat K_{R}\cdot \|z_0-z_0'\|_X+
\|S(z_0+ h^\infty(z_0)- z'_0- h^\infty(z'_0)) \|_X\\
&\le& \frac{\kappa}{\kappa-1}(1+C^2\widehat K_{R})\cdot
\|z_0-z_0'\|_X,
\end{eqnarray*}
that is,  $\mathcal{ L}^\infty$ is $C^{2-0}$. Moreover, if
$S=P^0_\infty\circ B(\infty)|_{X}$, then
$(S(z_0+ h^\infty(z_0)-
z'_0- h^\infty(z'_0) ),z)_H=0$ for all $z\in H^0_\infty$,
and hence
\begin{eqnarray*}
&&|d\mathcal{ L}^\infty(z_0)z- d\mathcal{ L}^\infty(z'_0)z\|\\
&=&|d\mathcal{ L}^\infty(z_0)z- d\mathcal{ L}^\infty(z'_0)z- (S(z_0+ h^\infty(z_0)-
z'_0- h^\infty(z'_0) ),z)_H|\nonumber\\
&\le& \frac{\kappa}{\kappa-1}C^2\widehat K_{R}\cdot
\|z_0-z_0'\|_X\cdot\|z\|
\end{eqnarray*}
for any $z\in H^0_\infty$. This implies
$$
\frac{\|d\mathcal{ L}^\infty(z_0)- d\mathcal{
L}^\infty(z'_0)\|_{L(H^0_\infty,\R)}}{\|z_0-z'_0\|}\to 0
$$
as $(\|z_0\|, \|z_0'\|)\to (\infty, \infty)$ and $z_0\ne z_0'$.
Hence  $d\mathcal{ L}^\infty$ has the  strict Fr\'{e}chet derivative
zero at infinity.

{\rm (f)}  Since $A$ is $C^1$ the corresponding conclusions can be
obtained as in \cite[Remark~2.14]{Lu2} or \cite[Remark~3.2]{Lu3}.

{\rm (g)} If  $\mathcal{ L}$ is $C^2$ then $\nabla\mathcal{
L}(x)=A(x)\;\forall x\in X_\infty$. For $z_0\in \bar
B_{H^0_\infty}(\infty, R_1)$ we have $(I-P^0_\infty)\nabla\mathcal{
L}(z_0+ h^\infty(z_0))=\theta$. By the implicit function theorem
there exists a neighborhood $\mathcal{ O}(z_0)$ of $z_0$ in $\bar
B_{H^0_\infty}(\infty, R_1)$ and a unique $C^1$ map $h:\mathcal{
O}(z_0)\to H^\pm_\infty$ such that $(I-P^0_\infty)\nabla\mathcal{
L}(z+ h(z))=\theta$ for all $z\in\mathcal{ O}(z_0)$. Moreover,
$$
(I-P^0_\infty)\nabla\mathcal{ L}(z+ h^\infty(z))=(I-P^0_\infty)A(z+
h^\infty(z))=\theta
$$
for all $z\in \bar B_{H^0_\infty}(\infty, R_1)$, and $h^\infty$ is
also continuous as a map to $H^\pm_\infty$, by the implicit function
theorem (precisely its proof) we get $h(z)=h^\infty(z)\;\forall
z\in\mathcal{ O}(z_0)$. The desired conclusion is proved.
\end{proof}

\begin{lemma}\label{lem:2.3}
Under the above assumptions $({\rm S})$, $({\rm F1_\infty})$--$({\rm
F3_\infty})$ and $({\rm C1_\infty})$--$({\rm C2_\infty})$, and $({\rm
E'_\infty})$  there exist $R_1>0$ and  a unique map
$$
h^\infty:\bar B_{H^0_\infty}(\infty, R_1)\to \bar B_{X}(\theta,
\rho_A)\cap X^\pm_\infty,
$$
which is continuous, such
that
\begin{enumerate}
\item[{\rm (a)}] $(I-P^0_\infty)A(z+
h^\infty(z))=\theta$ for all $z\in \bar B_{H^0_\infty}(\infty,
R_1)$;

\item[{\rm (b)}] $\lim_{\|z\|_X\to\infty}\|h^\infty(z)\|_X=0$ provided
that $M(A)=0$ in (${\rm E'_\infty}$);

\item[{\rm (c)}] If $A$ is $C^1$, then $h^\infty$ is $C^1$
and
$$
dh^\infty(z)=-\bigl[(I-P^0_\infty)A'(z+
h^\infty(z))|_{X^\pm_\infty}\bigr]^{-1}(I-P^0_\infty)A'(z+
h^\infty(z))|_{H^0_\infty}.
$$
Moreover, the functional
$\mathcal{ L}^\infty: \bar
B_{H^0_\infty}(\infty, R_1)\to\R,\;z\mapsto \mathcal{ L}(z+
h^\infty(z))$
 is $C^{2}$ and
$d\mathcal{ L}^\infty(z_0)(z)=(A(z_0+ h^\infty(z_0)),
z)_H$ for all $z_0\in \bar B_{H^0_\infty}(\infty, R_1)$ and $z\in
H^0_\infty$;
 \item[{\rm (d)}] If $\mathcal{ L}$ is  $C^2$ then $h^\infty$ is also $C^1$ as a map to
$H^\pm_\infty$ (hence $X^\pm_\infty$).
\end{enumerate}
\end{lemma}

 \begin{proof} Recall the proof of
Lemma~\ref{lem:2.2}(a). Under the condition (${\rm E'_\infty}$), we
can only obtain (\ref{e:1.4}) and (\ref{e:2.1}) for $z_1=z_2$. Hence
(\ref{e:2.2}) still holds. Unless (\ref{e:2.1}) and (\ref{e:2.5})
the proof of Lemma~\ref{lem:2.2}(a) is valid.

The proof of (b) is
the same as that of Lemma~\ref{lem:2.2}(c). (c)--(d) can be
obtained by the implicit function theorem as usual.
\end{proof}

 Define a continuous map
\begin{equation}\label{e:2.10}
F^\infty:\bar B_{H^0_\infty}(\infty, R_1)\times H^\pm_\infty\to\R
\end{equation}
by $F^\infty(z, u)=\mathcal{ L}(z+ h^\infty(z)+ u)-\mathcal{ L}(z+
h^\infty(z))$.  Then for each $z\in \bar B_{H^0_\infty}(\infty,
R_1)$ the map $F^\infty(z,\cdot)$ is continuously directional
differentiable on $H^\pm_\infty$, and the directional derivative of
it at $u\in H^\pm_\infty$ in any direction $v\in H^\pm_\infty$ is
given by
\begin{eqnarray*}
D_2F^\infty(z,u)(v)&=&(A(z+ h^\infty(z)+ u), v)_H
=((I-P^0_\infty)A(z+ h^\infty(z)+ u), v)_H.
\end{eqnarray*}
 It follows from this and (\ref{e:2.4}) that
\begin{eqnarray}\label{e:2.11}
F^\infty(z, \theta)=0\quad\hbox{and}\quad D_2F^\infty(z,
\theta)(v)=0\quad\hbox{for all}\; v\in H^\pm_\infty.
\end{eqnarray}
Later on, if (\ref{e:1.5}) holds we shall assume (by increasing
$R_1>0$) that
\begin{eqnarray}\label{e:2.12}
-\frac{a_\infty}{8}\|z+u\|^2\le {\cal
L}(z+u)-\frac{1}{2}(B(\infty)u, u)_H\le \frac{a_\infty}{8}\|z+u\|^2
\end{eqnarray}
for any $(z, u)\in\bar B_{H^0_\infty}(\infty, R_1)\times
H^\pm_\infty$.

Under the assumptions $({\rm C1_\infty})$--$({\rm C2_\infty})$ and
$({\rm D_\infty})$, with the same proof methods we can obtain the
corresponding results with \cite[Lemma~2.15]{Lu2} and
\cite[Lemma~2.16]{Lu2} (or \cite[Lemma~3.3]{Lu3} and
\cite[Lemma~3.4]{Lu3}) as follows.

\begin{lemma}\label{lem:2.4}
 There exists a
function $\omega_\infty:V_\infty\cap X\to [0, \infty)$  with the
property that $\omega_\infty(x)\to 0$
  as $x\in V_\infty\cap X$ and $\|x\|\to
\infty$, such that
$$
|(B(x)u, v)_H- (B(\infty)u, v)_H |\le \omega_\infty(x)
\|u\|\cdot\|v\|
$$
for any $x\in V_\infty\cap X$,  $u\in H^0_\infty\oplus H^-_\infty$
and $v\in H$.
\end{lemma}

\begin{lemma}\label{lem:2.5}
Let $a_\infty>0$ as in (\ref{e:1.1}). By increasing $R_1$ we may
find a number $a_1\in (0, 2a_\infty]$ such that for any $x\in \bar
B_H(\infty, R_1)\cap X$ one has
\begin{enumerate}
\item[{\rm (a)}] $(B(x)u, u)_H\ge a_1\|u\|^2\;\hbox{for all}\; u\in H^+_\infty$;
\item[{\rm (b)}] $|(B(x)u,v)_H|\le\omega_\infty(x)\|u\|\cdot\|v\|\;\hbox{for all}\; u\in H^+_\infty\;\hbox{and}\; \hbox{all}\; v\in
H^-_\infty\oplus H^0_\infty$;
\item[{\rm (c)}] $(B(x)u,u)_H\le-a_\infty\|u\|^2\;\hbox{for all}\; u\in H^-_\infty$.
\end{enumerate}
\end{lemma}

 Note: Actually, for the proof of Theorem~\ref{th:1.1} (resp.
Theorem~\ref{th:1.3}) we only need that Lemmas~\ref{lem:2.4} and
\ref{lem:2.5} hold in a set of form
$$
\bar B_{H^0_\infty}(\infty, R')\oplus X^\pm_\infty\quad\hbox{(resp.
$\bar B_{H^0_\infty}(\infty, R')\oplus (\bar B_H(\theta, r')\cap
X^\pm_\infty)$)}.
$$
 In this case we can only get the following
Lemma~\ref{lem:2.6} in such a set too.

As in the proof of \cite[Lemma~2.17]{Lu2} or \cite[Lemma~3.5]{Lu3}
we can use the above lemmas to prove:

\begin{lemma}\label{lem:2.6}
The  functional $F^\infty$ in (\ref{e:2.10}) satisfies (i)-(iv) in
Theorem~\ref{th:2.1}, i.e.
\begin{enumerate}
\item[{\rm (a)}] $F^\infty(z, \theta)=0$ and $D_2F^\infty(z, \theta)=0$ for any $z\in \bar B_{H^0_\infty}(\infty, R_1)$;
\item[{\rm (b)}] $[D_2F^\infty(z, u+ v_2)-D_2F^\infty(z, u+ v_1)](v_2-v_1)\le -a_\infty\|v_2-v_1\|^2<0$
for any $(z, u)\in \bar B_{H^0_\infty}(\infty, R_1)\times
H^+_\infty$, $v_1, v_2\in  H^-_\infty$ with $v_1\ne v_2$;

\item[{\rm (c)}] $D_2F^\infty(z, u+ v)(u-v)\ge a_1\|u\|^2+ a_\infty\|v\|^2>0$ for any
$(z, u, v)\in \bar B_{H^0_\infty}(\infty, R_1)\\
\times H^+_\infty\times H^-_\infty$ with $(u, v)\ne (\theta,
\theta)$;

\item[{\rm (d)}] $D_2F^\infty(z, u)u\ge a_1\|u\|^2> p(\|u\|)$ for any
$(z, u)\in \bar B_{H^0_\infty}(\infty, R_1)\times H^+_\infty$ with
$u\ne \theta$, where $p(t)=\frac{a_1}{2}t^2$.
\end{enumerate}
\end{lemma}

\begin{proof}
 By (\ref{e:2.11}) it suffices to prove that
$F^\infty$ satisfies conditions (b)--(d).

{\bf Step 1}. For  any $z\in \bar B_{H^0_\infty}(\infty, R_1)$,
$u^+\in  X^+_\infty$ and $u^-_1, u^-_2\in H^-_\infty$, as in the
proof of \cite[Lemma~2.17]{Lu2} or \cite[Lemma~3.5]{Lu3}, since the
function
$$
u\mapsto (A(z+ h^\infty(z)+ u^++u), u^-_2-u^-_1)_H.
$$
is continuously directional differentiable, by the condition (${\rm
F2_\infty}$) and the mean value theorem we have a number $t\in (0,
1)$ such that
\begin{eqnarray*}
\hspace{-10mm}&&\hspace{-8mm}[D_2F^\infty(z, u^+ + u^-_2)-D_2F^\infty(z, u^++ u^-_1)](u^-_2-u^-_1)\\
=&&\hspace{-6mm}(A(z+ h^\infty(z)+ u^++u^-_2), u^-_2-u^-_1)_H -
(A(z+
h^\infty(z)+ u^++u^-_1), u^-_2-u^-_1)_H\\
=&&\hspace{-6mm}\left(DA(z+ h^\infty(z)+ u^++ u^-_1+
t(u^-_2-u^-_1))(u^-_2-u^-_1),
u^-_2-u^-_1\right)_H\\
=&&\hspace{-6mm}\left(B(z+ h(z)+ u^++ u^-_1+
t(u^-_2-u^-_1))(u^-_2-u^-_1),
u^-_2-u^-_1\right)_H\\
\le &&\hspace{-6mm} -a_\infty\|u^-_2-u^-_1\|^2,
\end{eqnarray*}
where the third equality comes from $({\rm F3_\infty})$, and the
final inequality is due to Lemma~\ref{lem:2.5}(c). Hence the
density of $X^+_\infty$ in $H^+_\infty$ leads to
$$
[D_2F^\infty(z, u^+ + u^-_2)-D_2F^\infty(z, u^++
u^-_1)](u^-_2-u^-_1)\le -a_0\|u^-_2-u^-_1\|^2
$$
for all $z\in \bar B_{H^0_\infty}(\infty, R_1)$, $u^+\in H^+$ and
$u^-_1, u^-_2\in H^-$. This implies the condition (b).

{\bf Step 2}.  For $z\in \bar B_{H^0_\infty}(\infty, R_1)$, $u^+\in
X^+_\infty$ and $u^-\in H^-_\infty$,  using (\ref{e:2.11}), the mean
value theorem and  (${\rm F2_\infty}$)--(${\rm F3_\infty}$), for some
$t\in (0, 1)$ we have
\begin{eqnarray*}
\hspace{-3mm}&&D_2F^\infty(z, u^++u^-)(u^+-u^-)\\
\hspace{-3mm}&=&\hspace{-3mm} D_2F^\infty(z, u^++u^-)(u^+-u^-)- D_2F^\infty(z, \theta)(u^+-u^-)\\
\hspace{-3mm}&=&\hspace{-3mm}(A(z+ h^\infty(z)+ u^++u^-), u^+-u^-)_H-(A(z+ h^\infty(z)+ \theta), u^+-u^-)_H\\
\hspace{-3mm}&=&\hspace{-3mm}\left(B(z+ h^\infty(z)+ t(u^++u^-))(u^++u^-), u^+-u^-\right)_H\\
\hspace{-3mm}&=&\hspace{-3mm}\left(B(z+ h^\infty(z)+ t(u^++u^-))u^+,
u^+\right)_H-\left(B(z+ h^\infty(z)+
t(u^++u^-))u^-, u^-\right)_H\\
\hspace{-3mm}&\ge &\hspace{-3mm} a_1\|u^+\|^2+ a_\infty\|u^-\|^2.
\end{eqnarray*}
The final inequality comes from Lemma~\ref{lem:2.5}(a) and (c).
 The condition (c) follows because
$X^+_\infty$ is dense in $H^+_\infty$.

{\bf Step 3}. For $z\in \bar B_{H^0_\infty}(\infty, R_1)$ and
$u^+\in X^+_\infty$, as above we may use the mean value theorem to
get a number $t\in (0, 1)$ such that
\begin{eqnarray*}
D_2F^\infty(z, u^+)u^+
&=&D_2F^\infty(z, u^+)u^+- D_2F^\infty(z, \theta)u^+\\
&=&(A(z+ h^\infty(z)+ u^+), u^+)_H-(A(z+ h^\infty(z)+ \theta), u^+)_H\\
&=&\left(B(z+ h^\infty(z)+ tu^+)u^+, u^+\right)_H
\ge a_1\|u^+\|^2.
\end{eqnarray*}
The final inequality is because of Lemma~\ref{lem:2.5}(a). The
condition (d) follows.
\end{proof}

{\bf [}  Note: The condition $\nu_\infty>0$ is
essentially used in the proofs of the above lemma. If $\nu_\infty=0$
the arguments before Lemma~\ref{lem:2.4} is not needed. In this case
Lemmas~\ref{lem:2.4},~\ref{lem:2.5} also hold with
$H^0_\infty=\{\theta\}$. When replaceing $F^\infty$ with ${\cal L}$
the corresponding conclusions in Lemma~\ref{lem:2.6} cannot be
proved if no further conditions are imposed  on ${\cal L}$. (See proof of
Lemma~\ref{lem:2.16}). {\bf ]}

Now $\bar B_{H^0_\infty}(\infty, R_1)$ is only locally compact, we
cannot directly apply Theorem~\ref{th:2.1} to the function
$F^\infty$. Recall that the compactness are only used in Steps 1 and
 6 of proof of \cite[Theorem~A.1]{Lu2,Lu3}. (See the proof of
more general \cite[Claim~A.3]{Lu2,Lu3}). We shall directly prove
these two steps in the present case. To this end we need the
following result.

\begin{lemma}\label{lem:2.7}
\begin{enumerate}
\item[{\rm (a)}]  Let $\{z_k\}\subset V_\infty\cap H^0_\infty$ and $\{u_k\}\subset
H^\pm_\infty$ such that $\|z_k\|\to\infty$ and that $\|u_k-u_0\|\to
0$ for some $u_0\in H$. Then
$$
F^\infty(z_k, u_k)\to\frac{1}{2}(B(\infty)u_0,
u_0)_H\quad\hbox{as}\;k\to\infty.
$$
\item[{\rm (b)}] If $\mathcal{ L}(u)=\frac{1}{2}(B(\infty)u,u)_H+
o(\|u\|^2)$ as $\|u\|\to\infty$, then
\begin{eqnarray*}
&&\frac{a_\infty}{4}\|u^+\|^2- 2\|B(\infty)\|\cdot\|u^-\|^2-
\frac{2\|B(\infty)\|^2}{a_\infty}\cdot\|h^\infty(z)\|^2
-\frac{a_\infty}{2}\|z\|^2\\
&\le&F^\infty(z, u^++u^-)\\
&\le&2\|B(\infty)\|\cdot\| u^+\|^2- \frac{a_\infty}{4}\| u^-\|^2
+\frac{a_\infty}{2}\|z\|^2+\frac{2\|B(\infty)\|^2}{a_\infty}\|h^\infty(z)\|^2.
\end{eqnarray*}
for any $(z, u^+, u^-)\in \bar B_{H^0_\infty}(\infty, R_1)\times H^+_\infty\times H^-_\infty$.
Consequently, for any given $(z_0, u^+_0)\in
\bar B_{H^0_\infty}(\infty, R_1)\times H^+_\infty$ there exists a
neighborhood $\mathcal{ U}$ of it in $\bar B_{H^0_\infty}(\infty,
R_1)\times H^+_\infty$ such that
$F^\infty(z, u^++ u^-)\to-\infty$ uniformly in $(z,
u^+)\in\mathcal{U}$
as $u^-\in H^-_\infty$ and $\|u^-\|\to\infty$.
\end{enumerate}
\end{lemma}

\begin{proof}
{\rm (a)} Since $F^\infty$ is continuous and $X^\pm_\infty$ is dense in
$H^\pm_\infty$ we can choose $\{u'_k\}\subset X^\pm_\infty$ such
that $\|u'_k-u_0\|\to 0$ and $|F^\infty(z_k, u_k)-F^\infty(z_k,
u_k')|<1/k$ for $k=1,2,\cdots$. Hence we can assume that
$\{u_k\}\subset X^\pm_\infty$ in the sequel without loss of
generality.

Note that $h^\infty(z_k)+ stu_k\in X^\pm_\infty\subset H^\pm_\infty$
and
$$
\|z_k+ h^\infty(z_k)+ stu_k\|^2=\|z_k\|^2+ \|h^\infty(z_k)+
stu_k\|^2\ge\|z_k\|^2
$$
for all $s, t\in [0, 1]$ and $k=1,\cdots$. By (${\rm D2_\infty}$),
for any $u\in H$ we have
\begin{equation}\label{e:2.13}
 \lim_{k\to\infty}\|P(z_k +
h^\infty(z_k)+ stu_k)u-P(\infty)u\|=0
\end{equation}
uniformly in $s, t\in [0, 1]$. Then the principle of uniform
boundedness implies
\begin{equation}\label{e:2.14}
M(P):=\sup\{\|P(z_k + h^\infty(z_k)+ stu_k)\|_{L(H)}\,|\,
k\in\N,\;s, t\in [0, 1]\}<\infty
\end{equation}
Moreover, by (${\rm D3_\infty}$) we have also
\begin{equation}\label{e:2.15}
 \lim_{k\to\infty}\|Q(z_k +
h^\infty(z_k)+ stu_k)-Q(\infty)\|_{L(H)}=0
\end{equation}
uniformly in $s, t\in [0, 1]$. It follows from (\ref{e:2.13}) and
(\ref{e:2.14}) that
\begin{eqnarray*}
&&|(P(z_k + h^\infty(z_k)+ stu_k)u_k, u_k)_H-(P(\infty)u_0,
u_0)_H|\\
&=&|(P(z_k + h^\infty(z_k)+ stu_k)(u_k-u_0), u_k)_H \\
&&\qquad +(P(z_k + h^\infty(z_k)+ stu_k)u_0, u_k-u_0)_H\\
 && \qquad+ (P(z_k +
h^\infty(z_k)+ stu_k)u_0, u_0)_H- (P(\infty)u_0,
u_0)_H|\\
&\le&\|P(z_k + h^\infty(z_k)+ stu_k)\|_{L(H)}\|u_k-u_0\|\cdot\|u_k\| \\
&&\qquad +\|P(z_k + h^\infty(z_k)+ stu_k)u_0\|\cdot\|u_k-u_0\|\\
 && \qquad+ |(P(z_k +
h^\infty(z_k)+ stu_k)u_0, u_0)_H- (P(\infty)u_0, u_0)_H|\to 0
\end{eqnarray*}
uniformly in $(s, t)\in [0, 1]\times [0, 1]$ as $k\to\infty$.
Similarly, from (\ref{e:2.15}) we derive that
$$
|(Q(z_k + h^\infty(z_k)+ stu_k)u_k, u_k)_H-(Q(\infty)u_0, u_0)_H|\to
0
$$
uniformly in $(s, t)\in [0, 1]\times [0, 1]$ as $k\to\infty$. Since
$(I-P^0_\infty)A(z_k+ h^\infty(z_k))=0$ for all $k$,   by the mean
value theorem we obtain
\begin{eqnarray*}
F^\infty(z_k, u_k)&=&\int^1_0D\mathcal{ L}(z_k+ h^\infty(z_k)+ tu_k)(u_k)dt\\
&=&\int^1_0(A(z_k+ h^\infty(z_k)+ tu_k), u_k)_Hdt\\
&=&\int^1_0(A(z_k+ h^\infty(z_k)+ tu_k)-A(z_k+ h^\infty(z_k)), u_k)_Hdt\\
&=&\int^1_0\int^1_0(B(z_k+ h^\infty(z_k)+ stu_k)(tu_k), u_k)_Hdsdt\\
&=&\int^1_0\int^1_0t(P(z_k+ h^\infty(z_k)+ stu_k)u_k, u_k)_Hdsdt\\
&+&\int^1_0\int^1_0t(Q(z_k+ h^\infty(z_k)+ stu_k)u_k, u_k)_Hdsdt\\
&\to & \int^1_0\int^1_0t(P(\infty)u_0, u_0)_Hdsdt
+\int^1_0\int^1_0t(Q(\infty)u_0, u_0)_Hdsdt\\
&=&\int^1_0\int^1_0t(B(\infty)u_0,
u_0)_Hdsdt\\
&=&\frac{1}{2}(B(\infty)u_0, u_0)_H\hspace{10mm}\hbox{as
$k\to\infty$.}
\end{eqnarray*}

{\rm (b)} Since $a_\infty\le\|B(\infty)\|$ and
$$
\|B(\infty)\|\cdot\|h^\infty(z)\|\cdot\|u^+
+u^-\|\le\frac{\|B(\infty)\|^2}{2a_\infty}\|h^\infty(z)\|^2+
\frac{a_\infty}{2}\|u^+\|^2+ \frac{a_\infty}{2}\|u^-\|^2,
$$
from  (\ref{e:2.12}) and (\ref{e:1.1}) we derive
\begin{eqnarray*}
&&\mathcal{L}(z+ h^\infty(z)+ u^++u^-)\\
 &\le&\frac{1}{2}\bigl(B(\infty)(h^\infty(z)+ u^++u^-),
h^\infty(z)+
u^++u^-\bigr)_H\\
&+& \frac{a_\infty}{8}\|z+ h^\infty(z)+ u^++u^-\|^2\\
&= &\frac{1}{2}\bigl(B(\infty) u^+,
u^+\bigr)_H+\frac{1}{2}\bigl(B(\infty) u^-, u^-\bigr)_H\\
&+& \bigl(B(\infty)h^\infty(z), u^+
+u^-\bigr)_H+ \frac{a_\infty}{8}\|z+ h^\infty(z)+ u^++u^-\|^2\\
&\le &\frac{1}{2}\|B(\infty)\|\cdot\| u^+\|^2- a_\infty\| u^-\|^2+
\|B(\infty)\|\cdot\|h^\infty(z)\|\cdot\|u^+
+u^-\|\\
&+& \frac{a_\infty}{4}\|z\|^2+ \frac{a_\infty}{4}\|h^\infty(z)\|^2+
\frac{a_\infty}{4}\|u^+\|^2+ \frac{a_\infty}{4}\|u^-\|^2\\
&\le&2\|B(\infty)\|\cdot\| u^+\|^2- \frac{a_\infty}{4}\| u^-\|^2
+\frac{a_\infty}{4}\|z\|^2+\frac{\|B(\infty)\|^2}{a_\infty}\|h^\infty(z)\|^2.
\end{eqnarray*}
Similarly, we have
\begin{eqnarray*}
&&\mathcal{L}(z+ h^\infty(z)+ u^++u^-)\\
 &\ge&\frac{1}{2}\bigl(B(\infty)(h^\infty(z)+ u^++u^-),
h^\infty(z)+
u^++u^-\bigr)_H\\
&-& \frac{a_\infty}{8}\|z+ h^\infty(z)+ u^++u^-\|^2\\
&=&\frac{1}{2}\bigl(B(\infty) u^+,  u^+\bigr)_H+
\frac{1}{2}\bigl(B(\infty) u^-, u^-\bigr)_H \\
&+& \bigl(B(\infty)h^\infty(z), u^++u^-\bigr)_H- \frac{a_\infty}{8}\|z+ h^\infty(z)+ u^++u^-\|^2\\
&\ge &a_\infty\|u^+\|^2- \frac{1}{2}\|B(\infty)\|\cdot\|
u^-\|^2-\frac{\|B(\infty)\|^2}{2a_\infty}\|h^\infty(z)\|^2-
\frac{a_\infty}{2}\|u^+\|^2- \frac{a_\infty}{2}\|u^-\|^2\\
&-& \frac{a_\infty}{4}\|z\|^2- \frac{a_\infty}{4}\|h^\infty(z)\|^2-
\frac{a_\infty}{4}\|u^+\|^2- \frac{a_\infty}{4}\|u^-\|^2\\
&\ge&\frac{a_\infty}{4}\|u^+\|^2- 2\|B(\infty)\|\cdot\|u^-\|^2-
\frac{\|B(\infty)\|^2}{a_\infty}\cdot\|h^\infty(z)\|^2
-\frac{a_\infty}{4}\|z\|^2.
\end{eqnarray*}
Hence
\begin{eqnarray*}
&&\frac{a_\infty}{4}\|u^+\|^2- 2\|B(\infty)\|\cdot\|u^-\|^2-
\frac{\|B(\infty)\|^2}{a_\infty}\cdot\|h^\infty(z)\|^2
-\frac{a_\infty}{4}\|z\|^2\\
&\le&\mathcal{L}(z+ h^\infty(z)+ u^++u^-)\\
&\le&2\|B(\infty)\|\cdot\| u^+\|^2- \frac{a_\infty}{4}\| u^-\|^2
+\frac{a_\infty}{4}\|z\|^2+\frac{\|B(\infty)\|^2}{a_\infty}\|h^\infty(z)\|^2.
\end{eqnarray*}
In particular, we have
\begin{eqnarray*}
-\frac{\|B(\infty)\|^2}{a_\infty}\cdot\|h^\infty(z)\|^2
-\frac{a_\infty}{4}\|z\|^2 \le\mathcal{L}(z+ h^\infty(z)) \le
\frac{a_\infty}{4}\|z\|^2+\frac{\|B(\infty)\|^2}{a_\infty}\|h^\infty(z)\|^2.
\end{eqnarray*}
Since $F^\infty(z, u^++ u^-)=\mathcal{ L}(z+ h^\infty(z)+
u^++u^-)-\mathcal{L}(z+ h^\infty(z))$  by (\ref{e:2.10}), the
desired inequalities easily follow.
 \end{proof}

For $F^\infty$  we can directly prove the corresponding conclusions
with Step 1 in the proof of Theorem~\ref{th:2.1} (given in Appendix
A of \cite{Lu2},\cite{Lu3}) as follows.

\begin{lemma}\label{lem:2.8}
\begin{enumerate}
\item[{\rm (a)}] For any  $r\in (0, \infty)$ there exists a number
 $\varepsilon_r\in (0, r)$ such that
for each $(z, u)\in\bar B_{H^0_\infty}(\infty, R_1)\times \bar
B_{H^+_\infty}(\theta, \varepsilon_r)$ there exists a unique point
$\varphi_z(u)\in B_{H^-_\infty}(\theta, r)$ satisfying
$$
F^\infty(z, u+\varphi_z(u))=\max\{F^\infty(z, u+ v)\,|\, v\in
B_{H^-_\infty}(\theta, r)\}.
$$
One has also $\varphi_z(\theta)=\theta$.
 \item[{\rm (b)}] If $\mathcal{
L}(u)=\frac{1}{2}(B(\infty)u,u)_H+ o(\|u\|^2)$ as $\|u\|\to\infty$,
 for each $(z, u)\in\bar B_{H^0_\infty}(\infty, R_1)\times
H^+_\infty$ there exists a unique point $\varphi_z(u)\in
 H^-_\infty$ such that
$$
F^\infty(z, u+\varphi_z(u))=\max\{F^\infty(z, u+ v)\,|\, v\in
 H^-_\infty\}.
$$
Moreover, $\varphi_z(\theta)=\theta$, and
\begin{eqnarray*}
\| \varphi_z(u^+)\|^2\le \frac{8}{a_\infty}\|B(\infty)\|\cdot\| u^+\|^2
+ 4\|z\|^2+\frac{16\|B(\infty)\|^2}{a^2_\infty}\|h^\infty(z)\|^2
\end{eqnarray*}
Clearly, Lemma~\ref{lem:2.7}(b) implies that for any bounded subset $K\subset\bar B_{H^0_\infty}(R_1,\infty)$,
 $$
 F^\infty(z, u+\varphi_z(u))\ge
F^\infty(z,u)\to\infty\quad\hbox{uniformly in}\;z\in K
$$
as $u\in H^+_\infty$ and $\|u\|\to\infty$.
\end{enumerate}
\end{lemma}

Later on  we shall understand $r=\infty$ and
$\varepsilon_\infty=\infty$ for conveniences in case (ii). Note
that the cases (a) and (b) of Lemma~\ref{lem:2.8} correspond to
Theorems~\ref{th:1.3} and \ref{th:1.1}, respectively. Moreover, if
Lemmas~\ref{lem:2.4}--\ref{lem:2.6} only hold in a set $\bar
B_{H^0_\infty}(\infty, R')\oplus (\bar B_H(\theta, r')\cap
X^\pm_\infty)$, then $z$ and $r$ in (a) are restricted in $\bar
B_{H^0_\infty}(\infty, R')$ and $(0, r')$, respectively.

\noindent{\bf Proof of Lemma~\ref{lem:2.8}}. As at the beginning of
proof of Theorem~\ref{th:2.1} (given in Appendix A of
\cite{Lu2},\cite{Lu3}) we only need to consider the case $\dim
H^-_\infty>0$.

(a)  Since the function $H^-_\infty\to\R,\;u^-\mapsto F^\infty(z,
u^++ u^-)$ is strictly concave  by Lemma~\ref{lem:2.6}(b), it has
a unique maximum point on a convex set if existing. Clearly, it
attains the maximum on the compact subset $\bar
B_{H^-_\infty}(\theta,r)$. Suppose by contradiction that there exist
sequences $\{(z_n, x_n)\}\in\bar B_{H^0_\infty}(\infty, R_1)\times
\bar B_{H^+_\infty}(\theta, r)$
 with $x_n\to 0$,
and $\{v_n\}\subset\partial \bar B_{H^-_\infty}(\theta, r)$ such
that
\begin{equation}\label{e:2.16}
F^\infty(z_n, x_n+ v_n)> F^\infty(z_n, x_n+ u)\quad\hbox{for all}\; u\in
B_{H^-_\infty}(\theta, r),\;\hbox{for all}\; n\in\N.
\end{equation}

  If $\{z_n\}$ is bounded  we may assume up to subsequences that $z_n\to z_0\in
\bar B_{H^0_\infty}(\infty, R_1)$ and $v_n\to v_0\in\partial\bar
B_{H^-_\infty}(\theta, r)$   since both $\bar B_{H^0_\infty}(\infty,
R_1)$ and $\partial\bar B_{H^-_\infty}(\theta, r)$ are compact. It follows from
these and (\ref{e:2.16}) that
$$
F^\infty(z_0,  v_0)\ge F^\infty(z_0,  u)\quad\hbox{for all}\; u\in
B_{H^-_\infty}(\theta, r).
$$
On the other hand,  the mean value theorem yields a number $s\in (0,1)$ such
that
\begin{eqnarray*}
F^\infty(z_0, v_0)&=&F^\infty(z_0, v_0)-F^\infty(z_0,
\theta)=D_2F(z_0, sv_0)v_0\\
&=&\frac{1}{s}[D_2F(z_0, sv_0)(sv_0)-D_2F(z_0, \theta)(sv_0)]\\
&\le&-\frac{a_\infty}{s}\|sv_0\|^2=-sa_\infty\|v_0\|^2<0=F^\infty(z_0,
\theta)
\end{eqnarray*}
 by Lemma~\ref{lem:2.6}(a)--(b). A contradiction is obtained in this case.

Up to  subsequences we assume that $\|z_n\|\to\infty$ and $v_n\to
v_0\in
\partial \bar B_{H^-_\infty}(\theta, r)$ in $H$. Then $H^\pm_\infty\ni u_n:=x_n+ v_n\to v_0$.
By Lemma~\ref{lem:2.7} we get
\begin{eqnarray*}
F^\infty(z_n, x_n+ v_n)\to\frac{1}{2}(B(\infty)v_0,
v_0)_H<0,\quad
 F^\infty(z_n, x_n)\to\frac{1}{2}(B(\infty)\theta, \theta)_H=0.
\end{eqnarray*}
Hence (\ref{e:2.16}) leads to $(B(\infty)v_0, v_0)_H\ge 0$, and
therefore  a contradiction is obtained again.

To see $\varphi_z(\theta)=\theta$, note that  $D_2F^\infty(z,
\varphi_z(\theta))=0$. If $\varphi_z(\theta)\ne\theta$ then
$$
0=[D_2F^\infty(z, \varphi_z(\theta))-D_2F^\infty(z,
\theta)](\varphi_z(\theta)-\theta)\le
-a_\infty\|\varphi_z(\theta)\|^2<0
$$
by Lemma~\ref{lem:2.6}(b), which is a contradiction.

(b) By Lemma~\ref{lem:2.6}(b)  the function
$H^-_\infty\to\R,\;u^-\mapsto  -F^\infty(z, u^++ u^-)$ is strictly
convex. The second claim of Lemma~\ref{lem:2.7} also shows that this
function is coercive. Hence it attains the minimum at some point
$\varphi_z(u^+)\in H^-_\infty$. That is, the function
$H^-_\infty\to\R,\;u^-\mapsto F^\infty(z, u^++u^-)$ takes the
maximum at $\varphi_z(u^+)$. As in the proof of Lemma~ 2.1 of
\cite{DHK} the uniqueness of $\varphi_z(u^+)$ follows from
Lemma~\ref{lem:2.6}(b) as well.

The proof that $\varphi_z(\theta)=\theta$ may be obtained as above.
To see the another claim, by Lemma~\ref{lem:2.7}(b),
\begin{eqnarray*}
&&2\|B(\infty)\|\cdot\| u^+\|^2- \frac{a_\infty}{4}\| \varphi_z(u^+)\|^2
+\frac{a_\infty}{2}\|z\|^2+\frac{2\|B(\infty)\|^2}{a_\infty}\|h^\infty(z)\|^2\\
&\ge&F^\infty(z, u^++ \varphi_z(u^+))\ge F^\infty(z, u^+)\\
&\ge&\frac{a_\infty}{4}\|u^+\|^2-
\frac{2\|B(\infty)\|^2}{a_\infty}\cdot\|h^\infty(z)\|^2
-\frac{a_\infty}{2}\|z\|^2.
\end{eqnarray*}
The conclusion follows immediately.
\hfill$\Box$\vspace{2mm}

\begin{remark}\label{rm:2.9}
{\rm Note that a local maximum of a concave function (with finite
values) on a normed linear space is also a global maximum.  From
Lemma~\ref{lem:2.8}(a) it follows that for any $r>0$ there exists
a number $\varepsilon_r\in (0, r)$ such that
for each $(z, u)\in\bar B_{H^0_\infty}(\infty, R_1)\times \bar
B_{H^+_\infty}(\theta, \varepsilon_r)$ there exists a unique point
$\varphi_z(u)\in B_{H^-_\infty}(\theta, r)$ satisfying
\begin{eqnarray}\label{e:2.17}
F^\infty(z, u+\varphi_z(u))&=&\max\{F^\infty(z, u+ v)\,|\, v\in
B_{H^-_\infty}(\theta, r)\}\nonumber\\
&=&\max\{F^\infty(z, u+ v)\,|\, v\in H^-_\infty\}.
\end{eqnarray}
Define
\begin{eqnarray}\label{e:2.17+}
r_\mathcal{ L}:=\sup\{\varepsilon_r\,|\, r>0\}.
\end{eqnarray}
Then for each $(z, u)\in\bar B_{H^0_\infty}(\infty, R_1)\times
B_{H^+_\infty}(\theta, r_\mathcal{ L})$ there exists a unique point
$\varphi_z(u)\in H^-_\infty$ with $\varphi_z(\theta^+)=\theta^-$,
such that
$$
F^\infty(z, u+\varphi_z(u))=\max\{F^\infty(z, u+ v)\,|\, v\in
 H^-_\infty\}.
$$
Clearly, under the assumption  (\ref{e:1.5}), i.e. $\mathcal{
L}(u)=\frac{1}{2}(B(\infty)u,u)_H+ o(\|u\|^2)$ as $\|u\|\to\infty$,
 we have $r_{\cal L}=\infty$   by Lemma~\ref{lem:2.8}(b)
(because $\varepsilon_\infty=\infty$). {\bf [}  Note: if
Lemmas~\ref{lem:2.4}-\ref{lem:2.6} only hold in a set $\bar
B_{H^0_\infty}(\infty, R')\oplus (\bar B_H(\theta, r')\cap
X^\pm_\infty)$, we define $r_\mathcal{ L}:=\sup\{\varepsilon_r\,|\,
0<r<r'\}$. Then for each $(z, u)\in\bar B_{H^0_\infty}(\infty,
R')\times B_{H^+_\infty}(\theta, r_\mathcal{ L})$ there exists a
unique $\varphi_z(u)\in B_{H^-_\infty}(\theta,r')$ with
$\varphi_z(\theta^+)=\theta^-$, such that $F^\infty(z,
u+\varphi_z(u))=\max\{F^\infty(z, u+ v)\,|\, v\in
 B_{H^-_\infty}(\theta,r')\}$. In this case the following map $j$ is only defined on
$\bar B_{H^0_\infty}(\infty, R')\times B_{H^+_\infty}(\theta,
r_\mathcal{ L})$. {\bf ]}}
\end{remark}

It is easily seen that  the following map
\begin{eqnarray}\label{e:2.18}
j: \bar B_{H^0_\infty}(\infty, R_1)\times B_{H^+_\infty}(\theta, r_\mathcal{ L})\to\R,
 \quad (z,u)\mapsto F^\infty(z,
u+\varphi_z(u)),
\end{eqnarray}
 is well-defined.

\begin{lemma}\label{lem:2.10}
The map $j$ is continuous, and for every $z\in \bar
B_{H^0_\infty}(\infty, R_1)$  the map
$$B_{H^+_\infty}(\theta, r_\mathcal{ L})\to\R,\; u\mapsto j(z, u)$$
is
continuously directional differentiable.
\end{lemma}

\begin{proof}
Clearly, it suffices to prove that the restriction of $j$ to $\bar
B_{H^0_\infty}(\infty, R_1)\times B_{H^+_\infty}(\theta,
\varepsilon_r)$ is continuously directional differentiable.

If $r<\infty$, since $\bar B_{H}(\infty, R_1)\cap\bar B_H(\theta,
R)\cap H^0_\infty$ is compact for any $R>R_1$, as in Step 3 of the
proof of Theorem~\ref{th:2.1} (given in Appendix A of
\cite{Lu2,Lu3}) we can get the desired conclusion from Lemma~2.3 of
\cite{DHK}.

If $r=\infty$, i.e. (\ref{e:1.5}) holds,  for any  $(z_0, u^+_0)\in
\bar B_{H^0_\infty}(\infty, R_1)\times H^+_\infty$, by
Lemma~\ref{lem:2.8}(b) there exists a bounded neighborhood
$\mathcal{ U}$ of it in $\bar B_{H^0_\infty}(\infty, R_1)\times
H^+_\infty$ and a positive number $R$ such that $\varphi_z(u)\in
B_{H^0_\infty}(\theta, R)$ for all $(z,u)\in\mathcal{U}$. Suppose
that $\{(z_n, u^+_n)\}$ converges to $(z_0, u^+_0)$. As in Step 2 of
the proof of Theorem~\ref{th:2.1} (given in Appendix A of
\cite{Lu2,Lu3}) it is easily proved that $\varphi_{z_n}(u^+_n)\to
\varphi_{z_0}(u^+_0)$ as $n\to\infty$. Hence $j$ is continuous in
this case. The second claim follows from Lemma~2.3 of \cite{DHK}.
\end{proof}

By (\ref{e:2.17}), for $(z, u)\in \bar B_{H^0_\infty}(\infty,
R_1)\times B_{H^+_\infty}(\theta, r_\mathcal{ L})$ we have
\begin{equation}\label{e:2.19}
F^\infty(z, u+ \varphi_z(u))\ge F^\infty(z, u+ v)\quad\hbox{for all}\; v\in
H^-_\infty.
\end{equation}
Moreover, for any $z\in \bar B_{H^0_\infty}(\infty, R_1)$ we have
also
\begin{eqnarray}
&&F^\infty(z, u)\ge\frac{a_1}{4}\|u\|^2\qquad\hbox{for all}\; u\in H^+_\infty,\label{e:2.20}\\
&&F^\infty(z, v)\le -\frac{a_\infty}{4}\|v\|^2\qquad\hbox{for all}\; v\in
H^-_\infty.\label{e:2.21}
\end{eqnarray}
In fact,  using the mean value theorem and Lemma~\ref{lem:2.6}(d)
we get
\begin{eqnarray*}
F^\infty(z, u)&=&F^\infty(z, u)-F^\infty(z, \theta)=D_2F^\infty(z,
su)(u)\nonumber\\
&=&\frac{1}{s}D_2F^\infty(z, su)(su)\ge a_1s\|u\|^2\ge 0
\end{eqnarray*}
for some $s\in (0, 1)$. If $u\ne\theta$, the same reason yields a number
$s_u\in (1/2, 1)$ such that
\begin{eqnarray*}
F^\infty(z, u)>F^\infty(z, u)-F^\infty(z, u/2) = D_2F^\infty(z,
s_uu)(u/2)\ge \frac{a_1}{4}\|u\|^2.
\end{eqnarray*}

Similarly,  we get a number $s\in (0, 1)$
such that
\begin{eqnarray*}
F^\infty(z, v)&=&F^\infty(z, v)-F^\infty(z, \theta)=D_2F^\infty(z,
sv)(v)\nonumber\\
&=&\frac{1}{s}D_2F^\infty(z, s v)(s v)\le -a_\infty s\|v\|^2\le 0
\end{eqnarray*}
by Lemma~\ref{lem:2.6}(c). Moreover, if $v\ne\theta$ we have also a number $s_v\in (1/2, 1)$ such that
\begin{eqnarray*}
F^\infty(z, v)<F^\infty(z, v)-F^\infty(z, v/2) = D_2F^\infty(z,
s_vv)(v/2)\le -\frac{a_\infty}{4}\|v\|^2.
\end{eqnarray*}
For $r\in (0, \infty]$, $z\in \bar B_{H^0_\infty}(\infty, R_1)$ and
$(u,v)\in \bar B_{H^+_\infty}(\theta, \varepsilon_r)\times
B_{H^-_\infty}(\theta, r)$,  define
\begin{eqnarray*}
&&\psi_1(z, u+ v)=\left\{\begin{array}{ll}
 \frac{\sqrt{F^\infty(z, u+ \varphi_z(u))}}{\|u\|}u &\;\hbox{if}\;u\ne \theta,\\
 \theta&\;\hbox{if}\;u=\theta,
 \end{array}\right.\nonumber\\
&&\psi_2(z, u+ v)=\left\{\begin{array}{ll}
 \frac{\sqrt{F^\infty(z, u+ \varphi_z(u))-F^\infty(z, u+ v)}}{\|v-\varphi_z(u)\|}(v-\varphi_z(u))
  &\;\hbox{if}\;v\ne\varphi_z(u),\\
 \theta&\;\hbox{if}\;v=\varphi_z(u).
 \end{array}\right.\nonumber
\end{eqnarray*}
By Lemma~\ref{lem:2.10}, the map
\begin{equation}\label{e:2.22}
\psi:\bar B_{H^0_\infty}(\infty, R_1)\times \bigl(\bar
B_{H^+_\infty}(\theta, \varepsilon_r)\oplus B_{H^-_\infty}(\theta,
r)\bigr)\to H^\pm_\infty
\end{equation}
given by $\psi(z, u+ v)=\psi_1(z, u+ v)+\psi_2(z, u+ v)$, is
continuous. Clearly,
$$
\hbox{ $\psi(z, u+ v)\in {\rm Im}(\psi)\cap H^-_\infty$ if and only
if $u=\theta$, and}
$$
$$
F^\infty(z, u+ v)=\|\psi_1(z, u+
v)\|^2-\|\psi_2(z, u+ v)\|^2.
$$

As in Step 5 in the proof of Theorem~\ref{th:2.1} (given in Appendix
A of \cite{Lu2,Lu3}) we can prove

\begin{lemma}\label{lem:2.11}
For each $z\in \bar B_{H^0_\infty}(\infty, R_1)$ the map
$$
\psi(z,\cdot): \bar B_{H^+_\infty}(\theta, \varepsilon_r)\oplus
B_{H^-_\infty}(\theta, r)\to H^\pm_\infty
$$
is injective whether $r$ is finite or infinite.
\end{lemma}

{\bf [}  Note: If Lemmas~\ref{lem:2.4}-\ref{lem:2.6} only
hold in a set $\bar B_{H^0_\infty}(\infty, R')\oplus (\bar
B_H(\theta, r')\cap X^\pm_\infty)$, we require $z$ and $r$ in this
lemma and the following Lemma~\ref{lem:2.12}(i) to sit in  $\bar
B_{H^0_\infty}(\infty, R')$ and $(0, r')$, respectively. {\bf ]}

Now we are a position to prove the corresponding conclusions with
 Step 6 in the proof of Theorem~\ref{th:2.1} (given in Appendix A of \cite{Lu2,Lu3}).

\begin{lemma}\label{lem:2.12}
\begin{enumerate}
\item[{\rm (a)}] For any $r\in (0, \infty)$  there is a number $\epsilon_r\in
(0, \varepsilon_r/4)$ such that
$$
B_{H^+_\infty}\bigl(\theta, \sqrt{a_1}\epsilon_r\bigr)\oplus
B_{H^-_\infty}\bigl(\theta,
\sqrt{a_1}\epsilon_r\bigr)\subset\psi\bigl(z, B_{H^+_\infty}(\theta,
2\epsilon_r)\oplus B_{H^-_\infty}(\theta, r)\bigr)
$$
for any $z\in \bar B_{H^0_\infty}(\infty, R_1)$.

\item[{\rm (b)}] If   $\mathcal{
L}(u)=\frac{1}{2}(B(\infty)u,u)_H+ o(\|u\|^2)$ as $\|u\|\to\infty$,
that is, $r=\infty$,  then for each $z\in\bar B_{H^0_\infty}(\infty,
R_1)$ the map
$$
\psi(z,\cdot):H^+_\infty\oplus H^-_\infty\to H^+_\infty\oplus
H^-_\infty
$$
is surjective, and hence bijective due to Lemma~\ref{lem:2.11}. As a
consequence we get
$$
\psi^{-1}(H^+_\infty\oplus H^-_\infty)=\bar B_{H^0_\infty}(\infty,
R_1)\times(H^+_\infty\oplus H^-_\infty).
$$
\end{enumerate}
\end{lemma}

\begin{proof}
(a)  By (\ref{e:2.21}) there exists a number $C>0$ such that
\begin{equation}\label{e:2.23}
F^\infty(z, v)<-C\quad\forall (z, v)\in \bar B_{H^0_\infty}(\infty,
R_1)\times\partial B_{H^-_\infty}(\theta, r).
\end{equation}

{\bf Claim 2.12.1.} There exists a number $\epsilon_r\in
(0, \varepsilon_r/4)$ such that
\begin{equation}\label{e:2.24}
F^\infty(z, u+ v)\le 0
\end{equation}
for any $(z, u, v)\in\bar B_{H^0_\infty}(\infty, R_1)\times \bar
B_{H^+_\infty}(\theta, 2\epsilon_r)\times\partial
B_{H^-_\infty}(\theta, r)$.

Suppose by contradiction that there exists a sequence
$$
\{(z_n, u_n, v_n)\}\subset\bar B_{H^0_\infty}(\infty, R_1)\times
\bar B_{H^+_\infty}(\theta, \varepsilon_r)\times\partial
B_{H^-_\infty}(\theta, r)
$$
such that $u_n\to \theta$ and $F^\infty(z_n, u_n+ v_n)\ge 0$ for all
$n$. If $\{z_n\}$ has a bounded subsequence we can get a
contradiction as in Step 6 of proof of Theorem~\ref{th:2.1} (given
in Appendix A of \cite{Lu2,Lu3}). Otherwise, after passing to a
subsequence we may assume that $\|z_n\|\to\infty$ and $v_n\to v_0$.
Then using  Lemma~\ref{lem:2.7}(a) we derive
$$
F(z_k, u_k+ v_k)\to\frac{1}{2}(B(\infty)v_0,
v_0)_H<0\quad\hbox{as}\;k\to\infty.
$$
This leads to a contradiction again.  (\ref{e:2.24}) is proved.

{\bf Claim 2.12.2.} One can shrink the positive number $\epsilon_r$ in
(\ref{e:2.24}) such that
\begin{equation}\label{e:2.25}
\varphi_z(\bar B_{H^+_\infty}(\theta, 2\epsilon_r))\subset
B_{H^-_\infty}(\theta, r/2)\quad\forall z\in \bar
B_{H^0_\infty}(\infty, R_1).
\end{equation}

By a contradiction suppose that
 there exist sequences
$\{z_n\}\subset \bar B_{H^0_\infty}(\infty, R_1)$ and
$\{u_n\}\subset \bar B_{H^+_\infty}(\theta, \varepsilon_r)$ such
that
$$
\|u_n\|\to
0\quad\hbox{and}\;\varphi_{z_n}(u_n)\notin B_{H^-_\infty}(\theta,
r/2)\;\forall n=1,2,\cdots.
$$
By Lemma~\ref{lem:2.8} each $\varphi_{z_n}(u_n)$ is a unique point
in $B_{H^-_\infty}(\theta, r)$ such that
$$
F^\infty(z_n, u_n+\varphi_{z_n}(u_n))=\max\bigl\{F^\infty(z_n, u_n+
v)\,|\, v\in B_{H^-_\infty}(\theta, r)\bigr\}.
$$
Since $\bar B_{H^-_\infty}(\theta, r)$ is compact, after passing a
subsequence (if necessary) we may assume $\varphi_{z_n}(u_n)\to
v_0\in \bar B_{H^-_\infty}(\theta, r)\setminus
B_{H^-_\infty}(\theta, r/2)$.

$\bullet$ If $\{z_n\}$ has a bounded subsequence, passing to a subsequence we may assume
$z_n\to z_0\in \bar B_{H^0_\infty}(\infty, R_1)$. Then  by (\ref{e:2.21}) we get
$$
F^\infty(z_n, u_n+\varphi_{z_n}(u_n))\to F^\infty(z_0, v_0)\le -\frac{a_\infty}{4}\|v_0\|^2\le -\frac{r^2a_\infty}{16}<0
$$
as $n\to\infty$, and  $F^\infty(z_n, u_n)\to F^\infty(z_0, \theta)=0$ as $n\to\infty$.
This contradicts to the fact that
$F^\infty(z_n, u_n)\le F^\infty(z_n, u_n+\varphi_{z_n}(u_n))$ for all $n$.

$\bullet$ If $\{z_n\}$ has no bounded subsequences, passing to a subsequence we may assume
$\|z_n\|\to \infty$. In this case  Lemma~\ref{lem:2.7}(a) leads to
$$
F^\infty(z_n, u_n+ \varphi_{z_n}(u_n))\to\frac{1}{2}(B(\infty)v_0,
v_0)_H\le -a_\infty\|v_0\|^2\le -\frac{a_\infty r^2}{4}
$$
as $n\to\infty$, and  $F^\infty(z_n, u_n)\to
\frac{1}{2}(B(\infty)\theta, \theta)_H=0$ as $n\to\infty$.  This also yields a contradiction to the fact that
$F^\infty(z_n, u_n+ \varphi_{z_n}(u_n))\ge F^\infty(z_n, u_n)$ for all $n$.

Claim 2.12.2 is proved.

For $(z, u)\in \bar B_{H^0_\infty}(\infty, R_1)\times
B_{H^+_\infty}(\theta, 2\epsilon_r)$, by (\ref{e:2.19}) and
(\ref{e:2.20}) we get
\begin{eqnarray}\label{e:2.26}
F^\infty(z, u+ \varphi_z(u))\ge F^\infty(z, u)\ge \frac{a_1}{4}\|u\|^2.
\end{eqnarray}
This and (\ref{e:2.24}) imply that
\begin{eqnarray}\label{e:2.27}
F^\infty(z, u+\varphi_z(u))-F^\infty(z, u+ v)> a_1\epsilon_r^2
\end{eqnarray}
for any $(z, u, v)\in \bar B_{H^0_\infty}(\infty, R_1)\times\partial
B_{H^+_\infty}(\theta, 2\epsilon_r)\times\partial
B_{H^-_\infty}(\theta, r)$.

Note that (\ref{e:2.23})--(\ref{e:2.27}) correspond to (A.2), (A.3), (A.4),
(A.5) and (A.6) in  Step 6 in the proof of Theorem~\ref{th:2.1}
(given in Appendix A of \cite{Lu2,Lu3}), respectively. Using these
and repeating the remained arguments therein (i.e.,  Step 6 in the
proof of Theorem~\ref{th:2.1} given in Appendix A of \cite{Lu2,Lu3})
we may get
$$
\bar
B_{H^+_\infty}(\theta,\sqrt{a_1}\epsilon_r)\subset\psi_1\bigl(z,
B_{H^+_\infty}(\theta, 2\epsilon_r)\bigr)
$$
and the desired conclusion (a).

(b) For any given $(\bar u^+, \bar u^-)\in H^+_\infty\times
H^-_\infty$, without loss of generality, we assume $(\bar u^+, \bar
u^-)\ne(\theta, \theta)$ because $\psi(z,\theta)=\theta$.

$\bullet$ If $\bar u^+=\theta$ then $\bar u^-\ne\theta$.  Since
(\ref{e:2.17}) and  Lemma~\ref{lem:2.7}(b) imply
$$
0=F^\infty(z, \varphi_z(\theta)\ge F^\infty(z,
u)\to-\infty\quad\hbox{as}\quad u\in
H^-_\infty\;\hbox{and}\;\|u\|\to\infty,
$$
 the intermediate value theorem gives a number $t>0$ such that
$-F^\infty(z, t\bar u^-)=\|\bar u^-\|^2$. Set $u^-:=t\bar u^-$. Then
$\psi_1(z, \theta+ u^-)=\psi_1(z, \theta)=\theta$ and
$$
\psi_2(z, \theta+ u^-)=\frac{\sqrt{F^\infty(z,
\varphi_z(\theta))-F^\infty(z,
u^-)}}{\|u^--\varphi_z(\theta)\|}(u^--\varphi_z(\theta)=\bar u^-.
$$
Namely, $\psi(z, \theta+ u^-)=(\theta, \bar u^-)$.

$\bullet$ Let $\bar u^+\ne\theta$. By Lemma~\ref{lem:2.8}(b),
$\varphi_z(\theta)=\theta$ and $F^\infty(z,
u+\varphi_z(u))\to\infty$ as $u\in H^+_\infty$ and $\|u\|\to\infty$.
Lemma~\ref{lem:2.10} also tells us that $H^+_\infty\ni u\mapsto
F^\infty(z, u+\varphi_z(u))$ is continuous. By the intermediate
value theorem we have a number $t>0$ such that
$$
F^\infty(z, t\bar u^++\varphi_z(t\bar u^+))=\|\bar u^+\|^2.
$$
Set $u^+:=t\bar u^+$. Then $\psi_1(z, u^++ v)=\bar u^+$ for any
$v\in H^-_\infty$.  If $\bar u^-=\theta$, then
$$
\psi_2(z, u^++ u^-)=\theta=\bar u^-\quad\hbox{for}\quad
u^-=\varphi_z(u^+).
$$
If  $\bar u^-\ne\theta$, we define a function $g:[0, \infty)\to\R$
by
$$
g(s)= F^\infty(z, u^+ + \varphi_z(u^+))-F^\infty(z, u^+ +
\varphi_z(u^+)+ s\bar u^-).
$$
Then $g(s)\ge 0$, $g(0)=0$ and $g(s)\to \infty$ as $s\to\infty$ by
 Lemma~\ref{lem:2.7}(b). Using the intermediate
value theorem may yield a number $s_0>0$ such that $g(s_0)=\|\bar u^-\|^2$.
Hence for $u^-:=\varphi_z(u^+)+ s_0\bar u^-\in H^-_\infty$ we get
$$
\frac{\sqrt{F^\infty(z, u^+ + \varphi_z(u^+))-F^\infty(z, u^++
u^-)}}{\|u^--\varphi_z(u^+)\|}(u^--\varphi_z(u^+))=\bar u^-
$$
 This shows
$\psi(z, u^++u^-)=(\bar u^+, \bar u^-)$.

 Summarizing the above arguments we have proved that the map $
 \psi(z,\cdot)$ is surjective. The other conclusions of (ii) easily follows.
\end{proof}

The cases (a) and (b) of Lemma~\ref{lem:2.12} correspond to
Theorems~\ref{th:1.3} and \ref{th:1.1}, respectively. If
Lemmas~\ref{lem:2.4}-\ref{lem:2.6} only hold in a set $\bar
B_{H^0_\infty}(\infty, R')\oplus (\bar B_H(\theta, r')\cap
X^\pm_\infty)$, we require $z$ and $r$ in Lemma~\ref{lem:2.12}(a)  to sit in  $\bar
B_{H^0_\infty}(\infty, R')$ and $(0, r')$, respectively.

 The following
two lemmas give the corresponding conclusions with
 Step 7 of the proof of Theorem~\ref{th:2.1} (given in Appendix A of \cite{Lu2,Lu3}) in the cases of
Theorems~\ref{th:1.1} and \ref{th:1.3}, respectively. \

\begin{lemma}\label{lem:2.13}
 Let  $\mathcal{
L}(u)=\frac{1}{2}(B(\infty)u,u)_H+ o(\|u\|^2)$ as $\|u\|\to\infty$.
(That is, $r=\infty$). By Lemma~\ref{lem:2.11} and
Lemma~\ref{lem:2.12}(ii) we have a bijection
\begin{eqnarray*}
&&\bar B_{H^0_\infty}(\infty, R_1)\times( H^+_\infty\oplus
H^-_\infty)\to \bar B_{H^0_\infty}(\infty, R_1)\times(
H^+_\infty\oplus
H^-_\infty),\nonumber\\
&&\hspace{20mm} (z, u+ v)\mapsto (z, \psi(z, u+v)).
\end{eqnarray*}
Its inverse, denoted by $\phi$, has a form
$$
\phi(z, u+ v)= (z, \phi_z(u+v)):=(z, u'+ v'),
$$
where $(u', v')\in H^+_\infty\times H^-_\infty$ is a unique point
satisfying $u+ v=\psi(z, u'+v')$. Then $\phi$ is a homeomorphism and
$$
F^\infty(\phi(z, u+ v))=\|u\|^2-\|v\|^2
$$
for any $(z, u, v)\in\bar B_{H^0_\infty}(\infty, R_1)\times
H^+_\infty\times H^-_\infty$. In particular, for each $z\in\bar B_{H^0_\infty}(\infty, R_1)$, $\phi_z$
(and so $\psi_z=\psi(z,\cdot)$) is a homeomorphism from $H^+_\infty\oplus H^-_\infty$ onto
$H^+_\infty\oplus H^-_\infty$. Moreover, $\phi(z, u+ v)$ belongs to
${\rm Im}(\psi)\cap(\bar B_{H^0_\infty}(\infty, R_1)\times
H^-_\infty)$ if and only if $u=\theta$.
\end{lemma}

\begin{proof}
 By Lemma~\ref{lem:2.12}(b) it suffices to
prove that $\phi$ is continuous. Suppose that
\begin{eqnarray*}
&&(z_0, u_0', v_0')\in\bar B_{H^0_\infty}(\infty, R_1)\times
H^+_\infty\times H^-_\infty\quad\hbox{and}\\
&&\{(z_n, u_n', v_n')\}\subset \bar B_{H^0_\infty}(\infty,
R_1)\times H^+_\infty\times H^-_\infty
\end{eqnarray*}
satisfy: $z_n\to z_0$ and
\begin{eqnarray*}
&&u_n:=\psi_1(z_n, u_n'+ v_n')\to u_0=\psi_1(z_0, u_0'+
v_0'),\\
&&v_n:=\psi_2(z_n, u_n'+ v_n')\to v_0=\psi_2(z_0, u_0'+ v_0').
\end{eqnarray*}
Our goal is to prove that $u'_n\to u'_0$ and $v'_n\to v'_0$.

{\bf Step 1.} {\it Prove that $\{u_n'\}$ and $\{v_n'\}$ are
bounded.}

 For each $n$ either $u_n'=\theta$ or $u_n'\ne\theta$ and
$$
u_n=\frac{\sqrt{F^\infty(z_n, u_n' +
\varphi_{z_n}(u_n'))}}{\|u_n'\|}u_n'
$$
and hence
\begin{eqnarray*}
\|u_n\|^2=F^\infty(z_n, u_n'+ \varphi_{z_n}(u_n'))\ge F^\infty(z_n,
u_n')\ge\frac{a_1}{4}\|u_n'\|^2
\end{eqnarray*}
by  (\ref{e:2.19}) and (\ref{e:2.20}). Since $\|u_n\|\to\|u_0\|$  we
deduce that $\{u_n'\}$ is bounded and that $u_n'\to \theta=u_0'$ as
$n\to\infty$  if $u_0'=\theta$ (and so $u_0=\theta$ by the definition of $\psi_1$).

For each $n$ either $v_n'=\varphi_{z_n}(u_n')$ or $v_n'\ne
\varphi_{z_n}(u_n')$ and
$$
v_n=\frac{\sqrt{F^\infty(z_n, u_n' +
\varphi_{z_n}(u_n'))-F^\infty(z_n, u_n'+
v_n')}}{\|v_n'-\varphi_{z_n}(u_n')\|}(v_n'-\varphi_{z_n}(u_n')).
$$
In the latter case $F^\infty(z_n, u_n' +
\varphi_{z_n}(u_n'))-F^\infty(z_n, u_n'+ v_n')=\|v_n\|^2$. Since
$\{u_n'\}$, $\{z_n\}$ and thus $\{h^\infty(z_n)\}$ are bounded,  it follows from
Lemma~\ref{lem:2.8}(b) that
 $\{\varphi_{z_n}(u_n')\}$ is bounded, which implies by Lemma~\ref{lem:2.7}(b)
 that $\{F^\infty(z_n, u_n' + \varphi_{z_n}(u_n'))\}$
is bounded. Hence $\{F^\infty(z_n, u_n'+ v_n')\,|\, v_n'\ne
\varphi_{z_n}(u_n')\}$ is bounded. Using Lemma~\ref{lem:2.7}(b)
again we deduce that $\{v_n'\,|\, v_n'\ne \varphi_{z_n}(u_n')\}$ is
bounded. The claim is proved.

{\bf Step 2.} {\it Prove that $u_n'\to u_0'$ and $v_n'\to v_0'$.}

 The first claim has been proved if $u_0'=\theta$. Let us consider the case
$u_0'\ne\theta$. Since $\|\psi_1(z_0, u_0'+ v_0')\|=\sqrt{j(z_0,
u_0')}>0$, $\psi_1(z_n, u_n'+ v_n')\to \psi_1(z_0, u_0'+ v_0')$ and
hence $\|\psi_1(z_n, u_n'+ v_n')\|>0$ for large $n$, we deduce that
for large $n$, $u_n'\ne\theta$ and $j(z_n, u_n')=\|\psi_1(z_n, u_n'+
v_n')\|^2$  converges to $j(z_0, u_0')$. Now
$$
\frac{\sqrt{j(z_n,u_n')}}{\|u_n'\|}u_n'=u_n\to
u_0=\frac{\sqrt{j(z_0,u_0')}}{\|u_0'\|}u_0'\quad\Longrightarrow
\quad\frac{u_n'}{\|u_n'\|}\to \frac{u_0'}{\|u_0'\|}.
$$

Suppose by a contradiction that $\{u_n'\}$ does not converge to $u_0'$.
There exists a subsequence $\{u'_{n_k}\}$ and $\epsilon>0$ such that
$\|u'_{n_k}-u_0'\|\ge\epsilon$ for all $k$. We may assume that
$\|u'_{n_k}\|\to\alpha$ due to the boundedness of $\{u_n'\}$. Then
$\{u'_{n_k}\}$ converges to $\frac{\alpha}{\|u_0'\|}u_0'$ and hence
$j(z_{n_k}, u'_{n_k})\to j(z_0, \frac{\alpha}{\|u_0'\|}u_0')=j(z_0,
u_0')>0$. The latter implies
\begin{equation}\label{e:2.28}
\psi_1(z_0, \frac{\alpha}{\|u_0'\|}u_0')=\psi_1(z_0, u_0').
\end{equation}
Since  $\{v_n'\}$ is bounded, we may assume that $v'_{n_k}\to v'$ by
replacing $\{u'_{n_k}\}$ with a subsequence. Then
$$
\psi_2(z_0, \frac{\alpha}{\|u_0'\|}u_0'+ v')\longleftarrow
\psi_2(z_{n_k}, u'_{n_k}+ v'_{n_k})=v_{n_k}\to v_0=\psi_2(z_0, u_0'+
v_0').
$$
Obverse that $\psi_1$ is independent of elements in $H^-_\infty$. By
 (\ref{e:2.28}) we get
$$
\psi_1(z_0, \frac{\alpha}{\|u'_0\|}u'_0+ v')=\psi_1(z_0, u'_0+ v'_0)
$$
and hence
$$
\psi(z_0, \frac{\alpha}{\|u_0'\|}u_0'+ v')=\psi(z_0, u_0'+ v_0').
$$
The latter implies that $\frac{\alpha}{\|u_0'\|}u_0'=u_0'$ and
$v'=v_0'$ because $\psi(z_0, \cdot)$ is one-to-one. It follows that $\alpha=\|u_0'\|$
and $u'_{n_k}\to u_0'$. This contradiction shows that $u_n'\to
u_0'$.

Similarly,  suppose by a contradiction that $\{v_n'\}$ does not converge
to $v_0'$. There exists a subsequence $\{v'_{n_k}\}$ and
$\epsilon>0$ such that $\|v'_{n_k}-v_0'\|\ge\epsilon$ for all $k$.
Passing to a subsequence we may assume $v'_{n_k}\to v'$ as above.
Then we also obtain a contradiction because
$$
\psi_2(z_0, u_0'+ v')\longleftarrow \psi_2(z_{n_k}, u'_{n_k}+
v'_{n_k})=v_{n_k}\to v_0=\psi_2(z_0, u_0'+ v_0')
$$
and hence $\psi(z_0, u_0'+ v')=\psi(z_0, u_0'+ v_0')$ by
(\ref{e:2.28}), which implies $v'=v_0'$. It contradicts the
assumption that $\|v'-v_0'\|\ge\epsilon$.
\end{proof}

\begin{lemma}\label{lem:2.14}
 For any $r\in (0, \infty)$  there exists a number $\delta_r>0$
such that
$$
\bar B_{H^0_\infty}(\infty, R_1)\times \bigl(\bar
B_{H^+_\infty}(\theta, \delta_r)\oplus \bar B_{H^-_\infty}(\theta,
\delta_r)\bigr)
$$
is contained in
\begin{eqnarray}\label{e:2.29}
&& U(R_1, r):= \psi^{-1}\left(B_{H^+_\infty}(\theta,
\sqrt{a_1}\epsilon_r)\oplus B_{H^-_\infty}(\theta,
\sqrt{a_1}\epsilon_r)\right).
\end{eqnarray}
By Lemma~\ref{lem:2.11} and Lemma~\ref{lem:2.12}(a) we have a
bijection
\begin{eqnarray*}
&&\bar B_{H^0_\infty}(\infty, R_1)\times\Bigl(
B_{H^+_\infty}(\theta, \sqrt{a_1}\epsilon_r)\oplus
B_{H^-_\infty}(\theta, \sqrt{a_1}\epsilon_r)\Bigr)\to U(R_1,r),\nonumber\\
&&\hspace{20mm} (z, u+ v)\mapsto (z, \psi(z, u+v)),
\end{eqnarray*}
whose inverse,  denoted by $\phi$, has a form
$$
\phi(z, u+ v)= (z, \phi_z(u+v)):=(z, u'+ v'),
$$
where $(u', v')\in B_{H^+_\infty}(\theta,
\sqrt{a_1}\epsilon_r)\times B_{H^-_\infty}(\theta,
\sqrt{a_1}\epsilon_r)$ is a unique point satisfying $u+ v=\psi(z,
u'+v')$. This bijection $\phi$ is actually a homeomorphism and
$$
F^\infty(\phi(z, u+ v))=\|u\|^2-\|v\|^2
$$
for any $(z, u+ v)\in U(R_1, r)$.
 Moreover, $\phi(z, u+ v)\in{\rm
Im}(\psi)\cap(\bar B_{H^0_\infty}(\infty, R_1)\times H^-_\infty)$ if
and only if $u=\theta$.
\end{lemma}

\begin{proof}
We only prove the first claim. The proofs of others are the same as
those of Lemma~\ref{lem:2.13}.

Let $r\in (0,\infty)$ be given. Since $\psi$ is continuous and $\psi(z, \theta)=\theta$ for any
$z\in \bar B_{H^0_\infty}(\infty, R_1)$, it is easily seen that for
a given large $R>R_1$ we have
$$
(\bar B_{H^0_\infty}(\infty, R_1)\cap B_{H^0_\infty}(\theta, R))\times
\bigl(B_{H^+_\infty}(\theta, \delta)\oplus B_{H^-_\infty}(\theta,
\delta)\bigr)\subset U(R_1, r)
$$
for sufficiently small $\delta>0$. So if the conclusion in
Lemma~\ref{lem:2.14} does not hold for this $r$ then there exist sequences
$\{z_n\}\subset \bar B_{H^0_\infty}(\infty, R_1)$ and $\{u_n^+ +
u^-_n\}\subset H^\pm_\infty\setminus\{\theta\}$ such that
$\|z_n\|\to\infty$, $\|u_n^+ + u^-_n\|\to 0$ (hence $\|u^+_n\|\to 0$
and $\|u^-_n\|\to 0$) and
$$
 \psi(z_n, u^+_n+ u^-_n)\notin B_{H^+_\infty}(\theta, \sqrt{a_1}\epsilon_r)\oplus
B_{H^-_\infty}(\theta, \sqrt{a_1}\epsilon_r)\quad\hbox{for all}\;
n=1,2,\cdots.
$$
The last relation implies that
$$
\hbox{either}\quad \|\psi_1(z_n, u^+_n+ u^-_n)\|\ge
\sqrt{a_1}\epsilon_r\quad\hbox{or}\quad \|\psi_2(z_n, u^+_n+
u^-_n)\|\ge \sqrt{a_1}\epsilon_r
$$
for each $n=1,2,\cdots$. After passing to a subsequence two cases
happen:
\begin{enumerate}
\item[$\bullet$] $\|\psi_1(z_n,
u^+_n+ u^-_n)\|\ge \sqrt{a_1}\epsilon_r\quad\hbox{for all}\; n=1,2,\cdots$.

\item[$\bullet$] $\|\psi_2(z_n,
u^+_n+ u^-_n)\|\ge \sqrt{a_1}\epsilon_r\quad\hbox{for all}\; n=1,2,\cdots$.
\end{enumerate}

In the first case, by the definition of $\psi_1$ we have
$u^+_n\ne\theta$ and
$$
F^\infty(z_n, u^+_n+ \varphi_{z_n}(u^+_n))\ge
a_1\epsilon_r^2\quad\hbox{for all}\; n=1,2,\cdots.
$$
Since $\|u_n^+\|\to 0$, we may assume that $u_n^+\in B_{H^+_\infty}(\theta,\varepsilon_r)$ and hence
 $\varphi_{z_n}(u_n^+)\in B_{H^-_\infty}(\theta,r)$ for all $n\in\N$. After passing to a subsequence we may assume
$\varphi_{z_n}(u^+_n)\to v_0\in H^-_\infty$. Then
Lemma~\ref{lem:2.7}(a) leads to
$$
F^\infty(z_n, u^+_n+ \varphi_{z_n}(u^+_n))\to
\frac{1}{2}(B(\infty)v_0, v_0)_H\le 0
$$
and hence a contradiction.

In the second case we have $u^-_n\ne\varphi_{z_n}(u^+_n)$ and
$$
F^\infty(z_n, u^+_n+ \varphi_{z_n}(u^+_n))-F^\infty(z_n, u^+_n+
u^-_n)\ge a_1\epsilon_r^2\quad\hbox{for all}\; n=1,2,\cdots.
$$
As above we may assume $\varphi_{z_n}(u^+_n)\to v_0\in H^-_\infty$
and use Lemma~\ref{lem:2.7}(a) to obtain
$$
F^\infty(z_n, u^+_n+ \varphi_{z_n}(u^+_n))-F^\infty(z_n, u^+_n+
u^-_n)\to \frac{1}{2}(B(\infty)v_0, v_0)_H\le 0.
$$
This also gives a contradiction. Lemma~\ref{lem:2.14} is proved.
\end{proof}


 Note: If Lemmas~\ref{lem:2.4}--\ref{lem:2.6} only hold in a set $\bar
B_{H^0_\infty}(\infty, R')\oplus (\bar B_H(\theta, r')\cap
X^\pm_\infty)$, we require $z$ and $r$ in Lemma~\ref{lem:2.14} to
sit in $\bar B_{H^0_\infty}(\infty, R')$ and $(0, r')$,
respectively.

\vspace{2mm}

\noindent{\bf Completion of proof of Theorem~\ref{th:1.1}}: For the
homeomorphism in Lemma~\ref{lem:2.13},
\begin{eqnarray*}
&&\phi:\bar B_{H^0_\infty}(\infty, R_1)\times (H^+_\infty\oplus
H^-_\infty)\to \bar
B_{H^0_\infty}(\infty, R_1)\times (H^+_\infty\oplus H^-_\infty),\nonumber\\
&&\hspace{20mm} (z, u^+ + u^-)\mapsto (z, \phi_z(u^+ + u^-)),
\end{eqnarray*}
by (\ref{e:2.10}) we have
\begin{eqnarray*}
\mathcal{ L}(z+ h^\infty(z)+ \phi_z(u^+ + u^-))-\mathcal{
L}(z+
h^\infty(z))&=&F^\infty(\phi(z, u^++ u^-))\\
&=&\|u^+\|^2-\|u^-\|^2
\end{eqnarray*}
for any $(z, u^+, u^-)\in\bar B_{H^0_\infty}(\infty, R_1)\times
H^+_\infty\times H^-_\infty$. Define
\begin{eqnarray*}
&&\Phi:\bar B_{H^0_\infty}(\infty, R_1)\times (H^+_\infty\oplus
H^-_\infty)\to H,\nonumber\\
&&\hspace{10mm} (z, u^+ + u^-)\mapsto z+ h^\infty(z)+ \phi_z(u^+ +
u^-).
\end{eqnarray*}
Since $h^\infty$ takes values in $H^\pm_\infty$, it is easy to check that
$\Phi$ is a homeomorphism from $\bar B_{H^0_\infty}(\infty, R_1)\times (H^+_\infty\oplus
H^-_\infty)$ onto $\bar B_{H^0_\infty}(\infty, R_1)\times (H^+_\infty\oplus
H^-_\infty)$ (by Lemma~\ref{lem:2.13}), and that
$$
\mathcal{ L}(\Phi(z, u^+ + u^-)) =\|u^+\|^2-\|u^-\|^2+ \mathcal{L}(z+ h^\infty(z))
$$
for any $(z, u^+, u^-)\in\bar B_{H^0_\infty}(\infty, R)\times H^+_\infty\times
H^-_\infty$.
The other conclusions  in Theorems~\ref{th:1.1}   directly
follow from Lemmas~\ref{lem:2.2}, \ref{lem:2.3}, \ref{lem:2.7},
\ref{lem:2.8}(b), \ref{lem:2.10}-\ref{lem:2.12}(b)  and
Lemma~\ref{lem:2.13}. \hfill$\Box$\vspace{2mm}

\noindent{\bf Completion of proof of Theorem~\ref{th:1.3}}: For the
homeomorphism in Lemma~\ref{lem:2.14},
\begin{eqnarray*}
&&\phi: U(R_1,r)\to\bar B_{H^0_\infty}(\infty, R_1)\times\Bigl(
B_{H^+_\infty}(\theta, \sqrt{a_1}\epsilon_r)+
B_{H^-_\infty}(\theta, \sqrt{a_1}\epsilon_r)\Bigr),\nonumber\\
&&\hspace{10mm} (z, u+ v)\mapsto (z, \phi_z(u+v)),
\end{eqnarray*}
as above we may use (\ref{e:2.10}) to get
$$
\mathcal{ L}(z+ h^\infty(z)+ \phi_z(u^+ + u^-))-\mathcal{
L}(z+ h^\infty(z))= \|u^+\|^2-\|u^-\|^2
$$
for any $(z, u^+ + u^-)\in U(R_1, r)$.

By Lemma~\ref{lem:2.14} and Lemma~\ref{lem:2.12}(a) we have
\begin{eqnarray*}
\overline{C_{R_1, \delta_r}}&=&\bar B_{H^0_\infty}(\infty,
R_1)\times \bigl( \bar B_{H^+_\infty}(\theta, \delta_r)\oplus \bar
B_{H^-_\infty}(\theta,
\delta_r)\bigr)\\
&\subset& U(R_1, r)= \psi^{-1}\left(B_{H^+_\infty}(\theta,
\sqrt{a_1}\epsilon_r)\oplus B_{H^-_\infty}(\theta,
\sqrt{a_1}\epsilon_r)\right)\\
&\subset&\bar B_{H^0_\infty}(\infty, R_1)\times \bigl(
B_{H^+_\infty}(\theta, 2\epsilon_r)\oplus B_{H^-_\infty}(\theta,
r)\bigr)\subset \overline{C_{R_1, r}}
\end{eqnarray*}
(because we may assume $2\epsilon_r<r$). Define
\begin{eqnarray*}
\Phi:C_{R_1, \delta_r}\to H,\; (z, u^+ + u^-)\mapsto z+ h^\infty(z)+
\phi_z(u^+ + u^-),
\end{eqnarray*}
and $V(R,r):=\Phi\bigl(C_{R, \delta_r}\bigr)$ for every $R\ge R_1$.
Note that $h^\infty$ is a map from $\bar B_{H^0_\infty}(\infty, R_1)$ to $\bar B_{X^\pm_\infty}(\theta, \rho_A)$ by
Lemma~\ref{lem:2.2}. One easily prove that
$$
V(R_1,r)=\Phi\bigl(C_{R_1, \delta_r}\bigr)\subset \overline{C_{R_1,
r+\rho_A}}.
$$
By Lemma~\ref{lem:2.14}, (as in the proof of \cite[Lemma~2.18]{Lu2} or \cite[Lemma~3.6]{Lu3}) one may prove:\\
(i) $V(R_1,r)$ is an open set of $H$, \\
(ii) $\Phi$ is a homeomorphism from $C_{R_1, \delta_r}$ onto
$V(R_1,r)$,\\
(iii) for any $(z, u^+, u^-)\equiv z+ u^+ + u^-\in C_{R_1,
\delta_r}$,
$$
\mathcal{ L}(\Phi(z, u^+ + u^-)) =\|u^+\|^2-\|u^-\|^2+ \mathcal{
L}(z+ h^\infty(z)).
$$

The other conclusions in Theorem~\ref{th:1.3} follow from Lemmas~\ref{lem:2.2},
\ref{lem:2.3}, \ref{lem:2.7}, \ref{lem:2.8}(a),
\ref{lem:2.10}-\ref{lem:2.12}(a) and Lemma~\ref{lem:2.14}.
 \hfill$\Box$\vspace{2mm}

By the Note in Remark~\ref{rm:2.9} and the Notes under
Lemmas~\ref{lem:2.5},~\ref{lem:2.8},~\ref{lem:2.11},~\ref{lem:2.13},~\ref{lem:2.14}
one may obtain the conclusions in Remark~\ref{rm:1.4}. Similarly,
that of Remark~\ref{rm:1.2} can be obtained.

\begin{remark}\label{rm:2.15}
{\rm (a) Under the assumptions that
\begin{equation}\label{e:2.30}
\hbox{$\mathcal{ L}$ is $C^2$ and $D^2\mathcal{ L}(w)=B(\infty)+
o(1)$ as $\|w\|\to\infty$},
\end{equation}
by increasing $R_1$ we may assure that the map
$$
\bar B_{H^0_\infty}(\infty, R_2)\times B_{H^+_\infty}(\theta,
r_\mathcal{ L})\to H^-_\infty,\;(z, u)\mapsto\varphi_z(u)
$$
is $C^1$. In particular, if (\ref{e:1.5}) holds then $(z,
u)\mapsto\varphi_z(u)$ gives a $C^1$  map from $\bar
B_{H^0_\infty}(\infty, R_1)\times H^+_\infty$ to $H^-_\infty$.
 As a consequence,  the map $j$ in (\ref{e:2.18}) is $C^1$ on $\bar
B_{H^0_\infty}(\infty, R_1)\times B_{H^+_\infty}(\theta, r_\mathcal{
L})$.

In fact,  since $\mathcal{ L}$ is $C^2$, $h^\infty$ is $C^1$ by the
final claim of Lemma~\ref{lem:2.2}. Moreover,  by
Remark~\ref{rm:2.9} $\varphi_z(u)\in
 H^-_\infty$ is the unique maximum point of the function
 $$
H^-_\infty\to\R,\;v\mapsto F^\infty(z, u+ v)=\mathcal{ L}(z+
h^\infty(z)+ u+  v)-\mathcal{ L}(z+ h^\infty(z)).
$$
We derive $(\nabla\mathcal{ L}(z+ h^\infty(z)+ u+
\varphi_z(u)),v)_H=0\;\forall v\in H^-_\infty$, that is,
$$
P^-_\infty\nabla\mathcal{ L}(z+ h^\infty(z)+
u+\varphi_z(u))=\theta.
$$
Consider the map
$$
\Xi:\bar B_{H^0_\infty}(\infty, R_1)\times B_{H^+_\infty}(\theta,
r_\mathcal{ L})\times H^-_\infty\to H^-_\infty
$$
given by $\Xi(z, u, v)= P^-_\infty\nabla\mathcal{ L}(z+ h^\infty(z)+
u+ v)$. It is $C^1$ and
$$
D_v\Xi(z, u, \varphi_z(u))=P^-_\infty D^2\mathcal{ L}(z+
h^\infty(z)+ u+ \varphi_z(u))|_{H^-_\infty}:H^-_\infty\to
H^-_\infty.
$$
 Since $\|z+
h^\infty(z)+ u+ \varphi_z(u)\|^2=\|z\|^2+ \|h^\infty(z)+ u+
\varphi_z(u)\|^2\ge \|z\|^2$   and $D^2\mathcal{ L}(w)=B(\infty)+
o(1)$ as $\|w\|\to\infty$ we can increase $R_1$ so that for any
$(z,u)\in \bar B_{H^0_\infty}(\infty, R_1)\times
B_{H^+_\infty}(\theta, r_\mathcal{ L})$ the operator $D_v\Xi(z, u,
\varphi_z(u))$ has a bounded inverse. Hence the desired conclusion
follows from the implicit function theorem.

{(b)} Under the assumption (\ref{e:2.30}),  the
homeomorphism
\begin{eqnarray*}
&&\phi^{-1}:\bar B_{H^0_\infty}(\infty, R_2)\times\bigl(
B_{H^+_\infty}(\theta, \sqrt{a_1}\epsilon_r)+
B_{H^-_\infty}(\theta, \sqrt{a_1}\epsilon_r)\bigr)\to U(R_2,r),\nonumber\\
&&\hspace{10mm} (z, u+ v)\mapsto (z, \psi(z, u+v)),
\end{eqnarray*}
 is $C^1$ on $\bar B_{H^0_\infty}(\infty,
R_1)\times\bigl( B_{H^+_\infty}(\theta, \sqrt{a_1}\epsilon_r)+
B_{H^-_\infty}(\theta,
\sqrt{a_1}\epsilon_r)\bigr)\setminus\triangle_r$, where
$$
\triangle_r:=\Bigl\{(z, u+ \varphi_z(u))\,|\, (z,u)\in\bar
B_{H^0_\infty}(\infty, R_1)\times B_{H^+_\infty}(\theta, r_\mathcal{
L})\Bigr\}
$$
is a $C^1$-submanifold of $\bar B_{H^0_\infty}(\infty, R_1)\times
H^\pm_\infty$ of codimension $\mu_\infty$.

Indeed, it has been proved that the map $j$ in (\ref{e:2.18}) is
$C^1$ on $\bar B_{H^0_\infty}(\infty, R_1)\times
B_{H^+_\infty}(\theta, r_\mathcal{ L})$ above. Then the construction
of $\psi$ directly gives the desired conclusion.

Let $V(R_1, r)$ be as in the proof of Theorem~\ref{th:1.3}. Write a
point of $V(R_1,r)$ as $(z, u^+ + u^-)$, where $z\in
B_{H^0_\infty}(\infty, R_1)$ and $u^\ast\in H^\ast_\infty$,
$\ast=+,-$. It is easily checked that $\Phi^{-1}: V(R_1,r)\to
C_{R_1,\delta_r}$ is given by
$$
\Phi^{-1}(z, u^++ u^-)=\phi^{-1}(z, u^++ u^-- h^\infty(z))=\bigr(z,
\psi(z, u^++ u^-- h^\infty(z))\bigl).
$$
Note that $h^\infty$ is $C^1$ (because $\mathcal{ L}$ is $C^2$).
Hence $\Phi^{-1}$ is $C^1$ outside the submanifold of codimension
$\mu_\infty$,
$$
\widetilde\triangle_r:=\Bigl\{(z, u+ \varphi_z(u)+ h^\infty(z))\,|\,
(z,u)\in\bar B_{H^0_\infty}(\infty, R_1)\times
B_{H^+_\infty}(\theta, r_\mathcal{ L})\Bigr\}
$$
Furthermore, if (\ref{e:1.5})  holds,  the restriction of
 $\phi^{-1}$ to $\bar B_{H^0_\infty}(\infty, R_1)\times(
H^+_\infty\oplus H^-_\infty)$,
\begin{eqnarray*}
&&\bar B_{H^0_\infty}(\infty, R_1)\times( H^+_\infty\oplus
H^-_\infty)\to \bar B_{H^0_\infty}(\infty, R_1)\times(
H^+_\infty\oplus
H^-_\infty),\nonumber\\
&&\hspace{20mm} (z, u+ v)\mapsto (z, \psi(z, u+v)),
\end{eqnarray*}
is $C^1$ outside $\triangle_\infty:=\bigl\{(z, u+ \varphi_z(u))\,|\,
(z,u)\in\bar B_{H^0_\infty}(\infty, R_1)\times H^+_\infty\bigr\}$.
Since
$$\Phi^{-1}: B_{H^0_\infty}(\infty, R_1)\times
(H^+_\infty\oplus H^-_\infty)\to B_{H^0_\infty}(\infty, R_1)\times
(H^+_\infty\oplus H^-_\infty)
$$
 is given by
$$
\Phi^{-1}(z, u^++ u^-)=\phi^{-1}(z, u^++ u^-- h^\infty(z))=\bigr(z,
\psi(z, u^++ u^-- h^\infty(z))\bigl),
$$
 we see that $\Phi^{-1}$ is $C^1$ outside the submanifold of codimension
$\mu_\infty$,
$$
\widetilde\triangle_\infty:=\bigl\{(z, u+ \varphi_z(u)+
h^\infty(z))\,|\, (z,u)\in\bar B_{H^0_\infty}(\infty, R_1)\times
H^+_\infty\bigr\}.
$$
 }
\end{remark}

\subsection{The proof of Theorem~\ref{th:1.8}}

\subsubsection{Case $\mu_\infty=0$, i.e.,
$H^-_\infty=\{\theta\}$}

 By (\ref{e:1.1}) and (\ref{e:1.12}), for any $u\in\bar B_H(\infty, R)\cap X$ we have
\begin{eqnarray*}
D{\cal L}(u)u
&=&D{\cal L}(u)u- (B(\infty)u, u)_H+ (B(\infty)u, u)_H\nonumber\\
&=&(A( u)-B(\infty)u, u)_H+ (B(\infty)u, u)_H\nonumber\\
&\ge &2a_\infty\|u\|^2-\|A( u)-B(\infty)u\|\cdot\|u\|\nonumber\\
&\ge& (2a_\infty-\lambda)\|u\|^2\ge a_\infty\|u\|^2.
\end{eqnarray*}
Since ${\cal L}$ is continuously directional differentiable and $X$
is dense in $H$ we get
\begin{equation}\label{e:2.31}
D{\cal L}(u)u\ge a_\infty\|u\|^2\quad\hbox{for all}\; u\in \bar B_H(\infty,
R).
\end{equation}
Define $\psi:\bar B_H(\infty, R)\to H$ by $\psi(u)=\frac{\sqrt{{\cal
L}(u)}}{\|u\|}u$.

{\bf Claim}. {\it $\psi$ is injective}.

In fact, if there exist $u_1,u_2\in \bar B_H(\infty, R)$, $u_1\ne
u_2$, such that $\psi(u_1)=\psi(u_2)$. Then ${\cal L}(u_1)={\cal
L}(u_2)$ and so $u_1/\|u_1\|=u_2/\|u_2\|$. This implies
$\|u_1\|\ne\|u_2\|$. We may assume $\|u_2\|>\|u_1\|$. Then
$u_2=ku_1$, $k>1$. Obverse that $tu_1+ (1-t)u_2=(t+
(1-t)k)u_1\in\bar B_H(\infty, R)$ for all $t\in [0,1]$. We derive
\begin{eqnarray*}
{\cal L}(u_2)-{\cal L}(u_1)&=&{\cal L}(tu_2+
(1-t)u_1)|^{t=1}_{t=0}\\
&=&D{\cal L}(tu_2+ (1-t)u_1)(u_2-u_1)\\
&=&D{\cal L}([tk+ (1-t)]u_1)((k-1)u_1)\\
&=&\frac{k-1}{tk+(1-t)}D{\cal L}([tk+ (1-t)]u_1)((tk+ 1-t)u_1)\\
&\ge& a_\infty\frac{k-1}{tk+(1-t)}\|(tk+ 1-t)u_1\|^2\\
&=&a_\infty (k-1)(tk+1-t)\|u_1\|^2>0
\end{eqnarray*}
because of (\ref{e:2.31}). This contradiction shows that $\psi$ is
injective.

By (\ref{e:1.11}), for any $u\in\bar B_H(\infty,R)=\bar
B_{H^+_\infty}(\infty,R)$ we get
$$
(a_\infty+\lambda)\|u\|^2\ge {\cal L}(u)\ge
(a_\infty-\lambda)\|u\|^2
$$
and hence
$$
\sqrt{2a_\infty}\ge \frac{\sqrt{{\cal L}(u)}}{\|u\|}\ge
\sqrt{a_\infty-\lambda}\quad\forall u\in\bar B_H(\infty,R).
$$
For $\zeta\in\bar B_H(\infty, \sqrt{2a_\infty}R)$ let
$\bar\zeta=\frac{R}{\|\zeta\|}\zeta$. Take $t_2>1$ such that
$$
\sqrt{{\cal
L}(t_2\bar\zeta)}\ge\sqrt{a_\infty-\lambda}t_2\|\bar\zeta\|>\|\zeta\|\ge\sqrt{2a_\infty}R=
\sqrt{2a_\infty}\|\bar\zeta\|\ge\sqrt{ {\cal L}(\bar\zeta)}.
$$
Since $t\mapsto{\cal L}(t\bar\zeta)$ is continuous, the intermediate
value theorem yields a number $t_1\in [1, t_2]$ such that
$\|\zeta\|=\sqrt{{\cal L}(t_1\bar\zeta)}$ and hence
$$
\psi(t_1\bar\zeta)=\sqrt{{\cal
L}(t_1\bar\zeta)}\cdot\frac{t_1\bar\zeta}{\|t_1\bar\zeta\|}=\|\zeta\|\cdot
\frac{\zeta}{\|\zeta\|}=\zeta.
$$
This shows that
$\bar B_H(\infty,
\sqrt{2a_\infty}R)\subset\psi(\bar B_H(\infty, R))$.
Hence for each $u\in \bar B_H(\infty, \sqrt{2a_\infty}R)$ it follows
from the above claim that there exists a unique $\phi(u)\in \bar
B_H(\infty, R)$ such that  $\psi(\phi(u))=u$. Clearly, the map
$\phi:\bar B_H(\infty, \sqrt{2a_\infty}R)\to \bar B_H(\infty, R)$ is
injective. By the definition of $\psi$,
$$
u=\psi(\phi(u))=\frac{\sqrt{{\cal
L}(\phi(u))}}{\|\phi(u)\|}\phi(u)\quad\hbox{and so}\quad{\cal
L}(\phi(u))=\|u\|^2
$$
for any $u\in \bar B_H(\infty, \sqrt{2a_\infty}R)$. Since
$$
\sqrt{2a_\infty}\ge \frac{\sqrt{{\cal L}(\phi(u))}}{\|\phi(u)\|}\ge
\sqrt{a_\infty-\lambda},
$$
we deduce that
$$
\frac{\|u\|}{\sqrt{2a_\infty}}\le\|\phi(u)\|\le\frac{1}{\sqrt{a_\infty-\lambda}}\|u\|\quad\hbox{for all}\;
u\in\bar B_H(\infty, \sqrt{2a_\infty}R).
$$
 Let $\{\zeta_k\}^\infty_{k=1}\subset
\bar B_H(\infty, \sqrt{2a_\infty}R)$ converge to $\zeta\in \bar
B_H(\infty, \sqrt{2a_\infty}R)$. Set $\eta_k=\phi(\zeta_k)$ and
$\eta=\phi(\zeta)$. Then $\psi(\eta_k)=\zeta_k$ and
$\psi(\eta)=\zeta$. So $\|\zeta_k\|\to\|\zeta\|$ implies ${\cal
L}(\eta_k)\to{\cal L}(\eta)$. Note that
$$
\psi(\eta_k)=\frac{\sqrt{{\cal L}(\eta_k)}}{\|\eta_k\|}\eta_k=
\frac{\sqrt{{\cal L}(\eta_k)}}{\|\zeta_k\|}\zeta_k\quad\hbox{and}\quad
\psi(\eta)=\frac{\sqrt{{\cal L}(\eta)}}{\|\eta\|}\eta=\frac{\sqrt{{\cal L}(\eta)}}{\|\zeta\|}\zeta.
$$
We deduce that $\eta_k\to\eta$. That is, $\phi$ is continuous. Hence
$\phi$ is a homeomorphism onto its image and satisfies: ${\cal
L}(\phi(u))=\|u\|^2$ for all $u\in \bar B_H(\infty,
\sqrt{2a_\infty}R)$. Taking $\mathfrak{R}=\sqrt{2a_\infty}R$ gives
the desired conclusion.

\subsubsection{Case $\mu_\infty>0$}

 Note that Lemmas~\ref{lem:2.4},~\ref{lem:2.5} still
hold with $H^0_\infty=\{\theta\}$ under the conditions $({\rm
C1_\infty})$--$({\rm C2_\infty})$ and $({\rm D_\infty})$. Let us give
the corresponding result with Lemma~\ref{lem:2.6}.

\begin{lemma}\label{lem:2.16}
Let $R_1>0$ be as Lemmas~\ref{lem:2.4},~\ref{lem:2.5}, and
$R_2=\max\{R, R_1\}$. Then
\begin{enumerate}
\item[{\rm (a)}] $[D{\cal L}(u+ v_2)-D{\cal L}(u+ v_1)](v_2-v_1)\le -a_\infty\|v_2-v_1\|^2<0$
for any $u\in  H^+_\infty$ with $\|u\|\ge R_2$, and $v_1, v_2\in
H^-_\infty$ with $v_1\ne v_2$;

\item[{\rm (b)}] $D{\cal L}(u+ v)(u-v)\ge a_1\|u\|^2+ a_\infty\|v\|^2>0$ for any
$(u, v)\in  H^+_\infty\times H^-_\infty$ with $(u, v)\ne (\theta,
\theta)$;

\item[{\rm (c)}] $D{\cal L}(u)u\ge a_\infty\|u\|^2> p(\|u\|)$ for any
$u\in H^+_\infty$ with $\|u\|\ge R$, where $p(t)=\frac{a_\infty}{2}t^2$.
\end{enumerate}
\end{lemma}

\begin{proof}
{\rm (a)} For  any $u^+\in  X^+_\infty$ with $\|u^+\|\ge R_1$
and $u^-_1, u^-_2\in H^-_\infty$,  since the function
$$
X\ni u\mapsto (A(u^++u), u^-_2-u^-_1)_H.
$$
is continuously directional differentiable, by the condition (${\rm
F2_\infty}$) and the mean value theorem we have  a number $t\in (0,
1)$ such that
\begin{eqnarray*}
\hspace{-10mm}&&\hspace{-8mm}[D{\cal L}(u^+ + u^-_2)-D{\cal L}(u^++ u^-_1)](u^-_2-u^-_1)\\
=&&\hspace{-6mm}(A( u^++u^-_2), u^-_2-u^-_1)_H - (A( u^++u^-_1), u^-_2-u^-_1)_H\\
=&&\hspace{-6mm}\left(DA(u^++ u^-_1+
t(u^-_2-u^-_1))(u^-_2-u^-_1),
u^-_2-u^-_1\right)_H\\
=&&\hspace{-6mm}\left(B( u^++ u^-_1+
t(u^-_2-u^-_1))(u^-_2-u^-_1),
u^-_2-u^-_1\right)_H\\
\le &&\hspace{-6mm} -a_\infty\|u^-_2-u^-_1\|^2,
\end{eqnarray*}
where the third equality comes from $({\rm F3_\infty})$, and the
final inequality is due to the fact that $\|u^++ u^-_1+
t(u^-_2-u^-_1)\|\ge \|u^+\|\ge R_1$ and Lemma~\ref{lem:2.5}(c).
Hence the desired conclusion follows from the density of
$X^+_\infty$ in $H^+_\infty$.

 (b)  By (\ref{e:1.12}), $\|A(u)-B(\infty)u\|\le
\lambda\|u\|$ for any $u\in X$ with $\|u\|\ge R_2$.
 Because
$X^+_\infty$ is dense in $H^+_\infty$, as above it suffices to prove
the conclusion for $u^+\in X^+_\infty$ with $\|u^+\|\ge R_2$, and
$u^-\in H^-_\infty$. Note that $\|u^++ u^-\|\ge R_2$. We have
\begin{eqnarray*}
\hspace{-3mm}&&D{\cal L}(u^++u^-)(u^+-u^-)\\
\hspace{-3mm}&=&\hspace{-3mm}(A( u^++u^-)-B(\infty)(u^++u^-), u^+-u^-)_H+(B(\infty)(u^++u^-), u^+-u^-)_H\\
\hspace{-3mm}&\ge&\hspace{-3mm}(B(\infty)u^+, u^+)-(B(\infty)u^-, u^-)_H-
\|A( u^++u^-)-B(\infty)(u^++u^-)\|\cdot\|u^+-u^-\|\\
\hspace{-3mm}&\ge&\hspace{-3mm}2a_\infty(\|u^+\|^2+ \|u^-\|^2)-\lambda\|u^++u^-\|\cdot\|u^+-u^-\|\\
\hspace{-3mm}&=&\hspace{-3mm}2a_\infty(\|u^+\|^2+ \|u^-\|^2)-\lambda\sqrt{\|u^++u^-\|^2}\cdot\sqrt{\|u^+-u^-\|^2}\\
\hspace{-3mm}&=&\hspace{-3mm}2a_\infty(\|u^+\|^2+ \|u^-\|^2)-
\lambda\sqrt{\|u^+\|^2+\|u^-\|^2}\cdot\sqrt{\|u^+\|^2+\|u^-\|^2}\\
\hspace{-3mm}&\ge &\hspace{-3mm} a_\infty(\|u^+\|^2+ \|u^-\|^2).
\end{eqnarray*}

 (c) can be proved as that of (\ref{e:2.31}).
\end{proof}

By (\ref{e:1.11}), for any $u^++u^-\in\bar B_H(\infty, R_2)$ we have
\begin{eqnarray}\label{e:2.32}
{\cal L}(u^++u^-)&\le &\frac{1}{2}(B(\infty)(u^++u^-), u^++u^-)+\lambda\|u^++u^-\|^2\nonumber\\
&=&\frac{1}{2}(B(\infty)u^+, u^+)+\frac{1}{2}(B(\infty)u^-, u^-)+\lambda\|u^++u^-\|^2\nonumber\\
&\le&\|B(\infty)\|\cdot\|u^+\|^2- a_\infty\|u^-\|^2+ \lambda\|u^+\|^2+\lambda\|u^-\|^2\nonumber\\
&\le& 2\|B(\infty)\|\|u^+\|^2-(a_\infty-\lambda)\|u^-\|^2
\end{eqnarray}
because (\ref{e:1.1}) implies the inequality $a_\infty\le\|B(\infty)\|$.
In particular, for any $u^+\in H^+_\infty$ it holds that
${\cal L}(u^++u^-)\to -\infty$ as $u^-\in H^-_\infty$ and $\|u^-\|\to\infty$.
By Lemma~\ref{lem:2.16}(a), for each $u^+\in H^+_\infty$ with $\|u^+\|\ge R_2$ the function
$H^-_\infty\ni u^-\mapsto -{\cal L}(u^++u^-)$ is strictly convex. Hence
$H^-_\infty\ni u^-\mapsto {\cal L}(u^++u^-)$ attains the maximum at a unique point
$\varphi(u^+)\in H^-_\infty$. Define
$$
j:H^+_\infty\to\R,\;u^+\mapsto {\cal L}(u^++\varphi(u^+)).
$$
Then $j(u^+)\to +\infty$ as $u^+\in H^+_\infty$ and
$\|u^+\|\to\infty$ because
\begin{equation}\label{e:2.33}
{\cal L}(u^++\varphi(u^+))\ge {\cal
L}(u^+)=\frac{1}{2}(B(\infty)u^+, u^+)-\lambda\|u^+\|^2\ge
(a_\infty-\lambda)\|u^+\|^2.
\end{equation}
As in the proof of
Lemma~\ref{lem:2.10} we may prove that $j$ is continuous, and
continuously directional differentiable. For $(u,v)\in \bar
B_{H^+_\infty}(\infty, R_2)\times H^-_\infty$ define
\begin{eqnarray*}
&&\psi_1(u+ v)= \frac{\sqrt{{\cal L}(u+ \varphi(u))}}{\|u\|}u, \nonumber\\
&&\psi_2(u+ v)=\left\{\begin{array}{ll}
 \frac{\sqrt{{\cal L}(u+ \varphi(u))-{\cal L}(u+ v)}}{\|v-\varphi(u)\|}(v-\varphi(u))
  &\;\hbox{if}\;v\ne\varphi(u),\\
 \theta&\;\hbox{if}\;v=\varphi(u).
 \end{array}\right.\nonumber
\end{eqnarray*}
Then  the map
\begin{equation}\label{e:2.34}
\psi=\psi_1+ \psi_2: \bar
B_{H^+_\infty}(\infty, R_2)\oplus H^-_\infty\to H^\pm_\infty
\end{equation}
is continuous, and satisfies: ${\cal L}(u+ v)=\|\psi_1(u+
v)\|^2-\|\psi_2(u+ v)\|^2$.

For $u\in \bar B_{H^+_\infty}(\infty, R_2)$, since
$\|u+\varphi(u)\|^2=\|u\|^2+\|\varphi(u)\|^2$, by
(\ref{e:2.32})--(\ref{e:2.33}) we have
\begin{equation}\label{e:2.35}
2\|B(\infty)\|\|u\|^2\ge{\cal L}(u+\varphi(u))\ge {\cal L}(u)\ge
(a_\infty-\lambda)\|u\|^2.
\end{equation}
For $\zeta\in\bar B_{H^+_\infty}(\infty, \sqrt{2\|B(\infty)\|}R_2)$
let $\bar\zeta=\frac{R_2}{\|\zeta\|}\zeta$. By (\ref{e:2.35}) we may
take $t_2>1$ such that
\begin{eqnarray*}
\sqrt{{\cal L}(t_2\bar\zeta+
\varphi(t_2\bar\zeta))}&\ge&\sqrt{a_\infty-\lambda}\cdot
t_2\|\bar\zeta\|\\&>&\|\zeta\|
\ge\sqrt{2\|B(\infty)\|}R_2\\
&=&\sqrt{2\|B(\infty)\|}\cdot\|\bar\zeta\|\ge\sqrt{ {\cal
L}(\bar\zeta+ \varphi(\bar\zeta))}.
\end{eqnarray*}
Since $t\mapsto{\cal L}(t\bar\zeta+ \varphi(t\bar\zeta))$ is
continuous, as above we have a number $t_1\in [1, t_2]$ such that
$\|\zeta\|=\sqrt{{\cal L}(t_1\bar\zeta+\varphi(t_1\bar\zeta))}$ and
hence
$$
\psi_1(t_1\bar\zeta)=\sqrt{{\cal
L}(t_1\bar\zeta+\varphi(t_1\bar\zeta))}\cdot\frac{t_1\bar\zeta
}{\|t_1\bar\zeta\|} =\|\zeta\|\cdot\frac{\zeta}{\|\zeta\|}=\zeta.
$$
Let $\xi\in H^-_\infty$ and
$\xi\ne 0$. Note that  the function
$$
[0, \infty)\ni s\mapsto {\cal L}(t_1\bar\zeta+ \varphi(t_1\bar\zeta))-{\cal L}(t_1\bar\zeta+ \varphi(t_1\bar\zeta)+ sv)
$$
takes over all values in $[0, \infty)$ for any  $v\in
H^-_\infty\setminus\{\theta\}$. Take $v=\xi$. We have a number $s>0$
such that
$$\sqrt{{\cal L}(t_1\bar\zeta+ \varphi(t_1\bar\zeta))-{\cal
L}(t_1\bar\zeta+ \varphi(t_1\bar\zeta)+ s\xi)}=\|\xi\|.
$$
Set
$v:=\varphi(t_1\bar\zeta)+ s\xi$. Then
$$
\psi_2(t_1\bar\zeta+v)=\frac{\sqrt{{\cal L}(t_1\bar\zeta+
\varphi(t_1\bar\zeta))-{\cal L}(t_1\bar\zeta+
v)}}{\|v-\varphi(t_1\bar\zeta)\|}(v-\varphi(t_1\bar\zeta))=
\frac{\|\xi\|}{\|s\xi\|}s\xi=\xi.
$$
Hence $\psi(t_1\bar\zeta+v)=\zeta+\xi$. This shows that
$$
\bar B_{H^+_\infty}(\infty, \sqrt{2\|B(\infty)\|}R_2)\oplus
H^-_\infty\subset\psi(\bar B_H(\infty, R_2))=\psi_1(\bar B_H(\infty,
R_2))\oplus H^-_\infty.
$$

As in the proofs of Lemma~\ref{lem:2.11}, Lemma~\ref{lem:2.12}(b) and Lemma~\ref{lem:2.13}
we can show that
$\psi$ is a homeomorphism onto its image (by increasing $R_2>0$ if necessary).
 Let $\phi$ denote the restriction of
$\psi^{-1}$ to $\bar B_{H^+_\infty}(\infty, \sqrt{2\|B(\infty)\|}R_2)\oplus H^-_\infty$.
Set $\mathfrak{R}=\sqrt{2\|B(\infty)}R_2$. We get
$$
{\cal L}(\phi(u+v))=\|u\|^2-\|v\|^2\quad\forall (u,v)\in \bar B_{H^+_\infty}(\infty, \mathfrak{R})\times H^-_\infty.
$$


\section{Relations to previous splitting lemmas at
infinity}\label{sec:3}

\subsection{Relations to the splitting lemma at infinity in
\cite{BaLi}} \label{sec:3.1}

We begin with the following elementary functional analysis fact.

\begin{lemma}\label{lem:3.1}
Let $A_0$ be a bounded linear self-adjoint operator on a Hilbert
space $H$ and let $0$ be an isolated point of $\sigma(A_0)$. Let
$H^0=N(A_0)={\rm Ker}(A_0)$ and $H^+$ (resp. $H^-$) be the positive
(resp. negative) definite subspace of $A_0$. Suppose that both $H^0$ and $H^-$ are
 finite dimensional and
that there exists a number $\alpha>0$ such that $\ast (Au^\ast, u^\ast)\ge
2\alpha\|u^\ast\|^2$ for all $u^\ast\in H^\ast$, $\ast=+,-$.  Then
$A_0$ can be expressed as a sum $P+ Q$, where $Q\in L_s(H)$ is
compact and $P\in L_s(H)$ satisfies: $(Pu,u)\ge 2\alpha\|u\|^2$ for
all $u\in H$.
\end{lemma}

\begin{proof}
Since $A_0$ is self-adjoint and $0$ is an isolated point of
$\sigma(A_0)$, by Proposition 4.5 of \cite{Con} the range $R(A_0)$
is closed, and hence
$N(A_0)^\bot=N(A_0^\ast)^\bot=\overline{R(A_0)}=R(A_0)$. It follows
that $R(A)=H^+\oplus H^-$ and $H=H^0\oplus R(A)=H^0\oplus H^-\oplus
H^+$. Let $P^0:H\to H^0\oplus H^-$ be the orthogonal projection,
which is an operator of finite rank and hence compact. Define
operators $P, Q\in L_s(H)$ by
\begin{eqnarray*}
\begin{array}{clcc}
Pu=2\alpha u\;&\hbox{if}\;u\in H^0,& Pu=\ast
A_0u\;&\hbox{if}\;u\in H^\ast,\ast=+,-, \\
Qu=A_0u-Pu\;&\hbox{if}\;u\in H^0\oplus H^-,
&Qu=\theta\;&\hspace{-17mm}\hbox{if}\;u\in H^+.
\end{array}
\end{eqnarray*}
Then $A_0=P+ Q$,   $Q$ is of finite rank and hence compact,
and $P$ satisfies
\begin{eqnarray*}
(Pu,u)_H&=&(Pu^0,u^0)_H+ (Pu^-,u^-)_H+ (Pu^+, u^+)_H\\
&\ge& 2\alpha\|u^0\|^2+ 2\alpha\|u^-\|^2+
2\alpha\|u^+\|^2=2\alpha\|u\|^2
\end{eqnarray*}
for any $u=u^0+ u^-+ u^+\in H^0\oplus H^-\oplus H^+=H$.
\end{proof}

 Recall the following basic assumption in \cite[p. 425]{BaLi}:
\begin{enumerate}
\item[$({\rm A_\infty})$] $f(x)=\frac{1}{2}(A_0x, x)_H+ g(x)$ where $A_0:H\to H$
 is a self-adjoint linear operator such that $0$ is isolated in the spectrum of $A_0$.
The map $g\in C^1(H, \R)$ is of class $C^2$ in a neighborhood of
infinity and  satisfies $g''(x)\to 0$ as
$\|x\|\to\infty$. Moreover, $g$ and $g'$ map bounded sets to bounded
sets.
\end{enumerate}

({Note}: It was claimed below $({\rm A_\infty})$ in
\cite{BaLi} that $({\rm A_\infty})$ implies: $g(x)=o(\|x\|^2)$ and
$g'(x)=o(\|x\|)$ as $\|x\|\to\infty$, which are used in the proof of Lemma~4.2 of
\cite{BaLi}. The assumption $({\rm A_\infty})$ in \cite[p.226]{HiLiWa}
also required  $g'(x)\to 0$ as
$\|x\|\to\infty$.)

\begin{claim}\label{cl:3.2}
Under the assumption $({\rm A_\infty})$,
suppose that $A_0$ has the finite dimensional kernel and negative definite subspace.
Then the conditions of Corollary~\ref{cor:1.6} are satisfied.
\end{claim}

\begin{proof}
Since $0\in \sigma(A_0)$ is isolated, there exists $\alpha>0$ such
that $\ast (Au^\ast, u^\ast)\ge 2\alpha\|u^\ast\|^2$ for all
$u^\ast\in H^\ast$, $\ast=+,-$. By Lemma~\ref{lem:3.1} we may write
$A_0=P(\infty)+ Q(\infty)$, where $Q(\infty)\in L_s(H)$ is compact
and $P(\infty)\in L_s(H)$ satisfies: $(P(\infty)u,u)\ge
2\alpha\|u\|^2$ for all $u\in H$. We take $B(\infty):=A_0$. Choose
$R>0$ so large that $\|g''(x)\|\le \alpha$ as $\|x\|\ge R$. Since
$B(x)=A_0+ g''(x)=P(\infty)+ Q(\infty)+ g''(x)$,  we derive that
$$
([B(x)-Q(\infty)]u, u)_H=(P(\infty)u,u)_H+ (g''(x)u,u)_H\ge
\alpha\|u\|^2
$$
for all $u\in H$
and $x\in \bar B_H(\infty,R)$.
Namely, the condition (c) of Corollary~\ref{cor:1.6} is satisfied.
Clearly, the condition (d) therein also holds since
$B(x)-B(\infty)=B(x)-A_0=g''(x)\to 0$ as $\|x\|\to\infty$.
\end{proof}

That $A_0$ has a finite dimensional negative definite subspace
corresponds to the finiteness of the Morse index at infinity, which
is needed for computations of critical groups. The finiteness of
$\dim{\rm Ker}(A_0)$ is naturally satisfied in the most actual
applications. In this sense Claim~\ref{cl:3.2} shows that
Corollary~\ref{cor:1.6} is a generalization of the splitting lemma
at infinity on the page 431 of \cite{BaLi}. Our homeomorphism is not
necessarily $C^1$-smooth, but we do not use the condition that
 $g$ and $g'$ map bounded sets to
bounded sets yet.

Consider the following weaker assumption than $({\rm A_\infty})$,
which was given in Remark~2.3 of \cite[p. 226]{HiLiWa}:
\begin{enumerate}
\item[$({\rm A'_\infty})$] $f(x)=\frac{1}{2}(A_0x, x)_H+ g(x)$ where $A_0:H\to H$
 is a self-adjoint linear operator such that $0$ is isolated in the spectrum of $A_0$.
The map $g\in C^1(H, \R)$ is of class $C^2$ in a neighborhood of
infinity and  satisfies: $\exists\;\alpha>0$ such that
\begin{eqnarray*}
&&\ast (A_0u, u)_H\ge 2\alpha\|u\|^2\quad\forall
u\in H^\ast,\;\ast=+, -\quad\hbox{and}\\
&&\|g''(u^0+ u^\pm)\|<\alpha,\quad g'(u^0+ u^\pm)\to 0\;\hbox{as}\;
\|u^0\|\to\infty
\end{eqnarray*}
where $H^0={\rm Ker}(A_0)$ and $H^+$ (resp. $H^-$) is the positive
(resp. negative) definite subspace of $A_0$.  Moreover, $g$ and $g'$
map bounded sets to bounded sets.
\end{enumerate}

Under this condition $({\rm A'_\infty})$, Proposition~3.3 in
\cite{HiLiWa} stated  the following slightly different version of
the splitting lemma of \cite{BaLi}.

\begin{theorem}[\hbox{\cite[Prop.3.3]{HiLiWa}}]\label{th:3.3}
For any $M>0$ there exist $R_0>0$, $\delta>0$, a
$C^1$-diffeomorphism
$$
\psi:C_{R_0,M}=\{u=u^0+ u^\pm\,|\, \|u^0\|>R_0,\;\|u^\pm\|<M\}\to
C_{R_0, 2M}
$$
and a $C^1$-map $w: B_{H^0}(\infty, R_0)\to W^\delta=\{u^\pm\in
H^\pm\,|\, \|u^\pm\|\le\delta\}$ such that
$$
f(\psi(u))=\frac{1}{2}(A_0w, w)_H+ h(u^0)\quad\forall u\in
C_{R_0,M},
$$
where $h(u^0)=f(u^0+ w(u^0))$, $\delta$ can be chosen as small as we
please if we choose $R_0$ large, and $w=w(u^0)$ is the unique
solution of $P^\pm f'(u^0+ w)=0$. Furthermore, $(h'(u^0),
\xi)=(g'(u^0+ w(u^0)), \xi)$ for any  $\xi\in H^0$.
\end{theorem}

{ Note}: It was stated in \cite[p.235]{HiLiWa} that
one may refer to Lemma 4.3 and its proof in \cite{BaLi} for
the first part of this theorem.
Carefully checking the proof of its generalization in
\cite[Th.2.1]{ChenLi2} we believe that the diffeomorphism $\psi$ in this
theorem and Theorem~\ref{th:3.5} below should actually be from
$C_{R_0,M}$ onto an open subset $V$ of $C_{R_0, 2M}$ (possibly
satisfying $V\supseteq C_{R_0,r}$ for some $r>0$). In fact, the
equation (2.19) in \cite{ChenLi2} is solved on ball $B_{E_1}(0,2M)$
for each fixed $y\in Y$ with $\|y\|>R$. The condition that
$\|\chi^t_{x,y}\|_E\le\frac{1}{2}\|x\|_E$ implies that for each
$x\in B_{E_1}(0,M)$ the initial value problem
$\frac{d}{dt}\eta(t)=\chi^t_{\eta(t),y},\,\eta(t)=x$ has a unique
$C^1$-solution $\eta:[0,1]\to\eta(t,x,y)\in B_{E_1}(0,2M)$, which
depends $C^1$-smoothly  on the parameter $(t, y)$ and initial value
$x$. So $B_{E_1}(0,M)\ni x\mapsto\eta(1,x,y)\in B_{E_1}(0,2M)$ is a
$C^1$-diffeomorphism from $B_{E_1}(0,M)$ onto some open neighborhood
of $0$ in $B_{E_1}(0, 2M)$.  The desired $\psi$  given by
$\psi(x,y)=\eta(1,x,y)+ w(y)+ y$, is a $C^1$-diffeomorphism from
$C_{R_0,M}$ onto an open subset $V$ of $C_{R_0, 2M}$ containing
$\{y\in Y\,|\, \|y\|>R\}$. Since $\|w(y)\|\to 0$ as $\|y\|\to\infty$
it is possible to prove that for sufficiently large $R>0$ the image
of $\psi$ contains some $C_{R_0,r}$ for small $r>0$.

\begin{claim}\label{cl:3.4}
Under the assumption $({\rm A'_\infty})$, suppose that $A_0$ has the
finite dimensional kernel and negative definite subspace. Then the
conditions of Corollary~\ref{cor:1.7} are satisfied.
\end{claim}

\begin{proof}
Following the notations in the proof of Claim~\ref{cl:3.2}, since
$\|g''(u^0+ u^\pm)\|<\alpha$ for all $u^0+ u^\pm$, as in the proof
of Claim~\ref{cl:3.2} we may prove that the condition (c) is
satisfied. It remains to  prove that the condition (d) holds in the
present case. Now $B(\infty)=A_0$ and $H^\ast_\infty=H^\ast$,
$\ast=0, -, +$. Since $\ast (B(\infty)u, u)_H\ge
2\alpha\|u\|^2\;\forall u\in H^\ast_\infty$, $\ast=+,-$, the
restrictions $B(\infty)|_{H^\ast_\infty}:H^\ast_\infty\to
H^\ast_\infty$ are invertible and
$\|(B(\infty)|_{H^\ast_\infty})^{-1}\|\le\frac{1}{2\alpha}$. Write
$H^\pm_\infty=H^+_\infty\oplus H^-_\infty$ as before. Then
$B(\infty)|_{H^\pm_\infty}:H^\pm_\infty\to H^\pm_\infty$ is
invertible and
$$
(B(\infty)|_{H^\pm_\infty})^{-1}(u^++u^-)=
(B(\infty)|_{H^+_\infty})^{-1}u^+
+(B(\infty)|_{H^-_\infty})^{-1}u^-
$$
for any $u^++u^-\in H^+_\infty+ H^-_\infty$. This leads to
\begin{eqnarray*}
\|(B(\infty)|_{H^\pm_\infty})^{-1}(u^++u^-)\|^2&=&
\|(B(\infty)|_{H^+_\infty})^{-1}u^+\|^2
+\|(B(\infty)|_{H^-_\infty})^{-1}u^-\|^2\\
&\le&(\frac{1}{2\alpha})^2(\|u^+\|^2+\|u^-\|^2)
\end{eqnarray*}
and hence
$C_1^\infty=\|(B(\infty)|_{H^\pm_\infty})^{-1}\|_{L(H^\pm_\infty)}\le
\frac{1}{2\alpha}$. Since $B(x)-B(\infty)=g''(x)$,
 $$
\|B(z+y)|_{H^\pm_\infty}-B(\infty)|_{H^\pm_\infty}\|_{L(H^\pm_\infty)}=
\|g''(z+y)|_{H^\pm_\infty}\|_{L(H^\pm_\infty)}<\alpha\le\frac{1}{2C_1^\infty}
$$
for all $y\in H^\pm_\infty$ and $z\in H^0_\infty$. Hence the
condition (d) holds with $\rho_A=\infty$. But $M(A)=0$ because
$g'(u^0+ u^\pm)\to 0$ as $\|u^0\|\to\infty$ (we here only need
$g'(u^0)\to 0$ as $\|u^0\|\to\infty$). We can also take $\rho_A$ to
be any given $\delta>0$ so that the $C^1$-map $w$ in
Theorem~\ref{th:3.3} is assured to take values in $W^\delta=\{u^\pm\in H^\pm\,|\,
\|u^\pm\|\le\delta\}$. Without the condition that $g'(u^0)\to 0$ as
$\|u^0\|\to\infty$, we may also derive Theorem~\ref{th:3.3} except claims
that $\psi$ is $C^1$ and $w$ takes values in $W^\delta$.
\end{proof}

Hence Claim~\ref{cl:3.2} (and Note below Theorem~\ref{th:3.3}) shows
that Corollary~\ref{cor:1.7} is a generalization of
Theorem~\ref{th:3.3}. We only need that
$\sup\|g''(z+y)|_{H^\pm_\infty}\|_{L(H^\pm_\infty)}\le\frac{1}{\kappa
C_1^\infty}$ for some $1<\kappa\le 2$. This is better than the
condition that
$\sup\|g''(z+y)|_{H^\pm_\infty}\|_{L(H^\pm_\infty)}\le\alpha\le\frac{1}{2C_1^\infty}$.
Moreover, we  do not use the condition that
 $g$ and $g'$ map bounded sets to
bounded sets.

\subsection{Relations to the generalization version in
\cite{ChenLi1}} For convenience of comparison with ours we briefly
review it in our notations. Let $L:H\to H$ be a bounded self-adjoint
linear operator. Let $H^0_\infty={\rm Ker}(L)$ and
$H^\pm_\infty=(H^0_\infty)^\bot$. It was assumed in \cite{ChenLi1} that $L$
satisfies the condition
\begin{enumerate}
\item[(L)] The operator
$L|_{H^\pm_\infty}:H^\pm_\infty\to H^\pm_\infty$ is invertible and
its inverse operator $(L|_{H^\pm_\infty})^{-1}:H^\pm_\infty\to
H^\pm_\infty$ is bounded.
\end{enumerate}
By Proposition 4.5 of \cite{Con}   this condition is  equivalent to
our (${\rm C1_\infty}$), that is, {\it $0$ is at most an isolated
point of the spectrum $\sigma(L)$.} (See Proposition~B.3 in
\cite{Lu2},\cite{Lu3}.)

 Denote by $P^0_\infty$ the orthogonal projection onto
$H^0_\infty$. (Then $I-P^0_\infty$ is such a projection onto
$H^\pm_\infty$.) For a $C^2$ functional $\mathcal{
F}:H=H^0_\infty\oplus H^\pm_\infty\to\R$,  let $D^2\mathcal{ F}(x)$
be the Hessian operator of it at a critical point $x$. For $z+u\in
H$, where $z\in H^0_\infty$ and $u\in H^\pm_\infty$, let
$\nabla_2\mathcal{ F}(z,u)\in H^\pm_\infty$ be defined by
$(\nabla_2\mathcal{ F}(z,u), v)_H=d_u\mathcal{ F}(z,u)(v)$. Then
\begin{equation}\label{e:3.1}
\nabla_2\mathcal{ F}(z,u)= (I-P^0_\infty)\nabla\mathcal{ F}(z+u).
\end{equation}
There exists a unique operator $\mathcal{ J}(z,u)\in
L_s(H^\pm_\infty)$ such that
$$
d^2_u\mathcal{ F}(z,u)(v_1,v_2)=(\mathcal{ J}(z,u)v_1,
v_2)_H\;\forall v_1, v_2\in H^\pm_\infty.
$$
It is easily seen that
\begin{equation}\label{e:3.2}
\mathcal{ J}(z,u)=(I-P^0_\infty)D(\nabla \mathcal{
F})(z+u)|_{H^\pm_\infty}
\end{equation}
 because
\begin{eqnarray*}
d^2_u\mathcal{ F}(z,u)(v_1,v_2)&=&\frac{\partial^2}{\partial
s_1\partial
s_2}\mathcal{ F}(z,u+ s_1v_1+ s_2v_2)\Bigl|_{s_1=0, s_2=0}\\
&=&\frac{d}{ds_2}(\nabla_2\mathcal{ F}(z, u+ s_2v_2), v_1)_H\Bigl|_{s_2=0}\\
&=&\frac{d}{ds_2}\bigl((I-P^0_\infty)\nabla\mathcal{ F}(z+u+ s_2v_2), v_1\bigr)_H\Bigl|_{s_2=0}\\
&=&\left((I-P^0_\infty)D(\nabla\mathcal{ F})(z+u)(v_2),
v_1\right)_H.
\end{eqnarray*}

\begin{theorem}[{\rm (\cite[Theorem~2.1]{ChenLi1})}]\label{th:3.5}
For the above functional $\mathcal{ F}$ and operator $L$, suppose
that there exists some $M>0$ such that as $\|z\|\to\infty$ one has
\begin{enumerate}
\item[$({\rm L_1})$] $\|(I-P^0_\infty)\nabla\mathcal{ F}(z+u)-Lu\|\to 0$ uniformly for
$\|u\|\le M$,
\item[$({\rm L_2})$] $\|(I-P^0_\infty)D(\nabla \mathcal{F})(z+u)|_{H^\pm_\infty}
 -L|_{H^\pm_\infty}\|_{L(H^\pm_\infty)}\to
0$ uniformly for $\|u\|\le M$.
\end{enumerate}
Then there exist $R>0$, a $C^1$-homeomorphism
$$
\psi:{C_{R, M}}=\{z+ u\,|\, z\in H^0_\infty,\, u\in H^\pm_\infty,\,
\|z\|\ge R,\,\|u\|\le M\}\to {C_{R, 2M}}
$$
and a $C^1$-map $h^\infty: B_{H^0_\infty}(\infty, R)\to
B_{H^\pm_\infty}(\theta, M)$ such that
\begin{enumerate}
\item[{\rm (a)}] $\mathcal{ F}(\psi(z+ u))=\frac{1}{2}(Lu,u)_H+
\mathcal{ F}(z+ h^\infty(z))$ for all $z+ u\in {C_{R, M}}$,
\item[{\rm (b)}] $(I-P^0_\infty)\mathcal{ F}(z+ h^\infty(z))=0$ for all $z\in  B_{H^0_\infty}(\infty, R)$,
\item[{\rm (c)}] $\|h^\infty(z)\|\to 0$ as $\|z\|\to\infty$.
\end{enumerate}
\end{theorem}

  The following condition is slightly stronger than $({\rm L_2})$.
\begin{enumerate}
\item[$({\rm L'_2})$] $\|(I-P^0_\infty)D(\nabla \mathcal{ F})(z+u) -L\|_{L(H, H^\pm_\infty)}\to
0\;$ uniformly for $\;\|u\|\le M$.
\end{enumerate}

Take $X=H$, $A(z+u)=\nabla\mathcal{ F}(z+u)$ and $B(\infty)=L$. By
$({\rm L_1})$ we get
$$
M(A)=\lim_{R\to\infty}\sup\{\|(I-P^0_\infty)A(z)\|:\;z\in H^0,
\|z\|\ge R\}=0.
$$

\begin{lemma}\label{lem:3.6}
\begin{enumerate}
\item[{\rm (a)}]  $({\rm L_2})$ implies that $({\rm SE'_\infty})$ holds for
$\rho_A=M>0=C_1^\infty M(A)$.
\item[{\rm (b)}] $({\rm L_1})$ and $({\rm L'_2})$ imply that $({\rm SE_\infty})$ holds for
$\rho_A=M>0=C_1^\infty M(A)$.
\end{enumerate}
\end{lemma}

\begin{proof}
(a) For any $z\in H^0_\infty$ and $u_i\in H^-_\infty$ with
$\|u_i\|\le M$, $i=1,2$, using the mean value theorem in inequality
form we derive
\begin{eqnarray*}
&&\|(I-P^0_\infty)A(z+ u_1)-Lu_1-(I-P^0_\infty)A(z+u_2)+ Lu_2\|\\
&\le& \sup_{t\in [0, 1]}\|(I-P^0_\infty)DA(z+ tu_1+
(1-t)u_2)(u_1-u_2)-L(u_1-u_2)\|\\
&\le& \sup_{t\in [0, 1]}\|(I-P^0_\infty)DA(z+ tu_1+
(1-t)u_2)|_{H^\pm_\infty} -L|_{H^\pm_\infty}\|\cdot\|u_1-u_2\|.
\end{eqnarray*}
From this it is easily seen that $({\rm L_2})$ leads to $({\rm
SE'_\infty})$ with $\rho_A=M$.

(b) For any given $\varepsilon>0$, by $({\rm L_1})$ and $({\rm
L'_2})$ there exists $R>3$ such that
\begin{eqnarray}
&&\|(I-P^0_\infty)A(z+u)-Lu\|<M\varepsilon,\label{e:3.3}\\
&& \|(I-P^0_\infty)DA(z+u) -L\|_{L(H,
H^\pm_\infty)}<\varepsilon\label{e:3.4}
\end{eqnarray}
for any $u\in B_{H^\pm_\infty}(\theta, M)$  and $z\in
B_{H^0}(\infty, R)$. Hence for any $u_i\in B_{H^\pm_\infty}(\theta,
M)$  and $z_i\in B_{H^0}(\infty, R+ 4M)$, $i=1,2$, if
$\|z_1-z_2\|\ge 3M$ then from (\ref{e:3.3}) we derive
\begin{eqnarray*}
&&\|(I-P^0_\infty)A(z_1+ u_1)- Lu_1
-(I-P^0_\infty)A(z_2+ u_2)+ Lu_2\|\\
&&\le 2M\varepsilon\le 2\varepsilon\|z_1+ u_1-z_2-u_2\|
\end{eqnarray*}
because $\|z_1+ u_1-z_2-u_2\|\ge \|z_1-z_2\|-\|u_1-u_2\|\ge
\|z_1-z_2\|-2M\ge M$; and if $\|z_1-z_2\|<3M$ using the mean value
theorem we get a number $t\in (0, 1)$ such that
\begin{eqnarray*}
\!\!&&\!\!\!\!\!\|(I-P^0_\infty)A(z_1+x_1)- Lx_1
-(I-P^0_\infty)A(z_2+x_2)+ Lx_2\|\\
\!\!&\le&\!\!\!\!\!\|(I-P^0_\infty)DA(tz_1+ (1-t)z_2+ tx_1+ (1-t)x_2)(z_1+x_1-z_2-x_2)\\
&&\hspace{20mm}- L(z_1+x_1-z_2-x_2)\|\\
\!\!&\le&\!\!\!\!\! \|(I-P^0_\infty)DA(tz_1+ (1-t)z_2+ tx_1+ (1-t)x_2)- L\|\cdot\|z_1+x_1-z_2-x_2\|\\
\!\! &\le&\!\!\!\!\!\varepsilon\|z_1+ x_1-z_2-x_2\|
\end{eqnarray*}
by (\ref{e:3.4}) because $\|tz_1+ (1-t)z_2\|\ge
\|z_2\|-\|z_1-z_2\|>R+ 4M-3M\ge R+ M$. (ii) follows.
\end{proof}

Take $B(z+u)=\mathcal{ F}''(z+u)=DA(z+u)$. We have

\begin{lemma}\label{lem:3.7}
 $({\rm L_1})$ and $({\rm L'_2})$ imply that $({\rm D}''_\infty)$ in
Remark~\ref{rm:1.4} holds for $X=H$. Moreover, if $M=\infty$ in
$({\rm L_1})$ and $({\rm L'_2})$ then $({\rm D}'_\infty)$ in
Remark~\ref{rm:1.2} holds for $X=H$.
\end{lemma}

\begin{proof}
Let $B(x)=D(\nabla{\cal F})(x)$ and  $B(\infty)=L$. Since $0$ is at
most an isolated point in $\sigma(L)$, we have a positive number
$a_\infty>0$ such that
\begin{eqnarray*}
(Lu, u)_H\ge 2a_\infty\|u\|^2\quad\hbox{for all}\; u\in H^+_\infty,\quad (Lu,
u)_H\le -2a_\infty\|u\|^2\quad\hbox{for all}\; u\in H^-_\infty.
\end{eqnarray*}
 By $({\rm L'_2})$ we have a number $R_0>0$ such that
\begin{equation}\label{e:3.5}
\|(I-P^0_\infty)B(z+u) -L\|_{L(H, H^\pm_\infty)}<a_\infty\quad\forall (z,
u)\in W_\infty,
\end{equation}
where $W_\infty:=B_{H^0_\infty}(\infty, R_0)\times
B_{H^\pm_\infty}(\theta, M)$. Set
$$
\omega_\infty: W_\infty\to [0,
\infty),\;x\mapsto\|(I-P^0_\infty)B(x) -L\|_{L(H, H^\pm_\infty)}.
$$
Then $({\rm L'_2})$ implies that $\omega_\infty(x)\to 0$ as $x\in
W_\infty$ and $\|x\|\to \infty$.

For $x\in W_\infty$ and $v\in H^+_\infty$, we have
\begin{eqnarray*}
&&(B(x)v,v)_H=(B(x)v, (I-P^0_\infty)v)_H=((I-P^0_\infty)B(x)v,v)_H\\
&=&(L v, v)_H+ ((I-P^0_\infty)B(x)v-Lv,v)_H\\
&\ge& 2a_\infty\|v\|^2-\|(I-P^0_\infty)B(x)-L\|\cdot\|v\|^2\ge
a_\infty\|v\|^2
\end{eqnarray*}
because (\ref{e:3.5}).
 Similarly, for all $x\in W_\infty$ and $v\in
H^-_\infty$ we have
$$
(B(x)v,v)_H=(B(x)v, (I-P^0_\infty)v)_H=((I-P^0_\infty)B(x)v,v)_H\le
-a_\infty\|v\|^2.
$$
Finally,  for all $x\in W_\infty$, $u\in H$ and $v\in H^\pm_\infty$,
we get
\begin{eqnarray*}
&&|(B(x)u,v)_H-(B(\infty)u,v)_H|=|(B(x)u-B(\infty)u,
(I-P^0_\infty)v)_H|\\
&=&|((I-P^0_\infty)B(x)u- (I-P^0_\infty)Lu, v)_H|\\
&=&|((I-P^0_\infty)[B(x)- L]u,
v)_H|\le\omega_\infty(x)\|u\|\cdot\|v\|
\end{eqnarray*}
since $(I-P^0_\infty)Lu=L(I-P^0_\infty)u=Lu$.

The second claim is easily seen from the proof above.
\end{proof}


 By Lemmas~\ref{lem:3.6} and \ref{lem:3.7},
under the assumptions $({\rm L_1})$ and $({\rm L'_2})$, if $L$ has
the finite dimensional kernel and negative definite subspace, then
 Theorem~\ref{th:3.5} follows from Theorem~\ref{th:1.3} with
$X=H$ by Remark~\ref{rm:1.4} unless our homeomorphism is not
necessarily $C^1$-smooth. Furthermore, if $M=\infty$ in (${\rm
L_1}$) and (${\rm L'_2}$) a stronger result follows from
Remark~\ref{rm:1.2}, that is, there exist  a positive number
$R$, a (unique) continuous map $h^\infty: B_{H^0_\infty}(\infty,
R)\to X^\pm_\infty$ satisfying (\ref{e:1.6}), and a homeomorphism
$\phi: B_{H^0_\infty}(\infty, R)\oplus H^\pm_\infty\to \bar
B_{H^0_\infty}(\infty, R)\oplus H^\pm_\infty$ of form (\ref{e:1.7})
 such that (\ref{e:1.8}) and (a)-(e) in Theorem~\ref{th:1.1} hold.

{ Note}: (${\rm L_1}$) + (${\rm L'_2}$)= (${\rm L_1}$) + (${\rm
L_2}$) +  the following (\ref{e:3.6}), where
\begin{equation}\label{e:3.6}
\left\{\begin{array}{ll}
 &\|(I-P^0_\infty)D(\nabla \mathcal{
F})(z+u)|_{H^0_\infty}\|_{L(H^0_\infty, H^\pm_\infty)}\to 0\\
&\hbox{uniformly for $\|u\|\le M$ as $z\in H^0_\infty$ and
$\|z\|\to\infty$.}\end{array}\right.
\end{equation}

\section{A simple application}\label{sec:4}


To save the length of this paper we are only satisfied with a simple
application of generalizing  Theorem~5.2 in \cite{BaLi}. Some of the
results in \cite{HiLiWa}, \cite{Li}, \cite{ChenLi1} may be generalized with the
similar ideas. They shall be given in other places.

 Let $\Omega\subset\R^n$ be a bounded open domain with
$C^2$-boundary $\partial\Omega$,  $|\Omega|:={\rm mes}(\Omega)$, and let $p:\Omega\times\R\to\R$ be
a Carth\'eodory function satisfying  $p(x,0)=0\;\forall x\in\Omega$
and the following condition:
\begin{enumerate}
\item[{\rm (p)}] $a_0=\lim_{t\to 0}\frac{p(x,t)}{t}$ uniformly in $x\in\Omega$, and $a=\lim_{|t|\to\infty}\frac{p(x,t)}{t}$
uniformly in $x\in\Omega$.
\end{enumerate}
 Consider the BVP
\begin{equation}\label{e:4.1}
-\triangle u=p(\cdot, u)\;\hbox{in}\;\Omega\quad\hbox{and}\quad
u|_{\partial\Omega}=0.
\end{equation}
It is called {\it non-resonant} at infinity if $a$ is not an
eigenvalue of $-\triangle$ with $0$ boundary conditions. Let
$q(x,t)=p(x,t)-at$, $q_0(x,t)=p(x,t)-a_0t$, and
$$
Q(x,t)=\int^t_0q(x,\tau)d\tau,\quad Q_0(x,t)=\int^t_0q_0(x,\tau)d\tau.
$$
Here are the hypotheses on $q$ given in \cite{BaLi}.
\begin{enumerate}
\item[(${\rm q}_1$)] There exist constants $c_1>0$ and $r\in (0,1)$ such that
$$
|q(x,t)|\le c_1(1+|t|^r)\quad\hbox{for all}\; (x,t)\in\Omega\times\R;
$$
\item[(${\rm q}_2$)] There exist constants $c_2>0$ and $\alpha>1$ such that
\begin{eqnarray*}
&&\hbox{either}\hspace{20mm} Q(x,t)-\frac{1}{2}q(x,t)t\ge c_2(|t|^\alpha-1)\quad\hbox{for all}\; (x,t)\in\Omega\times\R,\\
&&\hbox{or}\hspace{25mm}\frac{1}{2}q(x,t)t-Q(x,t)\ge
c_2(|t|^\alpha-1)\quad\hbox{for all}\; (x,t)\in\Omega\times\R;
\end{eqnarray*}
\item[(${\rm q}_3$)] $q\in C^1(\overline{\Omega}\times\R)$ and $q'_t(x,t)\to 0$
 as $|t|\to\infty$ uniformly in $x\in\Omega$;

\item[(${\rm q}_4^\pm$)] $\pm Q_0(x,t)>0$ for $|t|>0$ small, $x\in\Omega$.
\end{enumerate}

\begin{theorem}[\hbox{\cite[Theorem~5.2]{BaLi}}]\label{th:4.1}
Let the assumptions ({\rm p}) and (${\rm q}_1$)--(${\rm q}_3$) be satisfied.
\begin{enumerate}
\item[{\rm (a)}] If $a_0$ is not an eigenvalue of $-\triangle$ then
(\ref{e:4.1}) has at least one nontrivial solution provided
$a_0<\lambda_m<a$ or $a<\lambda_m<a_0$ for some $m\in\N$.

\item[{\rm (b)}] If $a_0=\lambda_m$ is an eigenvalue but (${\rm q}_4^+$) holds
in addition, then (\ref{e:4.1}) has at least one nontrivial solution
provided $a<a_0$ or $a_0<\lambda_k<a$ for some $k>m$.

\item[{\rm (c)}] If $a_0=\lambda_m$ is an eigenvalue but (${\rm q}_4^-$) holds
in addition, then (\ref{e:4.1}) has at least one nontrivial solution
provided $a_0<a$ or $a<\lambda_k<a_0$ for some $k<m$.
\end{enumerate}
\end{theorem}

We wish to prove this theorem
 provided that the conditions ({p}) and (${\rm q}_1$)--(${\rm q}_3$) are replaced by
 the following four respective weaker ones
\begin{enumerate}
\item[(${\rm p}^\ast$)] $a_0=\lim_{t\to 0}\frac{p(x,t)}{t}$ for a.e.
$x\in\Omega$, and $a=\lim_{|t|\to\infty}\frac{p(x,t)}{t}$
for a.e. $x\in\Omega$;
\item[(${\rm q}^\ast_1$)] There exist constants $c_1>0$, $r\in (0,1)$ and a function $E\in L^2(\Omega)$ such that
$$
|q(x,t)|\le E(x)+ c_1|t|^r\quad\hbox{for all}\; (x,t)\in\Omega\times\R;
$$
\item[(${\rm q}^\ast_2$)] There exist constants $c_2>0$, $\alpha>1$ and $G\in L^1(\Omega)$ such that
\begin{eqnarray*}
&&\hbox{either}\hspace{20mm} Q(x,t)-\frac{1}{2}q(x,t)t\ge c_2|t|^\alpha-G(x)\quad\hbox{for all}\; (x,t)\in\Omega\times\R,\\
&&\hbox{or}\hspace{25mm}\frac{1}{2}q(x,t)t-Q(x,t)\ge
c_2|t|^\alpha-G(x)\quad\hbox{for all}\; (x,t)\in\Omega\times\R;
\end{eqnarray*}
({\it note}: one of both inequalities and (${\rm q}^\ast_1$) imply: $\alpha\le r+1$!)
\item[(${\rm q}^\ast_3$)] For almost every $x\in\Omega$
the function $\R\ni t\mapsto q(x,t)$ is differentiable and
$\Omega\times\R\ni (x, t)\mapsto q'_t(x,t)$ is a Carth\'eodory
function. Moreover, there exist  $s\in \left(\frac{n}{2},
\infty\right)$
$\ell\in L^s(\Omega)$, and a bounded measurable $h:\R\to\R$  such that
$$
h(t)\to\hbar\in\R\quad\hbox{as}\;|t|\to\infty,\quad\hbox{and}\quad |q'_t(x, t)|\le \ell(x)h(t)
$$
for almost every $x\in\Omega$ and for almost all $t\in\R$. (Clearly,
$h\ge 0$ and $\hbar\ge 0$.)
\end{enumerate}

Since $p(x,0)=0$ and $at+ q(x,t)=p(x,t)$ by the
definition, (${\rm p}^\ast$) and (${\rm q}_3^\ast$) imply that for a.e.
$x\in\Omega$, the derivative $q'_{t}(x,0)$ exists and
$a+ q'_{t}(x,0)=p'_t(x,0)=a_0$.

For $q$ in (${\rm q}_3$) let
$h(t):=\max_{x\in\overline{\Omega}}|q'_t(x,t)|$
for each $t\in\R$. It is easily proved that $h\in L^\infty(\R)$ and
$h(t)\to 0$ as $|t|\to\infty$. This shows that $q$ satisfies
(${\rm q}^\ast_3$). On the other hand, for $q$ in (${\rm q}^\ast_3$) we cannot
deduce that $q'_t(x,t)\to 0$
 as $|t|\to\infty$ uniformly in $x\in\Omega$ in case $\hbar=0$ (even
 if we also assume  $q\in C^1(\overline{\Omega}\times\R)$.)
Hence the condition (${\rm q}^\ast_3$) is much weaker than (${\rm q}_3$).

Recall that the Laplacian $-\triangle$ is a self-adjoint operator
defined on $L^2(\Omega)$, with domain $D(-\triangle)=H^2(\Omega)\cap
H^1_0(\Omega)$. It is invertible and $K=(-\triangle)^{-1}$
 is a positive, self-adjoint and completely
continuous operator from $L^2(\Omega)$ to $L^2(\Omega)$ (resp.  from
$H_0^1(\Omega)$ to $H_0^1(\Omega)$).
Moreover,  $K$  satisfies $(u,v)_{L^2}=(Ku,v)_{H_0^1}$ for any $v\in H_0^1(\Omega)$ and $u\in
L^2(\Omega)$, where $(w,v)_{H_0^1}=\int_\Omega\nabla w\cdot\nabla
vdx$.  The eigenvalues of  $-\triangle$ on $\Omega$ with $0$ boundary conditions form
an increasing sequence:
$0<\lambda_1\le\lambda_2\le\lambda_3\le\cdots$, and
$\lambda_n\to\infty$. (Actually, $\lambda_1<\lambda_2$).
$K:H_0^1(\Omega)\to H_0^1(\Omega)$ has a countable set of
eigenvalues $\{\mu_n\}^\infty_{n=1}=\{1/\lambda_n\}^{\infty}_{n=1}$
of finite multiplicity.

 By (${\rm q}^\ast_3$), $s\in (n/2,\infty)$. Hence $\frac{s}{s-1}<\frac{n}{n-2}$ for $n>2$. Set
 \begin{eqnarray}\label{e:4.2}
\xi(s,n)=\left\{\begin{array}{ll}
\frac{s}{s-1}+\frac{n}{n-2} &\hbox{if}\;n>2,\\
\frac{3s}{s-1}&\hbox{if}\;n=2
\end{array}\right.\quad\hbox{and}\quad
\eta(s,n)=\left\{\begin{array}{ll}
\frac{s}{2}\frac{2sn-2s-n}{s^2-s-n} &\hbox{if}\;n>2,\\
\frac{3s}{s-1}&\hbox{if}\;n=2.
\end{array}\right.
\end{eqnarray}
Let  $c(s,n,\Omega)>0$ be the best constant for the embedding $H=H_0^1(\Omega)\hookrightarrow L^{\xi(s,n)}(\Omega)$. Since $|q'_t(x,t)|\le\ell(x)h(t)$ by (${\rm q}^\ast_3$), $1/s+ 1/\eta(s,n)+ 2/\xi(s,n)=1$,
$\eta(s,n)>1$, and $\frac{2s}{s-1}<\xi(s,n)<\frac{2n}{n-2}$ for $n>2$,
using the generalized H\"older inequality, Sobolev embedding theorem  we deduce
\begin{eqnarray}
&&\left|\int_{\Omega}q'_t(x,u(x))v(x)w(x)dx\right|\le\int_{\Omega} |\ell(x)|\cdot |h(u(x))|
\cdot|v(x)|\cdot|w(x)|dx\nonumber\\
&\le&\|\ell\|_{L^s}\|v\|_{L^{\xi(s,n)}}\|w\|_{L^{\xi(s,n)}}
\left(\int_\Omega |h(u(x))|^{\eta(s,n)}dx\right)^{1/\eta(s,n)}\nonumber\\
&\le& (c(s,n,\Omega))^2\|\ell\|_{L^s}
\|v\|_H\|w\|_H\left(\int_\Omega |h(u(x))|^{\eta(s,n)}dx\right)^{1/\eta(s,n)}\label{e:4.3}\\
&\le& (c(s,n,\Omega))^2\|\ell\|_{L^s}
\|v\|_H\|w\|_H|\Omega|^{1/\eta(s,n)}\sup h\label{e:4.4}
\end{eqnarray}
for any $u,v,w\in H$.

Let $H=H^1_0(\Omega)$ for convenience. As a consequence of (\ref{e:4.4}) every $u\in H$ may determine
a bounded linear self-adjoint operator
 $B(u):H\to H$ by the following  equality
\begin{eqnarray}\label{e:4.5}
(B(u)v,w)_H=(v,w)_H-a\int_\Omega
v(x)w(x)dx-\int_{\Omega}q'_t(x,u(x))v(x)w(x)dx
\end{eqnarray}
for $v,w\in H$. Clearly, $(B(\theta)v,w)_H=(v,w)_H-a_0(v,w)_{L^2}$.   Consider  the functional
\begin{equation}\label{e:4.6}
J(u)=\int_\Omega\left(\frac{1}{2}|\nabla u|^2-\frac{1}{2}a
u^2-Q(x,u(x))\right)dx\quad\hbox{for all}\; u\in H,
\end{equation}
 and the bounded linear self-adjoint operator
\begin{equation}\label{e:4.7}
B(\infty):H_0^1(\Omega)\to H_0^1(\Omega), \;u\mapsto u-aKu.
\end{equation}
Then $J(u)=\frac{1}{2}(B(\infty)u,u)_H+ g(u)$ for all $u\in H$, where
 the functional
\begin{equation}\label{e:4.8}
g:H\to\R,\, u\mapsto -\int_\Omega Q(x,u(x))dx.
\end{equation}

\begin{proposition}\label{prop:4.2}
Suppose that the condition (${\rm p}^\ast$) is satisfied. Then
\begin{enumerate}
\item[{\rm (a)}] Under the assumption (${\rm q}^\ast_1$)  the functional
$J$ is $C^1$, and
\begin{eqnarray}\label{e:4.9}
g(u)=J(u)-\frac{1}{2}(B(\infty)u,u)_H=
o(\|u\|_H^2)\quad\hbox{as}\;\|u\|_H\to\infty.
\end{eqnarray}
Moreover, $\nabla J(u)=B(\infty)u+\nabla g(u)=u-aKu+ \nabla g(u)$
 for $u\in H$.

\item[{\rm (b)}] Under the assumption (${\rm q}^\ast_3$), $J$ is
$C^2$ and $J''(u):=D(\nabla J)(u)=B(u)$ for all $u\in H$. Moreover,  if
$a=\lambda_m$  it holds with the constant $c(s,n,\Omega)$ in (\ref{e:4.3}) that
\begin{eqnarray}\label{e:4.10}
\|g''(z+u)\|_{\mathcal{L}(H)}
&\le& (c(s,n,\Omega))^2\|\ell\|_{L^s}
\|h\circ(z+u)-\hbar\|_{L^{\eta(s,n)}}\nonumber\\
&&\quad+(c(s,n,\Omega))^2 |\Omega|^{\frac{1}{\eta(s,n)}}\|\ell\|_{L^s}\hbar
\end{eqnarray}
for any $z\in H^0_\infty={\rm Ker}(B(\infty))$ and
$u\in H^\pm_\infty:=(H^0_\infty)^\bot$.

 \item[{\rm (c)}] Under the assumptions (${\rm q}^\ast_1$) and (${\rm q}^\ast_3$),
\begin{eqnarray}\label{e:4.11}
\|\nabla J(u)-B(\infty)u\|_H=
o(\|u\|_H)\quad\hbox{as}\;\|u\|_H\to\infty.
\end{eqnarray}

\item[{\rm (d)}] Under the assumptions (${\rm q}^\ast_1$) and (${\rm q}^\ast_2$)
the functional $J$ satisfies the Palais-Smale condition.
\end{enumerate}
\end{proposition}
 Its proof is almost standard. For completeness we shall give it at
the end of this section.

Clearly, the origin $\theta$ of $H$ is a critical point of $J$, and
$J''(\theta)=B(\theta)$ is given by $J''(\theta)v=v- a_0Kv$ for $v\in
H=H_0^1(\Omega)$. Denote by
$$
H^0_\theta,\;  H^+_\theta,\;
H^-_\theta\quad\hbox{(resp.}\;H^0_\infty, \; H^+_\infty,\;
H^-_\infty\; \hbox{)}
$$
the kernel, positive and negative definite subspaces of
$J''(\theta)$ (resp. $B(\infty)$). Then $H=H^0_\theta\oplus
H^+_\theta\oplus H^-_\theta$ and $H=H^0_\infty\oplus
H^+_\infty\oplus H^-_\infty$. Both $H^0_\theta\oplus H^+_\theta$ and
$H^0_\infty\oplus H^+_\infty$ are finite dimensional. Let
$\nu_\theta=\dim H^0_\theta$ and $\mu_\theta=\dim H^-_\theta$ (resp.
$\nu_\infty=\dim H^0_\infty$ and $\mu_\infty=\dim H^-_\infty$). They are
the nullity and Morse index of $J$ at $\theta$ (resp. $\infty$). For
$m\in\N$ let
$$
m^-=\min\{j\in\N\,|\,\lambda_j=\lambda_m\}\quad\hbox{and}\quad
m^+=\max\{j\in\N\,|\,\lambda_j=\lambda_m\}.
$$
Clearly, $m^-=m^+=1$  for $m=1$, and $m^-=2$  for $m=2$. Let
$\{\varphi_j\}^\infty_{j=1}$ be a normal orthogonal basis of
$H$ consisting of the eigenfunctions associated with
the eigenvalues $\{\lambda_j\}^{\infty}_{j=1}$. (So
$\lambda_j\int_\Omega|\varphi_j(x)|^2dx=\int_\Omega|\nabla\varphi_j(x)|^2dx=1\;\hbox{for all}\;
j\in\N$.)  Note that
$$
J''(\theta)\varphi_j=\frac{\lambda_j-a_0}{\lambda_j}\varphi_j,\quad
B(\infty)\varphi_j=\frac{\lambda_j-a}{\lambda_j}\varphi_j\quad\hbox{for all}\;
j\in\N.
$$
It is clear that  $H^0_\theta\ne\{\theta\}$ (resp.
$H^0_\infty\ne\{\theta\}$) if and only if
$a_0\in\{\lambda_m\}^\infty_{m=1}$ (resp.
$a\in\{\lambda_m\}^\infty_{m=1}$). If $a_0=a=\lambda_m$ then
$H^0_\theta=H^0_\infty={\rm Span}(\{\varphi_j\,|\, m^-\le j\le
m^+\})$ and
\begin{eqnarray*}
H^-_\theta=H^-_\infty={\rm Span}(\{\varphi_j\,|\,
j<m^-\}),\qquad H^+_\theta=H^+_\infty={\rm
Span}(\{\varphi_j\,|\,
 j> m^+\}).
\end{eqnarray*}
If  $a_0=a>\lambda_1$ and $a_0=a\notin\{\lambda_i\}^\infty_{i=1}$,
then there exists $m\in\N$ such that
 $a_0=a\in (\lambda_{m^+}, \lambda_{m^++1})$ because $\lambda_n\to\infty$. In the case
 it holds that
\begin{eqnarray*}
H^-_\theta=H^-_\infty={\rm Span}(\{\varphi_j\,|\, j\le
m^+\})\quad\hbox{and}\quad H^+_\theta=H^+_\infty={\rm
Span}(\{\varphi_j\,|\,
 j> m^+\}).
\end{eqnarray*}
From these we obtain
\begin{eqnarray*}
&&\left.\begin{array}{ll}
&\nu_\theta=\nu_\infty=m^+-m^-+1,\\
&\mu_\theta=\mu_\infty=m^--1
\end{array}\right\}\quad\hbox{if}\quad a_0=a=\lambda_m,\\
&&\nu_\theta=\nu_\infty=0\quad\hbox{and}\quad
\mu_\theta=\mu_\infty=0
\quad\hbox{if}\quad a_0=a<\lambda_1,\\
&&\nu_\theta=\nu_\infty=0\quad\hbox{and}\quad
\mu_\theta=\mu_\infty=m^+ \quad\hbox{if}\quad a_0=a\in
(\lambda_{m^+}, \lambda_{m^++1}).
\end{eqnarray*}

By the splitting theorem for $C^2$ functionals on Hilbert spaces
(cf. \cite{Ch}, \cite{MaWi}) we get

\begin{proposition}\label{prop:4.3}
\begin{enumerate}
\item[{\rm (a)}]
 If $a_0<\lambda_1$, then
$C_k(J,\theta;\K)=\delta_{0k}\K$.

\item[{\rm (b)}] If $a_0\in (\lambda_{m^+}, \lambda_{m^++1})$ for some
$m\in\N$, then $C_k(J,\theta;\K)=\delta_{m^+k}\K$.

\item[{\rm (c)}] If $a_0=\lambda_m$, then $C_k(J,\theta;\K)=0$ for all
$k\notin [m^--1, m^+]$.
\end{enumerate}
 \end{proposition}

Under the assumptions of Proposition~\ref{prop:4.2}(d) the
functional $J$ satisfies the Palais-Smale condition and hence the
deformation condition $(D)_c$ at every $c\in\R$. Then the critical
group of $J$ at infinity, $C_\ast(J,\infty;\K)$, is well-defined.
The following is a generalization of Theorem~3.9 in \cite{BaLi}.

\begin{proposition}\label{prop:4.4}
Let the assumptions of Proposition~\ref{prop:4.2} be satisfied.
\begin{enumerate}
\item[{\rm (a)}] If $a<\lambda_1$, then
$C_k(J,\infty;\K)=\delta_{0k}\K$.

\item[{\rm (b)}] If $a\in (\lambda_{m^+}, \lambda_{m^++1})$ for
$m\in\N$, then $C_k(J,\infty;\K)=\delta_{m^+k}\K$.

\item[{\rm (c)}]  If $a=\lambda_m$, then $C_k(J,\infty;\K)=0\quad\forall
k\notin [m^--1, m^+]$ provided that
\begin{equation}\label{e:4.12}
(c(s,n,\Omega))^2 |\Omega|^{\frac{1}{\eta(s,n)}}\|\ell\|_{L^s}\hbar< \left\{\begin{array}{ll}
\frac{\lambda_2-\lambda_1}{\lambda_2},& m=1,\\
 \min\left\{\frac{\lambda_m-\lambda_{m^--1}}{\lambda_{m^--1}},
\frac{\lambda_{m^++1}-\lambda_2}{\lambda_{m^++1}}\right\}, & m>1
\end{array}\right.
\end{equation}
and
\begin{equation}\label{e:4.13}
(c(s,n,\Omega))^2|\Omega|^{\frac{1}{\eta(s,n)}}\|\ell\|_{L^s}\sup h<1.
\end{equation}
\end{enumerate}
 \end{proposition}

\begin{remark}\label{rm:4.5}
{\rm (a) The condition (\ref{e:4.13}) is only used in the first part of the proof of Lemma~\ref{lem:4.8}.
 It is easily seen that the condition  can be replaced by others. For example, when $n>2$
we may replace it by
\begin{equation}\label{e:4.14}
\|\ell\|_{L^{n/2}}\sup h<\frac{1}{c(n,\Omega)^2},
\end{equation}
where $c(n,\Omega)$ is the best constant such that $\|v\|_{L^\frac{2n}{n-2}}\le c(n,\Omega)\|v\|_H\;\forall v\in H$.
 Using (\ref{e:4.2}) and the embedding $L^\frac{2n}{n-2}(\Omega)\hookrightarrow L^{\xi(s,n)}(\Omega)$ it is not hard
 to prove that
 $$
 (c(s,n,\Omega))^2 |\Omega|^{\frac{1}{\eta(s,n)}}\le (c(n,\Omega))^2 |\Omega|^{\kappa(s,n)},
 $$
where $\kappa(s,n)=\frac{n}{n-2}-\frac{(n-2)s^2}{n(s-1)^2}+ \frac{2}{s}\frac{s^2-s-n}{2sn-2s-n}\to
\frac{1}{n-1}+\frac{2}{n}+\frac{2}{n-2}$ as $s\to\infty$.

\noindent{(b)}  When $\hbar=0$ the condition (\ref{e:4.12}) is naturally
satisfied because the left side of the inequality is always
positive. If $\hbar>0$ the upper bound of $\hbar$ given by (\ref{e:4.12}) might not be the
biggest one. In fact,  if $n>2$, by the proof of (\ref{e:4.8})
 $\xi(s,n)$ can be chosen  as any number in $(\frac{2s}{s-1}, \frac{2n}{n-2})$
 and then $\eta(s,n)$ is determined by  the equality $1/s+ 1/\eta(s,n)+ 2/\xi(s,n)=1$.
Hence the number $(c(s,n,\Omega))^2 |\Omega|^{\frac{1}{\eta(s,n)}}$ in (\ref{e:4.12}) and (\ref{e:4.13})
can be replaced by
$$
\inf\left\{C_\lambda^2 |\Omega|^{\frac{1}{\mu}}:\;{2s}/(s-1)<\lambda<2n/(n-2),\; 1/s+
2/\lambda+ 1/\mu=1\right\}
$$
where  $C_\lambda>0$ is the best constant for the embedding $H=H_0^1(\Omega)\hookrightarrow L^{\lambda}(\Omega)$.
Our choice of $\xi(s,n)$ (so $\eta(s,n)$) in (\ref{e:4.2}) is only for convenience. It has showed that
our arguments and results can be given in a quantitative way. }
\end{remark}

Before proving Proposition~\ref{prop:4.4} we point out
that the following  generalization for Theorem~\ref{th:4.1}
can be obtained by using Propositions~\ref{prop:4.3},~\ref{prop:4.4}
and  \cite[Propositions~2.3, 3.6]{BaLi} and
repeating the proof of \cite[Proposition 5.2]{BaLi}.

\begin{theorem}\label{th:4.6}
Suppose that the assumptions (${\rm p}^\ast$) and
(${\rm q}^\ast_1$)--(${\rm q}^\ast_3$)  are satisfied.
\begin{enumerate}
\item[{\rm (a)}] If $a_0$ is not an eigenvalue of $-\triangle$ then
(\ref{e:4.1}) has at least one nontrivial solution provided that for
some $m\in\N$, (\ref{e:4.12})--(\ref{e:4.13}) hold and either
$a_0<\lambda_m<a$ or $a<\lambda_m<a_0$.

\item[{\rm (b)}] If $a_0=\lambda_m$ is an eigenvalue but (\ref{e:4.12})--(\ref{e:4.13}) and (${\rm q}_4^+$) hold
in addition, then (\ref{e:4.1}) has at least one nontrivial solution
provided that either $a<a_0$ or   $a_0<\lambda_k<a$ for some $k>m$
and (\ref{e:4.12})--(\ref{e:4.13}) hold with $m=k$.

\item[{\rm (c)}] If $a_0=\lambda_m$ is an eigenvalue but (\ref{e:4.12})--(\ref{e:4.13}) and (${\rm q}_4^-$) hold
in addition, then (\ref{e:4.1}) has at least one nontrivial solution
provided that either $a_0<a$ or $a<\lambda_k<a_0$ for some $k<m$ and
(\ref{e:4.12}) holds with $m=k$.
\end{enumerate}
\end{theorem}


Indeed, the condition (${\rm q}_4^+$) (resp. (${\rm q}_4^+$)) is used to assure
that the local linking condition in  Propositions~2.3 of \cite{BaLi}
holds with $X^-=H^0_\theta\oplus H^-_\theta$ and $X^+=H^+_\theta$
(resp. $X^-= H^-_\theta$ and $X^+=H^0_\theta\oplus H^+_\theta$)
because $J(u)=\frac{1}{2}(B(\theta)u,u)_H-\int_\Omega Q_0(x,u(x))dx$
for all $u\in H$. They imply $C_{\mu_\theta+\nu_\theta}(J,\theta)\ne
0$ and $C_{\mu_\theta}(J,\theta)\ne 0$, respectively.

There exists a further possible improvement, that is, the limit
$a_0=\lim_{t\to 0}\frac{p(x,t)}{t}$ in (${\rm p}^\ast$) is not required to
be constant. For example, for Theorem~\ref{th:4.6}(a), we may assume that
  $a_0(x)=\lim_{t\to 0}\frac{p(x,t)}{t}$
exists for a.e. $x\in\Omega$. Then the second condition in
(${\rm p}^\ast$) and (${\rm q}_3^\ast$) imply that  $a_0(x)=a+ q'_t(x,0)$ for
a.e. $x\in\Omega$. Suppose that $1$ is not an eigenvalue of the
equation $-\triangle u=\lambda a_0u$ in $\Omega$ with $0$ boundary
conditions. Then $\theta$ is a nondegenerate critical point of $J$
with finite Morse index $\mu_\theta$, and hence
$C_k(J,\theta)=\delta_{k\mu_\theta}\K$. If $\theta$ is a unique
critical point of $J$ then $C_k(J,\infty)=\delta_{k\mu_\theta}\K$ by
Proposition~3.6 of \cite{BaLi}. Proposition~\ref{prop:4.4} shall
lead to a contradiction under the suitable condition on $a$. The
corresponding generalizations of Theorem~\ref{th:4.6}(b)-(c) can be
obtained similarly.

\noindent{\bf Proof of Proposition~\ref{prop:4.4}.}\quad {\it Step
1}. Carefully checking the proof of Lemma 4.2 in \cite{BaLi} one
easily sees that (\ref{e:4.9}) and (\ref{e:4.11}) imply the corresponding result:
{For sufficiently large $R>0$ and $b\ll 0$ the pair
$$
\bigl(B_{H^0_\infty}(\theta, R+1)\oplus H^\pm_\infty, J^b\cap
(B_{H^0_\infty}(\theta, R+1)\oplus H^\pm_\infty)\bigr)
$$
is homotopy to the pair
$$
\bigl(B_{H^0_\infty}(\theta, R+1)\oplus \bar B_{H^-_\infty}(\theta,
1), B_{H^0_\infty}(\theta, R+1)\oplus \partial\bar
B_{H^-_\infty}(\theta, 1)\bigr).
$$
The homotopy equivalence leaves the $H^0_\infty$-component fixed. In particular,
the pair $(H, J^b)$ is homotopy to the pair
$\bigl(\bar B_{H^-_\infty}(\theta, 1),  \partial\bar B_{H^-_\infty}(\theta, 1)\bigr)$
provided that $\nu_\infty=0$ and $\mu_\infty<\infty$.}
The final claim immediately leads to (a) and (b).

{\it Step 2}. We begin to prove (c). In this case obverse that
$$
(B(\infty)|_{H^\pm_\infty})^{-1}\Bigl(\sum_{\lambda_j\ne\lambda_m}x_j\varphi_j
\Bigr)=\sum_{\lambda_j\ne\lambda_m}\frac{\lambda_j}{\lambda_j-\lambda_m}x_j\varphi_j.
$$
For $X=H$, by the definitions of $C_1^\infty$ and $C_2^\infty$ above
(\ref{e:1.2}) we have
$$
C^\infty_1=\|(B(\infty)|_{H^\pm_\infty})^{-1}\|_{L(H^\pm_\infty,
H^\pm_\infty)}\quad\hbox{and}\quad C^\infty_2=\|I-P^0_\infty\|_{L(H,
H^\pm_\infty)}.
$$
From these ones easily derive

\begin{lemma}\label{lem:4.7}
 $C^\infty_2=1$ (because
$I-P^0_\infty=P^\pm_\infty\ne I$). If $a=\lambda_1$ then
$$
C_1^\infty=\|(B(\infty)|_{H^\pm_\infty})^{-1}\|_{{\cal
L}(H^\pm_\infty)}=\frac{\lambda_2}{\lambda_2-\lambda_1},
$$
and if $a=\lambda_m$ with $m\ge 2$ then
$$
C_1^\infty=\|(B(\infty)|_{H^\pm_\infty})^{-1}\|_{{\cal
L}(H^\pm_\infty)}=\max\left\{\frac{\lambda_{m^--1}}{\lambda_m-\lambda_{m^--1}},
\frac{\lambda_{m^++1}}{\lambda_{m^++1}-\lambda_2} \right\}.
$$
\end{lemma}

\begin{lemma}\label{lem:4.8}
 For $a=\lambda_m$, if
 { either} $\hbar=0$ { or} $\hbar>0$
and (\ref{e:4.12})--(\ref{e:4.13}) are satisfied, then taking
$\rho_{\nabla J}$ as any positive number $\rho$ there exist $R_1>0$ such
that the conditions of Corollary~\ref{cor:1.6}
   are satisfied.
\end{lemma}

We postpone the proof of it.
Under the assumptions of this lemma and Proposition~\ref{prop:4.2},
by Corollary~\ref{cor:1.6}  there exist a
positive number $R$, a $C^1$ map $h^\infty: B_{H^0_\infty}(\infty,
R)\to \bar B_{H^\pm_\infty}(\theta,\rho_{\nabla J})$ (satisfying
$(I-P^0_\infty)A(z+ h^\infty(z))=0$ for all $z\in \bar
B_{H^0_\infty}(\infty, R)$), and a homeomorphism $\Phi:
B_{H^0_\infty}(\infty, R)\oplus H^\pm_\infty\to
B_{H^0_\infty}(\infty, R)\oplus H^\pm_\infty$  such that
$$
J\circ\Phi(z+ u^++ u^-)=\|u^+\|^2-\|u^-\|^2+ J(z+ h^\infty(z))
$$
for all $(z, u^+ + u^-)\in B_{H^0_\infty}(\infty, R)\times
H^\pm_\infty$. Using this we may repeat the arguments on pages 432--433
of \cite{BaLi} to derive that Lemma~4.3 of \cite{BaLi} holds for
$J$: There exist a sufficiently large $R>0$, $b\ll 0$ and a
continuous map $\gamma:B_{H^0_\infty}(\infty, R)\to [0, 1]$ with
$\gamma(C)>0$ for $C:=B_{H^0_\infty}(\theta, R+1)\cap
B_{H^0_\infty}(\infty, R)$ such that the pair
$$
(B_{H^0_\infty}(\infty, R)\times H^\pm_\infty,
J^b\cap(B_{H^0_\infty}(\infty, R)\times H^\pm_\infty))
$$
is homotopy equivalent to the pair $(B_{H^0_\infty}(\infty, R)\times
H^-_\infty, \Gamma)$, where
\begin{eqnarray*}
&&\Gamma=\{(z, u)\in B_{H^0_\infty}(\infty, R)\times
H^-_\infty:\;\|u\|\ge\gamma(z)\}\quad\hbox{and}\\
&& \gamma(z)=\left\{
\begin{array}{ll}
0 &\hbox{if}\; J(z+ h^\infty(z))\le a,\\
1 &\hbox{if}\; J(z+ h^\infty(z))\ge a+ 1,\\
J(z+ h^\infty(z))-b &\hbox{elsewhere}.
\end{array}\right.
\end{eqnarray*}
Moreover, the  homotopy equivalence leaves the
$H^0_\infty$-component fixed.

 Combing this with Step 1  and repeating the proof of Theorem~3.9 in \cite{BaLi} we
get the claim in Proposition~\ref{prop:4.4}(c), i.e.,
$C_k(J,\infty;\K)\cong H_k(H,J^b;\K)=0$ for all $k\notin [m^--1, m^+]$
because $[\mu_\infty, \mu_\infty+\nu_\infty]=[m^--1, m^+]$ by the list
above Proposition~\ref{prop:4.3}.
 \hfill$\Box$\vspace{2mm}

In order to prove Lemma~\ref{lem:4.8} we need

\begin{lemma}[\hbox{\cite[Lemma~3.2]{BaBeFo}}]\label{lem:4.9}
Let $V$ be a finite dimensional subspace of $C(\overline{\Omega})$
such that every $v\in V\setminus\{0\}$ is different from zero a.e.
in $\Omega$. Let $h\in L^\infty(\R)$ such that $h(t)\to 0$ as
$|t|\to\infty$. Moreover, consider a compact subset $K$ of
$L^p(\Omega)$ ($p\ge 1$). Then
$$
\lim_{|t|\to\infty}\int_\Omega|h(tv(x)+ u(x))|dx=0
$$
uniformly as $u\in K$ and $v\in S$, where $S=\{v\in
V\,|\,\|v\|_{C^0}=1\}$.
\end{lemma}

Since any two norms on a finite dimensional linear space are
equivalent, and any bounded set in $H=H^1_0(\Omega)$ is compact
$L^1(\Omega)$, using this lemma we easily prove

\begin{claim}\label{cl:4.10}
For given numbers $\rho>0$ and $\varepsilon>0$ there exists a
$R_0>0$ such that
$$
\|h(z+ u)-\hbar\|_{L^{\eta(s,n)}}+ \hbar|\Omega|^{\frac{1}{\eta(s,n)}}
<\varepsilon+ \hbar |\Omega|^{\frac{1}{\eta(s,n)}}
$$
for any $u\in \bar B_{H^\pm_\infty}(\theta,\rho)$ and $z\in
H^0_\infty$ with $\|z\|_H\ge R_0$. Here $\eta(s,n)$ is given by (\ref{e:4.2}).
\end{claim}

\noindent{\bf Proof of Lemma~\ref{lem:4.8}}.\quad It suffices to check that the
conditions (c)--(d) of Corollary~\ref{cor:1.6} can be satisfied. Firstly,
we claim that the condition (c)
holds. In fact, since $Q(\infty)v=-aKv$ by (\ref{e:4.7}),
from  (\ref{e:4.5}) and (\ref{e:4.4}) we deduce that
\begin{eqnarray*}
(B(u)v-Q(\infty)v,v)_H&=&(v,v)_H-\int_{\Omega}q'_t(x,u(x))(v(x))^2dx\\
&\ge &(v,v)_H-(c(s,n,\Omega))^2|\Omega|^{\frac{1}{\eta(s,n)}}\|\ell\|_{L^s}
(\sup h)\cdot\|v\|^2_H\\
&\ge &\bigl(1-(c(s,n,\Omega))^2|\Omega|^{\frac{1}{\eta(s,n)}}\|\ell\|_{L^s}
\sup h\bigr)\|v\|_H^2.
\end{eqnarray*}
This and (\ref{e:4.13}) lead to the desired conclusion.

Next, we prove the condition (d) holds true.
 By (\ref{e:4.11}), $\|\nabla J(z)\|_H= o(\|z\|_H)$ as $z\in
H^0_\infty$ and $\|z\|_H\to\infty$. Hence
$$
M(A)=M(\nabla J)=\lim_{R\to\infty}\sup\{\|(I-P^0_\infty)\nabla
J(z)\|_H:\;z\in H^0_\infty, \|z\|_H\ge R\}=0.
$$
By Lemma~\ref{lem:4.7} and (\ref{e:4.12}) we may take a small
$\varepsilon>0$ such that
$$
(c(s,n,\Omega))^2\|\ell\|_{s}\left(\varepsilon+\hbar
|\Omega|^{\frac{1}{\eta(s,n)}}\right)<1/C_1^\infty.
$$
For this $\varepsilon>0$ and a given numbers $\rho>0$, by
Claim~\ref{cl:4.10}  there exist $R_0>0$ such that
$$
\|h(z+ u)-\hbar\|_{L^{\eta(s,n)}}+ \hbar
|\Omega|^{\frac{1}{\eta(s,n)}}<\varepsilon+ \hbar |\Omega|^{\frac{1}{\eta(s,n)}}
$$
for any $u\in \bar B_{H^\pm_\infty}(\theta,\rho)$ and $z\in
H^0_\infty$ with $\|z\|_H\ge R_0$. These and (\ref{e:4.10}) lead to
\begin{eqnarray*}
&&\|(I-P^0_\infty)[B(z+u)-B(\infty)]|_{H^\pm_\infty}\|_{\mathcal{L}(H^\pm_\infty)}\\
&\le&\|B(z+u)-B(\infty)\|_{L(H)}=\|g''(z+u)\|_{\mathcal{L}(H)}\\
&\le& (c(s,n,\Omega))^2\|\ell\|_{s}\left(\varepsilon+  \hbar
|\Omega|^{\frac{1}{\eta(s,n)}}\right)<\frac{1}{\kappa C_1^\infty}
\end{eqnarray*}
for any $u\in \bar B_{H^\pm_\infty}(\theta,\rho)$ and $z\in
H^0_\infty$ with $\|z\|_H\ge R_0$, and for some $\kappa>1$. \hfill$\Box$\vspace{2mm}

\noindent{\bf Proof of Proposition~\ref{prop:4.2}}.\quad {\rm (a)}
Since the functional $H\ni u\mapsto (B(\infty)u,u)_H$ is smooth, we
only need to prove that the functional $g$ in (\ref{e:4.8}) is
$C^1$. By (${\rm q}_1^\ast$),
\begin{eqnarray}\label{e:4.16}
|q(x,t_1+t_2)|&\le& E(x)+ c_1
|t_1+t_2|^r\le E(x)+ c_1(1+|t_1|+ |t_2|)^r\nonumber\\
&\le&  E(x)+ c_1+ c_1|t_1|+ c_1|t_2|
\end{eqnarray}
for $\;a.e.\;x\in\Omega$ and any $t_1,t_2\in\R$.
Obverse that $Q$ is also a Carth\'eodory function and that
\begin{eqnarray*}
&&\left|Q(x,u(x)+v(x))-Q(x,u(x))\right|\le\sup_{\tau\in[0,1]}|q(x, u(x)+\tau v(x))|\cdot|v(x)|\\
&&\le \bigl(E(x)+ c_1+ c_1|u(x)|+ c_1|v(x)|\bigr)\cdot|v(x)|
\end{eqnarray*}
for any $u,v\in H$. So $g$ (and hence $J$) is continuous because
\begin{eqnarray*}
|g(u+v)-g(u)|&\le& \int_\Omega \bigl(E(x)+ c_1+ c_1|u(x)|+ c_1|v(x)|\bigr)\cdot|v(x)|dx\\
&\le& \|E+c_1+ c_1|u|+c_1|v|\|_{L^2}\|v\|_{L^2}\\
&\le& (\|E\|_{L^2}+ c_1|\Omega|^{1/2}+ c_1\|u\|_H+ c_1\|v\|_H)\|v\|_H.
\end{eqnarray*}

In order to prove that $g$ is $C^1$, by the standard result in
functional analysis we only need to prove that $g$ has a bounded
linear G\^ateaux derivative $Dg(u)$ at every point $u\in H$ and that
$H\ni u\mapsto Dg(u)\in H^\ast$ is continuous.

For $u,v\in H_0^1(\Omega)$,  $\tau\in (-1,1)\setminus\{0\}$ and
almost every $x\in\Omega$, as above we get
\begin{eqnarray*}
\left|\frac{Q(x,u(x)+\tau
v(x))-Q(x,u(x))}{\tau}\right|&\le&\sup_{0<\theta<1}|q(x, u(x)+\theta\tau v(x))v(x)|\cdot|v(x)|\\
&\le& \bigl(E(x)+ c_1+ c_1|u(x)|+ c_1|v(x)|\bigr)\cdot |v(x)|
\end{eqnarray*}
by (\ref{e:4.16}). From this and the Lebesgue dominated convergence
theorem we derive
$$
Dg(u)[v]=\frac{d}{d\tau}\Bigl|_{\tau=0}g(u+\tau v)=-\int_\Omega
q(x,u(x))\cdot v(x)dx.
$$
That is, $g$ is G\^ateaux differentiable. Clearly, $Dg(u)\in
H^\ast$ since we have as above
$$
|Dg(u)[v]|=|\int_\Omega
q(x,u(x))\cdot v(x)dx|\le (\|E\|_{L^2}+ c_1|\Omega|^{1/2}+ c_1\|u\|_H)\|v\|_H.
$$
Moreover, for $u_1, u_2, v\in H_0^1(\Omega)$,
by (${\rm q}^\ast_3$) the functions
$\R\ni t\mapsto q(x,t)$ and $\R\ni t\mapsto q'_t(x,t)$ are
continuous for almost every $x\in\Omega$. The calculus fundamental theorem and (\ref{e:4.4}) lead to
\begin{eqnarray*}
&&\left|\int_{\Omega}[q(x,u_2(x))-q(x,u_1(x))]\cdot v(x)dx\right|\\
&=&\left|\int^1_0\left[\int_{\Omega}q'_t(x, u_1(x)+\tau(u_2(x)-u_1(x)))(u_2(x)-u_1(x))\cdot v(x) dx\right]d\tau\right|\\
&\le&\int^1_0\left[\int_{\Omega}|q'_t(x, u_1(x)+\tau(u_2(x)-u_1(x)))|\cdot |u_2(x)-u_1(x)|\cdot |v(x)| dx\right]d\tau\\
&\le&(c(n,s, \Omega))^2\|\ell\|_{L^s}|\Omega|^{\frac{1}{\eta(s,n)}}(\sup h)\cdot\|u_2-u_1\|_{H}\|v\|_{H}
\end{eqnarray*}
and hence
$\|Dg(u_1)-Dg(u_2)\|_{H^\ast}\le (c(n,s, \Omega))^2\|\ell\|_{L^s}|\Omega|^{\frac{1}{\eta(s,n)}}(\sup h)\cdot\|u_2-u_1\|_{H}$.
It follows that $J$ is $C^{1,1}$.

The expression of $\nabla J$ is clear. It remains to prove  (\ref{e:4.9}). Since
  (${\rm q}^\ast_1$) implies
$|Q(x,t)|\le |t|E(x)+
c_1|t|^{r+1}$ for all $(x,t)\in\Omega\times\R$, and $H\hookrightarrow L^{r+1}$ we have
\begin{eqnarray*}
|g(u)|&\le &\int_\Omega|Q(x,u(x))|dx\le
\int_\Omega(E(x)|u(x)|+ c_1|u(x)|^{r+1})dx\\
&\le& \|E\|_{L^2}\|u\|_{L^2}+ c_1\|u\|^{r+1}_{L^{r+1}}
\le \|E\|_{L^2}\|u\|_{H}+
C_rc_1\|u\|_{H}^{r+1}\quad\forall u\in H
\end{eqnarray*}
for some constant $C_r>0$. (\ref{e:4.9}) follows immediately.

 {\rm (b)} For $C^2$-smoothness of $J$  we shall
  prove that $\nabla g$ is $C^1$ and $D(\nabla g)(u)=B(u)-B(\infty)$
  for any $u\in H$. For $v,w\in H$ and $\tau\in
(-1,1)\setminus\{0\}$,  (\ref{e:4.5}) and (\ref{e:4.7}) give
\begin{eqnarray*}
&&\left|([\nabla g(u+\tau v))-\nabla g(u))]/\tau,w)_H-([B(u)-B(\infty)]v,w)_H\right|\\
&=&\left|(Dg(u+\tau v))[w]- Dg(u)[w])/\tau-([B(u)-B(\infty)]v,w)_H\right|\\
&=&\left|\int_{\Omega}\left[\frac{q(x, u(x)+\tau v(x))-q(x,
u(x))}{\tau}-q'_t(x,u(x))v(x)\right]\cdot w(x)dx\right|\\
&\le&\int_{\Omega}\left|\frac{q(x, u(x)+\tau v(x))-q(x,
u(x))}{\tau}-q'_t(x,u(x))v(x)\right|\cdot |w(x)|dx\\
&\le&\left(\int_{\Omega}\left|\frac{q(x, u(x)+\tau v(x))-q(x,
u(x))}{\tau}-q'_t(x,u(x))v(x)\right|^adx\right)^{1/a}\cdot \|w\|_{L^b}
\end{eqnarray*}
where $a=\frac{2n}{n+2}$ and $b=\frac{2n}{n-2}$ if $n>2$, and $a\in (1,s)$ and $b=\frac{a}{a-1}$ if $n=2$.
Note that $H\hookrightarrow L^b$. So for some constant $C>0$ it holds that
\begin{eqnarray}\label{e:4.17}
&&\left\|[\nabla g(u+\tau v))-\nabla g(u))]/\tau- [B(u)-B(\infty)]v\right\|_H\\
&\le& C\left(\int_{\Omega}\left|\frac{q(x, u(x)+\tau v(x))-q(x,
u(x))}{\tau}-q'_t(x,u(x))v(x)\right|^adx\right)^{1/a}.\nonumber
\end{eqnarray}
 By (${\rm q}^\ast_3$), for a.e. $x\in\Omega$ using the intermediate value theorem we obtain
$$
\left|\frac{q(x, u(x)+\tau v(x))-q(x,
u(x))}{\tau}-q'_t(x,u(x))v(x)\right|\le 2\ell(x)\sup h\cdot|v(x)|.
$$
Since $\ell\in L^s(\Omega)$, $s>\frac{n}{2}$, and $d:=\frac{s}{a}>1$ for $n=2$,  the H\"older's inequality gives
\begin{eqnarray*}
&&\int_\Omega(\ell(x))^a|v(x)|^adx\le \left(\int_\Omega(\ell(x))^{s}\right)^{a/s}\left(\int_\Omega|v(x)|^{\frac{d}{d-1}}\right)^{\frac{d-1}{d}}
<\infty\quad\hbox{for}\;n=2,\\
&&\int_\Omega(\ell(x))^a|v(x)|^adx\le \left(\int_\Omega(\ell(x))^{\frac{n}{2}}\right)^{\frac{4}{n+2}}
\left(\int_\Omega|v(x)|^{\frac{2n}{n-2}}\right)^{\frac{n-2}{n+2}}<\infty\quad\hbox{for}\;n>2.
\end{eqnarray*}
  Hence using  the Lebesgue dominate convergence theorem it follows from (\ref{e:4.17}) that
\begin{eqnarray*}
&&\lim_{\tau\to 0}\left\|[\nabla g(u+\tau v))-\nabla g(u))]/\tau- [B(u)-B(\infty)]v\right\|_H\\
&\le& C\left(\lim_{\tau\to 0}\int_{\Omega}\left|\left[\frac{q(x, u(x)+\tau
v(x))-q(x, u(x))}{\tau}-q'_t(x,u(x))v(x)\right]\right|^adx\right)^{1/a}=0.
\end{eqnarray*}
Hence $\nabla g$ is G\^ateaux differentiable, and $\nabla J$ has the
G\^ateaux derivative $B(u)$ at any $u\in H$.

  For any $u_1, u_2\in H$, by (\ref{e:4.5}) and the generalized H\"older inequality we have
\begin{eqnarray*}
&&|(B(u_1)v-B(u_2)v,w)_H|=\left|\int_{\Omega}[q'_t(x,u_2(x))-q'_t(x,u_1(x))]v(x)w(x)dx\right|\\
&\le&\left(\int_{\Omega}|q'_t(x,u_2(x))-q'_t(x,u_1(x))|^s\right)^{1/s}\|v\|_{L^b}\|w\|_{L^b},
\end{eqnarray*}
  where $b=\frac{2s}{s-1}$. Since $s>\frac{n}{2}\ge 1$  we have $b\in (2, 2n/(n-2)]$ for $n>2$
  and $b\in (2,\infty)$ for $n=2$.
These imply that $\|v\|_{L^b}\le\sqrt{c_1(s,n,\Omega)}\|v\|_H$ and
$\|w\|_{L^b}\le\sqrt{c_1(s,n,\Omega)}\|w\|_H$ for some constant $c_1(s,n,\Omega)>0$.
It follows that
$$
\|B(u_1)-B(u_2)\|_{\mathcal{L}(H)}\le c_1(s,n,\Omega)\left(\int_{\Omega}|q'_t(x,u_2(x))-q'_t(x,u_1(x))|^s\right)^{1/s}.
$$
 Since $t\mapsto q'_t(x,t)$ is continuous,  and $|q'_t(x,
u(x))|\le\|h\|_{L^\infty}\ell(x)$ by (${\rm q}^\ast_3$), the corresponding Nemytskii operator
$L^p(\Omega)\ni u\mapsto \mathfrak{q}_1(u)\in L^s(\Omega)$ given by
$\mathfrak{q}_1(u)(x)=q'_t(x,u(x))$ is continuous for any $p\in [1,\infty)$.
Taking $p=1$ and using $H\hookrightarrow L^1(\Omega)$ we deduce that
$\|B(u_1)-B(u_2)\|_{\mathcal{L}(H)}\to 0$ as $\|u_1-u_2\|_H\to 0$.
Hence   $\nabla J$ has the Fr\'echlet derivative
$B(u)$ at $u\in H$, and therefore that $J$ is $C^2$.

As to (\ref{e:4.10}), since $g''(u)=B(u)-B(\infty)$,
for $z\in H^0_\infty$ and $u_1\in H^\pm_\infty, u_2\in H$, it follows from (\ref{e:4.3}) that
\begin{eqnarray*}
&&\|g''(z+u_1)u_2\|_{H}=\sup_{\|w\|_H\le 1}|([B(z+u_1)-B(\infty)]u_2,w)_H|\\
&\le&\sup_{\|w\|_H\le 1}\int_\Omega|q'_t(x, z(x)+u_1(x))u_2(x)w(x)|dx\\
&\le& (c(s,n,\Omega))^2\|\ell\|_{L^s}
\|u_2\|_H\left(\int_\Omega |h(z(x)+u_1(x))|^{\eta(s,n)}dx\right)^{\frac{1}{\eta(s,n)}}\\
&\le& (c(s,n,\Omega))^2\|\ell\|_{L^s}
\|u_2\|_H\Bigl[\hbar |\Omega|^{\frac{1}{\eta(s,n)}}+\|h\circ(z+u_1)-\hbar\|_{L^{\eta(s,n)}}\Bigr]
\end{eqnarray*}
and hence the expected conclusion.

{\rm (c)} Since $|q(x,t)|\le E(x)+ c_1|t|^r\le E(x)+ 1/p+ c_1^q|t|^{rq}/q$
for any $q\in [1,\infty)$ and $p$ with $1/p+ 1/q=1$, we can assume: $1>r\ge\frac{n-2}{2n}$
for $n>2$.  Note that
\begin{eqnarray*}
\|\nabla g(u)\|_H&=& \sup_{\|v\|_H\le 1}|(\nabla g(u),
v)_H|\le\sup_{\|v\|_H\le 1}\int_\Omega |q(x,u(x))|\cdot|v(x)|dx\\
&\le&\sup_{\|v\|_H\le 1}\left[\int_\Omega |E(x)|\cdot|v(x)|dx+ c_1\int_\Omega|u(x)|^r\cdot|v(x)|dx\right]\\
&\le&\sup_{\|v\|_H\le 1}\left[\|E\|_{L^s}\|v\|_{L^{s^\ast}}+ c_1\|u\|_{L^{rp^\ast}}^r\|v\|_{L^{p}}\right],
\end{eqnarray*}
where $s^\ast=\frac{s}{s-1}<\frac{2n}{n-2}$ because $s> n/2$ and $n\ge 2$; moreover
$p=\frac{2n}{n+2}$ (so $p^\ast=\frac{2n}{n-2}$) for $n>2$, and $p=\frac{1}{1-r}$ (so $p^\ast=1/r$) for
$n=2$. Now $1\le rp^\ast<\frac{2n}{n-2}$ for $n>2$, and $1=rp^\ast$ for $n=2$. The Sobolev embedding theorems
yield a constant $c_3(n,s,\Omega)>0$ such that
$$
\|E\|_{L^s}\|v\|_{L^{s^\ast}}+ c_1\|u\|_{L^{rp^\ast}}^r\|v\|_{L^{p}}\le
c_3(n,s,\Omega)( \|E\|_{L^s}\|v\|_{H}+ \|u\|_{H}^r\|v\|_{H}).
$$
It follows that $\|\nabla J(u)-B(\infty)u\|_H=\|\nabla g(u)\|_H\le c_3(n,s,\Omega)( \|E\|_{L^s}+ \|u\|_{H}^r)$.
This implies $\|\nabla J(u)-B(\infty)u\|_H=o(\|u\|_H)$ as $\|u\|_H\to\infty$.

 {\rm (d)} By (\ref{e:4.7}),  $B(\infty)=I-aK$. Let $H^0_\infty:={\rm Ker}(B(\infty))$.
Note that the positive (resp. negative) definite subspace of
$B(\infty)$, $H^+_\infty$ (resp. $H^-_\infty$), is spanned by the
eigenfunctions of $-\triangle$ which correspond to the eigenvalues
less than (resp. greater than) $a$. Since $H=H^0_\infty\oplus
H^+_\infty\oplus H^-_\infty$ we may write $u\in H$ as $u=u^0+ u^++
u^-$. Hence $(B(\infty)u, u)_H=(B(\infty)u^+, u^+)_H+(B(\infty)u^-,
u^-)_H$. It follows that
$$
\|u\|_\ast=\left(\|u^0\|^2_{L^2}+ (B(\infty)u^+, u^+)_H-(B(\infty)u^-,
u^-)_H\right)^{\frac{1}{2}}
$$
 defines an equivalent norm on $H$.
 Let $(u_k)$ be a Palais-Smale sequence for $J$ in $H$. That is, $J'(u_k)\to 0$ and $|J(u_k)|\le M$
 for some $M>0$ and all $k\in\N$.
 The former and the definition of the norm $\|\cdot\|_\ast$ imply
 $$
 (\|u_k^\pm\|_\ast)^2\le \left|\int_\Omega q(x,u_k)u_k^\pm
dx\right|+ \|u_k^\pm\|_\ast
 $$
 for $k$ large. Noting  $r\in (0,1)$, by the H\"older inequality and  ($q^\ast_1$) we deduce that
  $\|u_k\|^{2r}_{L^{2r}}=\int_\Omega|u_k|^{2r}dx\le(\int_\Omega dx)^{1-r}(\int_\Omega|u_k|^{2}dx)^r=
  |\Omega|^{1-r}\|u_k\|_{L^2}^{2r}$ and
 \begin{eqnarray*}
&&\left|\int_\Omega q(x,u_k)u_k^\pm dx\right|\le \left(\int_\Omega
|q(x,u_k)|^2dx\right)^{1/2}\cdot \|u_k^\pm\|_{L^2}\\
&\le& \|u_k^\pm\|_H (\|E\|_{L^2}+c_1\|u_k\|^r_{L^{2r}})
\le c\|u_k^\pm\|_{\ast} (\|E\|_{L^2}+c_1|\Omega|^{(1-r)/2}\|u_k\|^r_{L^{2}})
\end{eqnarray*}
for some constant $c>0$ since  the norms $\|\cdot\|_H$ and
  $\|\cdot\|_\ast$ on $H$ are equivalent. Hence
 \begin{eqnarray}\label{e:4.18}
 \|u_k^\pm\|_\ast\le 1+ c (\|E\|_{L^2}+c_1|\Omega|^{(1-r)/2}\|u_k\|^r_{L^{2}})\le 1+ c\|E\|_{L^2}+ c_r\|u_k\|^r_{\ast}
 \end{eqnarray}
 for some constant $c_r>0$. By (${\rm q}^\ast_2$) we may assume
\begin{eqnarray*}
\frac{1}{2}q(x,t)t-Q(x,t)\ge c_2|t|^\alpha-G(x)\quad\forall
(x,t)\in\Omega\times\R.
\end{eqnarray*}
(Another case can be proved in the same way). Then for $k$ large we
have
\begin{eqnarray}\label{e:4.19}
M+\|u_k\|_\ast&\ge& |J(u_k)-\frac{1}{2}J'(u_k)u_k|\nonumber\\
&=&\left|\int_\Omega\left(\frac{1}{2}q(x,u_k(x))u_k(x)-Q(x,u_k(x))\right)dx\right|\nonumber\\
&\ge&\left|\int_\Omega\left(\frac{1}{2}q(x,u_k(x))u_k(x)-Q(x,u_k(x)+ G(x))\right)dx\right|-\|G\|_{L^1}\nonumber\\
&\ge & c_2\int_\Omega|u_k(x)|^\alpha dx-\|G\|_{L^1}\ge c_2\|u_k\|^\alpha_{L^\alpha}-\|G\|_{L^1}.
\end{eqnarray}
Since (${\rm q}^\ast_1$) and (${\rm q}^\ast_2$) imply $\alpha\le r+1<2<
\frac{2n}{n-2}$, the Sobolev embedding theorem yields
\begin{eqnarray*}
\|u_k^0\|_{L^\alpha}&\le& \|u_k\|_{L^\alpha}+\|u_k^+\|_{L^\alpha}+\|u_k^-\|_{L^\alpha}\le \|u_k\|_{L^\alpha}+ c_4\|u_k^+\|_{\ast}+ c_4\|u_k^-\|_{\ast}\\
&\le&((\|G\|_{L^1}+ M+\|u_k\|_\ast)/c_2)^{1/\alpha}+ 2c_4+ 2c_4c\|E\|_{L^2}+ 2c_4c_r\|u_k\|^r_{\ast}
\end{eqnarray*}
for some constants $c_4>0$ and $c_5>0$, where the third inequality comes from (\ref{e:4.19}) and (\ref{e:4.18}).
Moreover any two norms on the finitely dimensional space $H^0_\infty$ are equivalent.
We have $\|u_k^0\|_\ast\le c_5\|u_k^0\|_{L^\alpha}\;\forall k\in\N$
 for some constant $c_5>0$ independent of $k\in\N$.
From these and (\ref{e:4.18}) we deduce that
\begin{eqnarray*}
\|u_k\|_{\ast}&\le& \|u_k^+\|_{\ast}+ \|u_k^-\|_{\ast}+\|u_k^0\|_{\ast}\le
2+ 2c\|E\|_{L^2}+ 2c_r\|u_k\|^r_{\ast}+\\
&+& c_5\bigl[((\|G\|_{L^1}+ M+\|u_k\|_\ast)/c_2)^{1/\alpha}+ 2c_4+ 2c_4c\|E\|_{L^2}+ 2c_4c_r\|u_k\|^r_{\ast}\bigr]
\end{eqnarray*}
for $k$ large, and thus  $(\|u_k\|_{\ast})$ must be bounded because $0<r<1$ and $0<1/\alpha<1$.

Passing to a subsequence if necessary we may assume that $u_k\rightharpoonup u$ in $H$ and $u_k\to u$ in $L^2(\Omega)$.
As in (\ref{e:4.16}) we have $|q(x,t)|\le E(x)+ c_1|t|^r\le E(x)+c_1+ c_1|t|$ for all $(x,t)\in\Omega\times\R$.
Hence the  Nemytskii operator $L^2(\Omega)\ni u\mapsto \mathfrak{q}(u)\in L^2(\Omega)$ given by
$\mathfrak{q}(u)(x)=q(x,u(x))$ is continuous. Note that $\nabla J(u_k)=u_k-a Ku_k-K(\mathfrak{q}(u_k))\to 0$
in $H$ and that $K:L^2(\Omega)\to H=H^1_0(\Omega)$ is compact.
We obtain that $u_k\to a Ku-K(\mathfrak{q}(u))$ in $H$. The Palais-Smale condition holds true for $J$.
 \hfill $\Box$\vspace{2mm}

\bigskip
\footnotesize \noindent\textit{Acknowledgments.}
 Partially supported
by the NNSF   10971014 and 11271044 of China,  PCSIRT, RFDPHEC (No. 200800270003)
 and the Fundamental Research Funds for the Central Universities (No. 2012CXQT09).


\begin{thebibliography}{SK}



\normalsize \baselineskip=17pt



\bibitem{BaBeFo} P. Bartolo, V. Benci and D. Fortunato, {\it Abstract critical point theorems and
applications to some nonlinear problems with ``strong" resonance at infinity},  Nonlinear Anal.
  {\bf 7}(1983), 981--1012.


\bibitem{BaLi} T. Bartsch, S.-J. Li, {\it Critical point theory for asymptotically quadratic functionals and
applications to problems with resonance},   Nonlinear Anal.
  {\bf 28}(1997), no. 3, 419 --441.

\bibitem{Bre} H. Brezis,  {\it Functional Analysis, Sobolev Spaces and
Partial Differential Equations}, Springer 2011.


\bibitem{Ch} K. C. Chang, {\it Infinite Dimensional Morse Theory and Multiple Solution Problem.}
Birkh\"{a}user, 1993.


\bibitem{ChenLi1} S. W. Chen and S.-J. Li, {\it Splitting lemma in the infinity and a strong resonant problem with
periodic nonlinearity}, Nonlinear Anal. {\bf 65}(2006),
567--582.

\bibitem{ChenLi2} S. W. Chen and S.-J. Li,  {\it Splitting lemma at infinity and a strongly
resonant problem with periodic nonlinearity},  Calc. Var. {\bf
27}(2006), no. 1, 105--123.

\bibitem{Con} J. B. Conway, {\it A Course in Functional Analysis},
Springer, New York, 1990.


\bibitem{DrMi} P. Drabek and J. Milota, {\it Methods of Nonlinear Analysis. Applications to Differential Equations,}
Birkh\"user Advanced Texts: Basler Lehrb\"ucher. Birkh\"user Verlag, Basel, 2007.

\bibitem{DHK} D. M. Duc, T. V. Hung, N. T. Khai, {\it Morse-Palais lemma
for nonsmooth functionals on normed spaces},  Proc. Amer. Math. Soc.
{\bf 135}(2007), no. 3., 921--927.

\bibitem{HiLiWa} N. Hirano, S.-J. Li, Z.Q.Wang, {\it Morse theory without {\rm (PS)} condition
at isolated values and strong resonance problem},  Calc. Var.
Partial Differential Equations, {\bf 10}(2000), 223 -- 247.



\bibitem{JM} M. Jiang, {\it A generalization of Morse lemma and its
applications},  Nonlinear Anal. {\bf 36}(1999), 943--960.

\bibitem{Li} S.-J. Li, {\it A new Morse theory and strong resonance
problems},  Topol. Methods Nonlinear Anal. {\bf
21}(2003), 81--100.

\bibitem{Lu2} G. Lu, {\it The splitting lemmas for nonsmooth functionals
on Hilbert spaces}, arXiv:1102.2062v1.

\bibitem{Lu3} G. Lu, {\it The splitting lemmas for nonsmooth functionals
on Hilbert spaces} I.  Discrete Contin. Dyn. Syst.  {\bf 33}(2013), no.7, 2939--2990.
 arXiv:1211.2127.

\bibitem{Lu4} G. Lu, {\it Some critical point
theorems and applications}, arXiv:1102.2136v1.

\bibitem{MaWi} Jean Mawhin and Michel Willem, {\it Critical Point Theory and Hamiltonian
Systems}, Applied Mathematical Sciences Vol.74., Springer-Verlag,
1989.

\bibitem{Skr} I. V. Skrypnik, {\it Nonlinear Elliptic Equations of a Higher
Order} [in Russian], Naukova Dumka, Kiev (1973).


\bibitem{Va1} S. A. Vakhrameev, {\it Critical point theory for smooth functions on
Hilbert manifolds with singularities and its application to some
optimal control problems},  J. Sov. Math. {\bf 67}(1993), No.
1, 2713--2811.


\end{thebibliography}
\end{document}